\documentclass[a4paper,10pt]{amsart}
\usepackage{txfonts,amsfonts}
\usepackage[arrow,matrix]{xy}
\usepackage{amsmath,amssymb,amscd,bbm,amsthm,mathrsfs,dsfont,bm}
\usepackage{extarrows}
\usepackage{latexsym}
\usepackage{graphicx}
\usepackage{color}
\usepackage{xy}
\usepackage{leftidx}
\input xy
\input xypic
\xyoption{all}

\numberwithin{equation}{section}

\theoremstyle{plain}
\newtheorem{theorem}{Theorem}[section]
\newtheorem{corollary}[theorem]{Corollary}
\newtheorem{lemma}[theorem]{Lemma}

\newtheorem{proposition}[theorem]{Proposition}

\theoremstyle{definition}
\newtheorem{definition}[theorem]{Definition}

\newtheorem{example}[theorem]{Example}

\theoremstyle{remark}
\newtheorem{remark}[theorem]{Remark}

\DeclareMathOperator{\lw}{\rm LW}

\DeclareMathOperator{\uHom}{\underline{Hom}}
\DeclareMathOperator{\uExt}{\underline{Ext}}

\DeclareMathOperator{\gkdim}{\rm GKdim }
\DeclareMathOperator{\id}{\rm id }
\DeclareMathOperator{\sh}{\rm Sh }

\DeclareMathOperator{\lex}{\rm lex }
\DeclareMathOperator{\glex}{\rm glex }

\DeclareMathOperator{\gr}{\rm gr }
\DeclareMathOperator{\ad}{\rm ad }

\newcommand{\B}{\mathcal{B}}
\newcommand{\C}{\mathcal{C}}
\newcommand{\D}{\mathcal{D}}
\newcommand{\N}{\mathcal{N}}
\newcommand{\F}{\mathcal{F}}
\newcommand{\OO}{\mathcal{O}}

\renewcommand{\L}{\mathbb{L}}

\begin{document}
	
	\title{Coideal subalgebras of pointed and connected Hopf algebras}
	
	\author{G.-S. Zhou}

	\address{\rm School of Mathematics and Statistics, Ningbo University, Ningbo 315211, China
		\newline \indent Shanghai Key Labratory of Pure Mathematics and Mathematical Practice
		\newline \indent E-mail: zhouguisong@nbu.edu.cn}
	
\begin{abstract}
Let $H$ be a pointed Hopf algebra with abelian coradical. Let $A\supseteq B$ be left (or right) coideal subalgebras of $H$ that contain the coradical of $H$. We show that $A$ has a PBW basis over $B$, provided that $H$ satisfies certain mild conditions. In the case that $H$ is a connected graded Hopf algebra of characteristic zero and $A$ and $B$ are both homogeneous of finite Gelfand-Kirillov dimension, we show that $A$ is a graded iterated Ore extension of $B$. These results turn out to be conceptual consequences of a structure theorem for each pair $S\supseteq T$ of homogeneous coideal subalgebras of a connected graded braided bialgebra $R$ with braiding satisfying certain mild conditions. The structure theorem claims the existence of a well-behaved PBW basis of $S$ over $T$. The approach to the structure theorem is constructive by means of a combinatorial method based on Lyndon words and braided commutators, which is originally developed by V. K. Kharchenko \cite{Kh} for primitively generated braided Hopf algebras of diagonal type. Since in our context we don't priorilly assume $R$ to be primitively generated, new methods and ideas are introduced to handle the corresponding difficulties, among others.
\end{abstract}
	


	
	\maketitle

	
	\tableofcontents
	
	
	
\section*{Introduction}

\newtheorem{maintheorem}{\bf{Theorem}}
\renewcommand{\themaintheorem}{\Alph{maintheorem}}
	
The (one-sided) coideal subalgebras of a Hopf algebra $H$ have been an important research focus since the classical work on the commutative case in the last century \cite{DeGa}. Recall that a \emph{left ({\rm resp.} right) coideal subalgebra}  of  $H$ is a subalgebra $A$ of $H$ with $\Delta_H(A) \subseteq H\otimes A$ (resp. $\Delta_H(A) \subseteq A\otimes H$). They are conceptually the right context to define and understand quantum homogeneous spaces. Many studies of coideal subalgebras have concentrated on this aspect. Among others, freeness and faithfully flatness of a Hopf algebra over its coideal subalgebras are investigated \cite{Take1, Masu1, MuSc, Skry}; fundamental homological properties and invariants of coideal subalgebras are characterized \cite{Krae, LiWu}. Nevertheless, a big reason that coideal subalgebras became more and more important is that many Hopf algebras, prominently the quantized enveloping algebras $U_q(\mathfrak{g})$, do not have ``enough'' Hopf subalgebras. For example, coideal subalgebras are employed to develop a Galois theory for Hopf algebra actions \cite{FMP, Yana}. Particularly noteworthy is the theory of quantum symmetric pairs originally developed by G. Letzter as coideal subalgebras of quantized enveloping algebras \cite{Letz0, Letz}. This theory has been evolving rapidly over the past decade, led to a flurry of work aiming to extend many quantum group related constructions to the setting of quantum symmetric pairs, see the recent survey by W. Wang \cite{Wang} and the references there. 

In this paper, we consider the coideal subalgebras of pointed Hopf algebras with abelian coradical. This class of Hopf algebras includes many important examples, and under some finiteness conditions, the general structure of them have been explored in deep via the lifting method introduced  by N. Andruskiewitsch and H.-J. Schneider in \cite{AnSc1}. This paper will focus on the construction of PBW bases for the coideal subalgebras. A well-behaved PBW basis or even the existence of a PBW basis is manifestly useful in the understanding of an algebra. There are many publications on the construction of PBW bases for Hopf algebras, notably for quantized enveloping algebras \cite{Lusz, Ring, Lec}. By means of a remarkable combinatorial method based on Lyndon words, V. K. Kharchenko has managed to construct a PBW basis for each character Hopf algebra \cite{Kh} and later for each of their coideal subalgebras that containing the coradical \cite{Kh1}. Note that character Hopf algebras are all pointed with abelian coradical. Kharchenko's PBW basis is one of  the keystones in the study of Nichols algebras of diagonal type. There is a complete classification of right coideal subalgebras of the Borel part  of  $U_q(\mathfrak{g})$ containing the coradical   in the case that $q$ is not a root of unity \cite{HS0}. When  $\mathfrak{g}$ is of type $A_n$, $B_n$ and $G_2$, it has been already obtained in \cite{KhSa}, \cite{Kh2} and \cite{Pogo} respectively by similar methods in a parallel way. Kharchenko's PBW basis also plays a significant role in all of these works. Following the idea of Kharchenko, we prove the following general statement on the PBW bases of the coideal subalgebras.

\begin{maintheorem}[Theorem \ref{pointed-Hopf-1}]\label{maintheorem-pointed} 	
Let $H$ be a pointed Hopf algebra with  $G=G(H)$ abelian. Assume that one of the following two conditions hold: 
(1) $H$ is locally finite as a $kG$-module under the adjoint action of $kG$ on $H$, and the base field $k$ is algebraically closed; (2) $H$ is generated over  $kG$ by a set of semi-invariants of $H$.
Let $A\supseteq B$ be left (resp. right) coideal subalgebras of $H$ that contain $G$. Then there exists a family $\{z_\xi\}_{\xi\in \Xi}$ of elements in $A$ and a map $h:\Xi\to \mathbb{N}_\infty:=\mathbb{N}\cup \{\infty\}$ such that $(\{z_\xi\}_{\xi\in \Xi}, h, <)$ is a system of PBW generators of $A$ over $B$ for each total order $<$ on $\Xi$.
\end{maintheorem}	

Here, $G=G(H)$ is the group of group-like elements of $H$, and by a \emph{semi-invariant of  $H$} we mean an element $x\in H$ such that $kx$ is stable under the adjoint action of $kG$ on $H$. For the notion of systems of PBW generators of an algebra over its subalgebras, see Definition \ref{PBW-definition}. The setting of condition (1) covers the case that $G$ is finite and $k$ is algebraically closed; and the setting of condition (2) includes character Hopf algebras, which in turn covers quantized enveloping algebras, all their generalizations, bosonizations of Nichols algebras of diagonal type, and so on. Note that the setting of condition (1) is more general than that of condition (2) when $k$ is algebraically closed. Theorem \ref{maintheorem-pointed} generalizes and strengthens the main results of \cite{Kh1} in several aspects. First of all, we greatly expand the class of Hopf algebras; secondly, we have considerably more freedom on the pairs $A\supseteq B$, while in loc. cit. either $A=H$ or $B=kG$; finally and most importantly, we get rid of the ad-invariant condition on coideal subalgebras, which is assumed in  loc. cit. and seems to be very restrictive.

A \emph{connected Hopf algebra} is a Hopf algebra with coradical of dimension one. Clearly, they are pointed Hopf algebras with abelian coradical. Though the condition of being connected is quite restrictive, it has rich examples from diverse fields of mathematics, such as  universal enveloping algebras of Lie algebras, the coordinate rings of unipotent algebraic groups, the  Hopf algebras of symmetric functions, of quasi-symmetric functions and of permutations, and so on. Recently, connected Hopf algebras of finite Gelfand-Kirillov dimension (GK dimension, for short) have been the subject of a series of papers. Those of GK dimension up to  $4$ were completely classified, when the base field is algebraically closed of characteristic zero \cite{Zh,WZZ}; the subclass of iterated Hopf Ore extensions has been characterized \cite{BOZZ, BZ2, LZ, ZSL3}; some fundamental properties of themself are extended to  their coideal subalgebras \cite{BG}; and an example of connected graded Hopf algebra of GK dimension $5$ is founded which is neither commutative nor cocommutative \cite{BGZ}. Finite-dimensional connected Hopf algebras in positive characteristic  also have been studied \cite{ESW19, NWW19}. As a motivation of this work, the author and his collaborators have constructed a well-behaved PBW basis for any connected graded Hopf algebra over any of its homogeneous Hopf subalgebras when the base field is of characteristic zero; and consequently it turns out  that those of finite GK dimension are   iterated Hopf Ore extensions of their homogeneous Hopf subalgebras \cite{LZ, ZSL3}. In this paper, we extend the PBW basis to the setting of coideal subalgebras (Theorem \ref{connected-graded-Hopf-algebra-PBW}), and prove the following result in the viewpoint of Ore extensions.

\begin{maintheorem}[Theorem \ref{IHOE-graded}]\label{maintheorem-connected} 
Assume that $k$ is of characteristic $0$. Let $H$ be a connected $\Gamma$-graded Hopf algebra, where $\Gamma$ is a  nontrivial free abelian monoid. Let $A\supseteq B$ be homogeneous left (resp. right) coideal subalgebras of $H$, which  are of finite GK dimension $m$ and $n$ respectively. Then there is a sequence  $B=K_0 \subset K_1\subset \cdots \subset K_{m-n} = A$ of homogeneous left (resp. right) coideal subalgebras of $H$ such that each $K_i$ is a graded Ore extension of $K_{i-1}$ of derivation type. Moreover, if $A$ and $B$ are Hopf subalgebras of $H$ then $K_i$ can be chosen to be Hopf subalgebras of $H$.
\end{maintheorem}
Here, an Ore extension $S=R[x; \phi, \delta]$ is called \emph{of derivation type} if $\phi=\id_R$. In addition to Theorem \ref{maintheorem-connected}, we also generalize some of the main results of \cite{BG, Zh}, either from connected Hopf algebras to their coideal subalgebras or from algebraically closed fields of characteristic zero to arbitrary fields of characteristic zero. For example, coideal subalgebras of connected Hopf algebras of characteristic zero are shown to be deformations of polynomial algebras (Theorem \ref{GKdim-transfer}). Particularly noteworthy is that our argument is purely algebraic, while in \cite{BG, Zh} some geometric facts are employed.

The contexts of the above two theorems seem to be of different nature. However, we manage to handle them in a same front by taking advantage of the theory of braided Hopf algebras. Obviously, connected graded Hopf algebras may be considered to be braided by the flipping map. For a pointed Hopf algebra $H$ with coradical $K$, it is well-known that the associated graded Hopf algebra $\gr(H)$ with respect to the coradical filtration of $H$ is isomorphic to the bosonization (or the Radford biproduct) of $R:=\gr(H)^{{\rm co}\, K}$ and $K$ \cite{Rad85}. Note that $R$ is a connected graded braided Hopf algebra, and the structure as well as the coideal subalgebras  of $H$ are closely related to that of $R$. It turns out that Theorem \ref{maintheorem-pointed} and Theorem \ref{maintheorem-connected} are conceptual consequences of the following result.

\begin{maintheorem}[Theorem \ref{braided-bialgebra-1}]\label{maintheorem-braided} 
	Let $R=(R,\tau)$ be a connected $\Gamma$-graded braided bialgebra, where $\Gamma$ is a  nontrivial free abelian monoid. Assume that $\tau=\tau_\chi$ for some bicharacter $\chi$ of $\Gamma$.
	Let $A\supseteq B$ be homogeneous left (resp. right) coideal subalgebras of $R$. Then there exists a family $\{z_\xi\}_{\xi\in \Xi}$ of homogeneous elements of $A$, a map $h:\Xi\to \mathbb{N}_\infty$ and a total order $\vartriangleleft$ on $\Xi$ such that
	\begin{enumerate}	
		\item  $(\{z_\xi\}_{\xi\in \Xi}, h, <)$ is a system of PBW generators of $A$ over $B$ for each total order $<$ on $\Xi$.
		\item  For each $\xi\in \Xi$, the subalgebra $A^{\trianglelefteq \xi}$ of $A$ generated by $\{~z_\eta~|~\eta\in \Xi,~ \eta \trianglelefteq \xi ~\}$ over $B$ has  $$\{~ z_{\xi_1}^{r_1}\cdots z_{\xi_m}^{r_m} ~|~ m\geq 0,\, \xi_1\vartriangleleft \cdots \vartriangleleft \xi_m \trianglelefteq \xi,\, \xi_i \in \Xi,  \, r_i< h(\xi_i),  \, i=1,\ldots, m ~\}$$ as a basis of left $B$-module as well as of right $B$-module. 
		\item  For each $\xi\in \Xi$, the subalgebra $A^{\trianglelefteq \xi}$ is a left (resp. right) coideal of $R$, and  it is a subcoalgebra of $R$ in the case that $A, B$ are both subcoalgebras of $R$  and $\chi^2=\varepsilon_\Gamma$.
		\item For each $\xi\in \Xi$ with $h(\xi)<\infty$,  it follows that  $z_\xi^{h(\xi)}\in \bigcup_{\delta\vartriangleleft \xi} A^{\trianglelefteq \delta}$.
		\item For all $\xi,\eta\in \Xi$ with $\eta \vartriangleleft \xi$, it follows that  $[z_\xi, z_\eta]_\tau \in \bigcup_{\delta\vartriangleleft \xi} A^{\trianglelefteq \delta}$ (resp. $[z_\xi, z_\eta]_{\tau^{-1}} \in \bigcup_{\delta\vartriangleleft \xi} A^{\trianglelefteq \delta}$).
		\item  For each $\xi \in \Xi$ with $h(\xi) <\infty$, the scalar $a_\xi=\chi(\deg(z_\xi),\deg(z_\xi))$ is a root of unity. More precisely,  if  $k$ is of characteristic $0$ then $a_{\xi}$ is of order $h(\xi)$, and in particular, $a_{\xi}\neq 1$; and if $k$ is of characteristic $p>0$ then  $a_{\xi} $ is of order $h(\xi)/p^s$ for some integer $s\geq 0$.
	\end{enumerate}
\end{maintheorem}

We refer to Example \ref{bicharacter} and the paragraph before it for the notions and notation related to bicharacters of abelian monoids. Particularly, $\varepsilon_\Gamma$ denotes the trivial bicharacter of $\Gamma$.

Now let us sketch the proof of Theorem \ref{maintheorem-braided}. The approach is constructive, following the ideas used in \cite{Kh, Uf, ZSL3, LZ}. First we choose a set $X$ of homogeneous generators for $R$ such that $Y_1:= X\cap B$ and $Y_2:= X\cap A$ generate $B$ and $A$ respectively; and then fix an appropriate well order $<$ on $X$ so that  $Y_1$ and $Y_2$ are both closed (see Definition \ref{definition-closed}), among others. Let $k\langle X\rangle$ be the free algebra on $X$. Note that there is a canonical braiding on $k\langle X\rangle$ induced from that of $R$. It is also denoted by $\tau$ and  makes  $k\langle X\rangle$ a braided algebra. Denote by $I$ the ideal of defining relations  on $X$ for  $R$. Then we  may identify  $R$  with $k\langle X\rangle /I$ as graded algebras canonically.  In addition, let $\N_I$ be the set of $I$-irreducible  Lyndon words on $X$ with respect to the graded lex order associated to $<$. By the standard Gr\"{o}bner basis theory, we may further define a height map $h_I:\mathbb{L} \to \mathbb{N}_\infty$ (see Definition \ref{definition-height}), where $\mathbb{L}$ is the set of Lyndon words on $X$. By the combinatorial properties of Lyndon words and the braided commutator of polynomials, one may construct from $X$ a new family of homogeneous generators $(z_\gamma)_{\gamma\in \N_I}$ for $R$ (Lemma \ref{basis-quotient-algebra}).  Next we employ the coalgebra structure of $R$. Since $A$ and $B$ are  left (resp. right) coideals of $R$, it is easy to lift  $\Delta_R$ into a graded algebra homomorphism $\Delta: k\langle X\rangle \to  k\langle X\rangle \otimes^{\tau}  k\langle X\rangle$ that  is right (resp. left) bounded  (see Definition \ref{definition-bounded-comultiplication}).  Generally, $\Delta(x) \neq 1\otimes x+x\otimes 1$ because $x$ is not necessarily primitive in $R$; and moreover $\Delta$ is not necessarily coassociative in the sense that $(\Delta \otimes \id ) \circ \Delta = (\id \otimes \Delta) \circ \Delta.$  Nevertheless, as an immediate consequence of  a  technical result (Proposition \ref{PBW-diagonal}), the one-sided boundedness of $\Delta$ is sufficient to deduce that this new  family of  generators satisfy some desirable conditions. Due to the combinatorial feature of Lyndon words and the closedness of $Y_1$ and $Y_2$, we further show that  $(z_\gamma)_{\gamma\in \Xi_1}$ and $(z_\gamma)_{\gamma \in \Xi_2}$, where $\Xi_i:= \N_I\cap \langle Y_i\rangle$, are  well-behaved families of generators of $B$ and $A$ respectively (Corollary \ref{Lyndon-compare}). Combine Proposition \ref{PBW-diagonal} and Corollary \ref{Lyndon-compare}, as well as some other ideas (particularly, Proposition \ref{PBW-generator}), we conclude that the family $(z_\gamma)_{\gamma\in \Xi}$, where $\Xi:= \Xi_2\backslash \Xi_1$, together with the restrictions of the height map $h_I$ and the pseudo-lexicographic order $<_{\lex}$ to $\Xi$ satisfy the requirements.

The paper is organized as follows. In section 1, we collect some basic notions that will be used in this paper, including those of PBW generators, braided structures and Lyndon words. Sections 2-3 are the technical part of this paper. In Section 2, we construct a system of PBW generators for algebras that satisfy certain conditions by a combinatorial method based on Lyndon words and braided commutators. In Section 3, we introduce the notion of bounded comultiplications on free algebras and study their properties. Some keystone observations are presented in this section, with the aim to prove Theorem \ref{maintheorem-braided}.  Sections 4-6 are the main part of this paper. They are devoted to study the coideal subalgebras of connected graded braided bialgebras, pointed Hopf algebras with abelian coradical and connected Hopf algebras respectively. Particularly, the above three main theorems are proved there. The final two sections are devoted to address the proof of two technical results (Proposition \ref{PBW-diagonal} and Proposition \ref{braided-bialgebra-3}), one for each. In Appendix A, we introduce the class of Lyndon ideals for free algebras and study the combinatorial and homological properties of the corresponding quotient algebras. It is considered as an organic part of the combinatorial method that  used in this paper.

\subsection*{Notation and conventions} Throughout the paper, we work over a fixed field  denoted by $k$. All vector spaces, algebras, coalgebras and unadorned tensors are over $k$.  For an algebra $A$, we denote by $\mu_A$ and $\eta_A$  the multiplication map and the unit map of $A$ respectively; for a coalgebra $C$, we denote by $\Delta_C$ and $\varepsilon_C$  the comultiplication map and the counit map of $C$ respectively; and for a Hopf algebra $H$, we denote by $\mathcal{S}_H$ the antipode of $H$. The notation $\mathbb{N}$ denotes the set of natural numbers including $0$. Let $\mathbb{N}_\infty=\mathbb{N}\cup \{\infty\}$. For a set $X$, let $\#(X)$ be the number of elements of $X$. By convention, $\#(X)=\infty$ if $X$ is infinite. For all positive integers $i\leq n$ and all scalars $q\in k$, let
$$\binom{n}{i}_q:=\frac{(n)_q\cdots (n-i+1)_q}{(i)_q\cdots (1)_q},$$ where $(r)_q:=\sum_{j=0}^{r-1}q^r$ for $r\geq 1$. By convention, we write $\binom{n}{0}_q:=1$ for $n\geq 0$.

\section{Preliminaries}
	\label{section-Lyndon}

In this section, we collect some basic notions that will be used in the sequel. We refer to \cite{HS, Take} for those on braided structures, and to \cite{GF, Kh, Lo, ZSL3} for those on Lyndon words.

\subsection{PBW generators}\hfill

\begin{definition}\label{PBW-definition}
Let $A$ be an algebra and $B$ a subalgebra of $A$. A \emph{system of PBW generators of $A$ over $B$} is a triple $(\{z_\xi\}_{\xi\in \Xi}, h, <)$, where $\{z_\xi\}_{\xi\in \Xi}$ is a family  of elements of $A$, $h$ is a map from $\Xi$ to  $\mathbb{N}_\infty$  and  $<$ is a total order on $\Xi$, such that $h(\xi)\geq 2$ for each $\xi\in \Xi$ and the set
\[
B(\{z_\xi\}_{\xi\in \Xi}, h, <):=\{~ z_{\xi_1}^{r_1}\cdots z_{\xi_m}^{r_m} ~|~ m\geq 0,\, \xi_1< \cdots < \xi_m, \, \xi_i \in \Xi, \, r_i< h(\xi_i), \, i=1,\ldots, m ~\}
\]
is a basis of $A$ as a left and right $B$-module. If $B=k$ is the base field, then we shall call the triple $(\{z_\xi\}_{\xi\in \Xi}, h, <)$   a  \emph{system of PBW generators} of $A$. In the case that $h(\xi)=\infty$ for all $\xi\in \Xi$, we simply call the pair $(\{z_\xi\}_{\xi\in \Xi}, <)$ a system of PBW generators of $A$ over $B$.
\end{definition}

\begin{example}
Let  $\mathfrak{g}$ be a Lie algebra and $\mathfrak{h}$ a Lie subalgebra of $\mathfrak{g}$. Let  $\{x_i\}_{i\in I}$ be a family of elements of $\mathfrak{g}$  which spans a complement of $\mathfrak{h}$ in $\mathfrak{g}$. Then by the well-known PBW theorem, $(\{x_i\}_{i\in I}, <)$ is a system of PBW generators of $U(\mathfrak{g})$ over $U(\mathfrak{h})$  for any total order $<$ on $I$. 
\end{example}

\begin{example}
Let $\mathcal{A}$ be a finitary  exact category which is Krull-Schmidt. Let $H_{\mathcal{A}}$ be the Hall algebra of $\mathcal{A}$ over the field of rational numbers, which as a  vector space is freely spanned by the isomorphism classes of objects of $\mathcal{A}$. Then for any total order $<$ on the set ${\rm Ind } (\mathcal{A})$ of isomorphism classes of indecomposable objects of $\mathcal{A}$, the pair $({\rm Ind } (\mathcal{A}), <)$ is a system of PBW generators of $H_{\mathcal{A}}$. We refer to \cite{BeGr} for the undefined notions, and to Theorem 2.4 of loc. cit. for this fact.
\end{example}

\begin{lemma}\label{PBW-GK-dimension}
Let $A$ be an algebra and $B$ a subalgebra of $A$. Assume that $A$ has a system of PBW generators $(\{z_\xi\}_{\xi\in \Xi}, h, <)$ over $B$. Then $\gkdim A\geq \gkdim B + \#(\{\xi\in \Xi ~|~h(\xi)=\infty\})$.
\end{lemma}

\begin{proof}
Let $U$ be any finite-dimensional subspace of $B$ containing $1_B$. Let $\Omega$ be any finite subset of $\{\xi\in \Xi ~|~h(\xi)=\infty\}$. Let $V=U+k\Omega$. Then for each integer $n\geq 0$,
\[
U^n \Omega^{(n)} \subseteq (U+k\Omega)^{2n} =V^{2n}, 
\]
where $\Omega^{(n)} =\{z_{\xi_1}\cdots z_{\xi_m}~| ~ m\leq n, ~ \xi_1\leq \cdots \leq \xi_m, ~ \xi_i\in \Omega\}$. Since  $\Omega^{(n)}$ is linearly independent over $B$, one has $\dim (U^n \Omega^{(n)}) =\dim (U^n) \cdot \#(\Omega^{(n)})$. Thus
\begin{eqnarray*}
&&\gkdim A \geq \overline{\lim_{n\to \infty}} \log_n \dim (V^{2n}) \geq \overline{\lim_{n\to \infty}} \log_n \big(\dim (U^n) \cdot \#(\Omega^{(n)})\big) \\
&& \quad \quad  \quad   \quad  = \overline{\lim_{n\to \infty}} \log_n \dim (U^n) + \lim_{n\to \infty} \log_n \binom{n+\#(\Omega)}{\#(\Omega)} \\
&& \quad \quad  \quad   \quad  = \overline{\lim_{n\to \infty}} \log_n \dim (U^n) + \#(\Omega).
\end{eqnarray*}
Here, note that $\#(\Omega^{(n)})=\binom{n+\#(\Omega)}{\#(\Omega)}$. Since $U$ is arbitrary, the result follows.
\end{proof}

\begin{remark}
The equality of Lemma \ref{PBW-GK-dimension} fails in general. Counterexamples already occur in the context of Ore extensions, see \cite[Proposition 3.9]{KL}. We expect to find some efficient sufficient conditions on $A$, $B$ or the system of PBW generators itself to make the equality happen.
\end{remark}


\begin{definition}
A partial order $<$ on an abelian monoid $\Gamma=(\Gamma, +)$ is said to be \emph{admissible} if  
$$\gamma_1<\gamma_2\Longrightarrow \gamma+\gamma_1<\gamma+\gamma_2, \quad \gamma_1,\, \gamma_2,\, \gamma\in \Gamma.$$
A \emph{well-ordered abelian monoid} is an abelian monoid  with a specific admissible well order. 

Note that  free abelian monoids all have admissible well orders, and the neutral element $0$ is the smallest element in a well-ordered abelian monoid.
\end{definition}

Let $\Gamma= (\Gamma,<)$ be a well-ordered abelian monoid. Let  $\mathcal{F}=(F_\gamma (V))_{\gamma\in \Gamma}$ be a \emph{$\Gamma$-filtration} of a vector space $V$. So by definition, it is a  family of subspaces of $V$ such that  
\begin{eqnarray*}
F_\gamma (V)\subseteq F_\delta (V) ~\text{ for all }~ \gamma, \delta \in \Gamma  ~ \text{ with } ~ \gamma<\delta, \quad \text{ and } \quad  V=\bigcup_{\gamma\in \Gamma}F_\gamma (V).\end{eqnarray*} 
The \emph{associated  $\Gamma$-graded vector space of $V$ with respect to $\mathcal{F}$} is defined to be
\[
\gr(V,\mathcal{F}):= \bigoplus_{\gamma\in \Gamma}F_\gamma (V)/F_\gamma^-(V), \quad \text{ where } F_\gamma^- (V):= \bigcup_{\delta<\gamma} F_\delta (V).
\]
For each element $a\in V$, we shall write $$\overline{a}:= a+ F_{\gamma(a)}^- (V)\in \gr(V,\mathcal{F}(V)),$$ where $\gamma(a)$ is the  smallest element $\gamma\in \Gamma$ such that $a\in F_\gamma (V)$. 
Note that every subspace $U$ of $V$ has a natural $\Gamma$-filtration $\mathcal{F}|_U:=(U\cap F_\gamma (V))_{\gamma\in \Gamma}$. We identify $\gr(U,\mathcal{F}|_U)$ with a homogeneous (i.e. $\Gamma$-graded) subspace of $\gr(V,\mathcal{F})$ under the natural injection $\gr(U,\mathcal{F}|_U) \to \gr(V,\mathcal{F})$. 

An \emph{algebra $\Gamma$-filtration} of an algebra $A$ is a $\Gamma$-filtration $\mathcal{F}=(F_\gamma (A))_{\gamma\in \Gamma}$ of $A$  such that  
\[
1_A\in F_0(A) \quad \text{and} \quad F_\alpha(A)F_\beta(A) \subseteq F_{\alpha+\beta} (A), \quad \alpha,\beta\in \Gamma.
\]
A \emph{coalgebra  $\Gamma$-filtration} of a coalgebra $C$ is a $\Gamma$-filtration  $\mathcal{F}=(F_\gamma (C))_{\gamma\in \Gamma}$ of $C$ such that  
\[
\Delta_C(F_\gamma(C)) \subseteq \sum_{\beta_1,\, \beta_2\in \Gamma, ~\beta_1+\beta_2 \leq \gamma} F_{\beta_1}(C) \otimes F_{\beta_2}(C), \quad \gamma \in \Gamma.
\]
Note that $\gr(A,\mathcal{F})$ (resp. $\gr(C,\mathcal{F})$) is naturally a $\Gamma$-graded algebra (resp. coalgebra).

\begin{lemma}\label{PBW-tranfer}
Let $B$ be a subalgebra of an algebra $A$. Let $\mathcal{F}=(F_\gamma(A))_{\gamma\in \Gamma}$ be a $\Gamma$-filtration  of $A$, where $\Gamma$ is a  nontrivial well-ordered abelian monoid. Given a family $\{z_\xi\}_{\xi\in \Xi}$ of elements of $A$,  a map $h:\Xi\to \mathbb{N}_\infty$ and  a total order $<$  on $\Xi$, if  $(\{\overline{z_\xi}\}_{\xi\in \Xi}, h, <)$ is a system of PBW generators of  $\gr(A, \mathcal{F})$ over $\gr(B,\mathcal{F}|_B)$, then $(\{z_\xi\}_{\xi\in \Xi}, h, <)$ is a system of PBW generators of $A$ over $B$.
\end{lemma}

\begin{proof}
	It follows by a standard application of the filtered-graded  method.
\end{proof}


\subsection{Braided structures}\hfill

\begin{definition} A \emph{braiding} on a vector space $V$ is a linear automorphism $\tau:V\otimes V \to V\otimes V$ which satisfies the braid equation (i.e. quantum Yang-Baxter equation):
	\[
	(\tau\otimes \id_V )\circ (\id_V \otimes \tau) \circ (\tau\otimes \id_V ) = (\id_V \otimes \tau)\circ(\tau\otimes \id_V )\circ(\id_V \otimes \tau).
	\]
A \emph{braided vector space} is a vector space together with a braiding on it. 
A \emph{homomorphism of braided vector spaces} $f:(V, \tau_V)\to (W, \tau_W)$ is a linear map   such that $\tau_W( f \otimes f ) = ( f \otimes f )\tau_V$.

A \emph{braided subspace} of a braided vector space $(V,\tau)$ is a subspace $U$ of $V$ such that $\tau(U\otimes U) =U\otimes U$. In this case, $U$ itself is a braided vector space with braiding the restriction of $\tau$ on $U\otimes U$.
\end{definition}

\begin{example}
Let $\mathcal{X}$ be a basis of a vector space $V$. Let $\tau: V\otimes V\to V\otimes V$ be a linear map such that $\tau(x\otimes y) \in k^\times\cdot (y\otimes x)$ for all $x, y\in \mathcal{X}$, where $k^\times:=k\backslash\{0\}$. Clearly, $\tau$ is a braiding on $V$. We call such braidings on $V$ \emph{diagonal with respect to $\mathcal{X}$} (or simply \emph{$\mathcal{X}$-diagonal}).
\end{example}

Let $\Gamma$ be an abelian monoid with binary operation $+$ and neutral element $0$. Recall that a \emph{bicharacter} of $\Gamma$
is  map $\chi: \Gamma\times \Gamma \to k^\times$ such that  for all $\alpha, \beta,\gamma\in \Gamma$,
\[
\chi(\alpha, \beta+\gamma) = \chi(\alpha, \beta) ~ \chi(\alpha, \gamma) \quad \text{and} \quad \chi(\alpha+\beta, \gamma) = \chi(\alpha, \gamma)~\chi(\beta, \gamma).
\]
Note that a bicharacter $\chi$ satisfies $\chi(0,\gamma)=\chi(\gamma, 0) =1$ for all $\gamma\in \Gamma$. Let $\varepsilon_\Gamma$ be the  \emph{trivial bicharacter} given by $\varepsilon_\Gamma(\alpha,\beta)=1$. For bicharacters $\chi,\chi'$ of $\Gamma$, the bicharacter $\chi\chi'$ of $\Gamma$ is defined by
\[
(\chi\chi') (\alpha, \beta) = \chi(\alpha,\beta) ~ \chi'(\alpha,\beta), \quad \alpha,\beta\in \Gamma.
\]
With this binary operation, the set of all bicharacters of $\Gamma$ is an abelian monoid with unit $\varepsilon_\Gamma$.

\begin{example}\label{bicharacter}
Let $V=\bigoplus_{\gamma\in \Gamma} V_\gamma$ be a $\Gamma$-graded vector space, where $\Gamma$ is an abelian monoid. For each bicharacter $\chi$ of $\Gamma$, there is associated with a canonical braiding $\tau_\chi$ on $V$ defined by 
\[
\tau_{\chi} (u\otimes v)=\chi(\alpha, \beta) \cdot v\otimes u, \quad u\in V_\alpha,~ v\in V_\beta, ~ \alpha, ~ \beta\in \Gamma.
\]
Note that $\tau_\chi$ is diagonal with respect to each homogeneous basis of $V$.
\end{example}

Let  $\tau$ be a braiding on a vector space $V$. For all integers $m,n\geq1$, we define the isomorphisms $\tau_{m,n}: V^{\otimes m}\otimes V^{\otimes n} \to  V^{\otimes n}\otimes V^{\otimes m}$ inductively by $\tau_{1,1} :=\tau$, 
\begin{eqnarray*}
	\tau_{1,n} &:=& (\id_V\otimes \tau_{1,n-1})\circ (\tau\otimes \id_{V^{\otimes n-1}})\quad \text{for} \quad n\geq 2, \quad \text{and} \\
	\tau_{m,n} &:=& (\tau_{m-1,n}\otimes \id_V)\circ (\id_{V^{\otimes m-1}}\otimes \tau_{1,n}) \quad \text{for}\quad  m\geq 2,\, n\geq 1.
\end{eqnarray*}
Let $V^{\otimes 0}=k$, and denote  for all $n\geq 0$ by $\tau_{0,n}:  k\otimes V^{\otimes n} \to  V^{\otimes n}\otimes k$ and $\tau_{n,0}: V^{\otimes n}\otimes k \to  k\otimes V^{\otimes n}$ the canonical isomorphisms. We say that a linear map $f:V^{\otimes m} \to V^{\otimes n}$  \emph{commutes with $\tau$} if $$(f\otimes \id_V)\circ \tau_{1,m} =  \tau_{1,n}\circ (\id_V\otimes f) \quad \text{and} \quad (\id_V\otimes f)\circ \tau_{m,1} = \tau_{n,1} \circ (f\otimes \id_V).$$
The braid equation tells us that $\tau$ commutes with itself.

\begin{definition}
A \emph{braided algebra (resp. coalgebra)} is an algebra (resp. coalgebra) $R$ together with a braiding $\tau$ on $R$ such that the structure maps of $R$ all commute with $\tau$. 

A \emph{homomorphism of braided algebras (resp.  coalgebras)} $R\to S$ is a homomorphism of algebras (resp. coalgebras) which is also a homomorphism of braided vector spaces.
\end{definition}

For each braided algebra $(A, \tau)$, there is an algebra $A\otimes^\tau A$. It is $A\otimes A$ as a vector space, but  with multiplication  $\mu_{A\otimes^\tau A}:= (\mu_A \otimes \mu_A) \circ (\id_A\otimes \tau \otimes id_A) $   and  unit   $\eta_{A\otimes^\tau A} =\eta_A\otimes \eta_A$. Dually, for each braided coalgebra $(C, \tau)$, there is a  coalgebra $C\otimes^\tau C$. It is $C\otimes C$ as a vector space, but with comultiplication   $\Delta_{C\otimes^\tau C}:= (\id_C\otimes \tau\otimes \id_C)\circ (\Delta_C\otimes \Delta_C)$ and counit  $\varepsilon_{C\otimes^\tau C}:= \varepsilon_C\otimes \varepsilon_C$.


\begin{definition}
A \emph{braided bialgebra} is a tuple $(R, \mu, \eta, \Delta, \varepsilon, \tau)$ such that $(R, \mu, \eta, \tau)$ is a braided algebra, $(R, \Delta, \varepsilon, \tau)$ is a braided coalgebra, and one of the following equivalent conditions holds:
\begin{itemize}
	\item $\Delta: R\to R\otimes^\tau R$ and $\varepsilon: R\to k$ are homomorphisms of algebras.
	\item $\mu: R\otimes^\tau R \to R$ and $\eta: k\to R$ are homomorphisms of coalgebras.
\end{itemize}
A \emph{braided Hopf algebra} is a braided bialgebra $R$ such that $\id_R$ is convolution invertible. In this case, the 
convolution inverse of $\id_R$ is called the antipode of $R$. 
	
A \emph{homomorphism of braided bialgebras (Hopf algebras)} $R\to S$ is a homomorphism of braided algebras which is also a homomorphism of braided coalgebras.
\end{definition}

The above braided structures  have graded version. Let $\Gamma$ be an abelian monoid.  A \emph{$\Gamma$-graded braided vector space} is a braided vector space $(V,\tau)$ with a  $\Gamma$-grading $V=\bigoplus_{\gamma\in \Gamma}V_\gamma$ such that 
$$\tau(V_\alpha\otimes V_\beta)= V_\beta \otimes V_\alpha, \quad \alpha,~ \beta\in \Gamma.$$
A \emph{homomorphism of $\Gamma$-graded braided vector spaces} $V\to W$ is a homomorphism of braided vector spaces which is also a homomorphism of $\Gamma$-graded vector spaces. 

\begin{definition}
A \emph{$\Gamma$-graded braided algebra (coalgebra, bialgebra, respectively)} is a braided algebra (coalgebra, bialgebra, respectively)  together with a $\Gamma$-grading  such that it is a   $\Gamma$-graded algebra (coalgebra, algebra and coalgebra, respectively) and a $\Gamma$-graded braided vector space. A \emph{$\Gamma$-graded braided Hopf algebra} is a $\Gamma$-graded braided bialgebra with an antipode.

A \emph{homomorphism of $\Gamma$-graded braided algebras (coalgebras, bialgebras, Hopf algebras, respectively)} $R\to S$ is a homomorphism of braided algebras  (coalgebras, bialgebras, Hopf algebras, respectively) which is also a homomorphism of $\Gamma$-graded vector spaces. 
\end{definition}

\begin{remark}
Let $K$ be a Hopf algebra with bijective antipode. Let ${}^K_K\mathcal{Y}{D}$ be the category of  (left)  Yetter-Drinfeld modules over  $K$. It is a braided monoidal category.  The \emph{braiding of $V, W\in {}^K_K\mathcal{Y}{D}$ } is defined to be the linear map $\tau_{V,W}:V\otimes W\to W\otimes V$ given by 
\[
\tau_{V,W} (v\otimes w) = \sum (v_{[-1]}\cdot w)\otimes v_{[0]}, \quad  v\in V, ~w\in W.
\]
In particular, every  object in ${}^K_K\mathcal{Y}{D}$ is naturally a braided vector space. Moreover, every algebra ( coalgebra, etc.) in ${}^K_K\mathcal{Y}{D}$ is naturally a braided algebra (coalgebra, etc.) respectively in the sense of above definitions. Similar remarks apply to the graded versions of these structures. 
\end{remark}


\subsection{Lyndon words} \hfill

Let $X$ stand for a well-ordered alphabet.  If $X=\{x_1,\ldots, x_\theta\}$ for some positive integer $\theta$, then $X$ is tacitly equipped with the order $x_1<\cdots <x_\theta$. 
We denote by $\langle X\rangle$ the set of all words on $X$, by $1$ the empty word and by $\langle X\rangle_+$ the set of nonempty words on $X$. Note that $\langle X\rangle$ is a monoid under the concatenation of words. The length of a word $u$ is denoted by $|u|$. 

Let $u, \, v$ be two words on $X$. We call $v$ a \textit{factor}  of  $u$  if there exist words $w_1,w_2$ on $X$ such that  $w_1vw_2=u$. If $w_1$ (resp. $w_2$) can be chosen to be the empty word then $v$ is called a \emph{prefix} (resp. \emph{suffix}) of $u$.  A  factor (resp. prefix, resp. suffix) $v$ of $u$ is called  \textit{proper} if $v\neq u$. 

One may extend the order $<$ on $X$ to a partial order $\prec$ on $\langle X\rangle$  as follows:
\[
u \prec v \Longleftrightarrow u=rxs,\, v=ryt,\, \text{ with } x,y\in X;\, x<y;  \, r,s,t\in \langle X\rangle.
\]
It is easy to see that $\prec$ is well-founded and for all words $u, v, w_1,w_2, w_3\in \langle X\rangle$,
\[
u\prec v \Longrightarrow w_1uw_2\prec w_1v w_3.
\]
The \emph{pseudo-lexicographic order} on $\langle X\rangle$, denoted by $<_{\lex}$, is defined as follows:
\begin{eqnarray*}
	\label{definition-deglex}
	u <_{\lex} v\,\, \Longleftrightarrow \,\, \left\{
	\begin{array}{llll}
		v \text{ is a  proper prefix of } u,\quad \text{or}&& \\
		u\prec v.
	\end{array}\right.
\end{eqnarray*}
Note that it is a total order compatible with the concatenation of words from left but not from right.  For example if $x,y\in X$ with $x>y$, one has $x>_{\lex}x^2$ but $xy<_{\lex} x^2y$. 

\begin{lemma}\label{lex-order-1}
Let $\Gamma$ be a  well-ordered abelian monoid and  $l:\langle X\rangle \to \Gamma$ a monoid homomorphism such that $l(x)>0$ for each $x\in X$. Then for  words $u, v\in \langle X\rangle$, if $l(u)\leq l(v)$ and $u<_{\lex} v$ then $u\prec v$.
\end{lemma}

\begin{proof}
If $l(u) \leq l(v)$ then it follows necessarily that $v$ is not a proper prefix of $u$.
\end{proof}
	
\begin{proposition}\label{Proposition-character-Lyndon}
 The following statements are equivalent for a word $u\in \langle X\rangle_+$:
\begin{enumerate}
\item $u>_{\lex} wv$ for every factorization $u=vw$ with $v, w\neq 1$.
\item $u >_{\lex} w$ for every factorization $u=vw$ with $v,w \neq 1$.
\item $v >_{\lex} w$ for every factorization $u = vw$ with $v,w\neq1$.
\end{enumerate}
\end{proposition}

\begin{proof}
$(2) \Rightarrow (3)$ is clear, and by \cite[Lemma 2]{Kh} one has $(1) \Leftrightarrow (2)$. In the following we assume $(3)$ and to see (2). Let $u = vw$ with $v,w\neq1$. Then there is an integer $s\geq0$ and a word $w'\in \langle X\rangle$ such that $v^sw'=w$ and $v$ is not a prefix of $w'$. It follows that $v\geq_{\lex} v^{s+1} >_{\lex} w'$ and therefore $vw'>_{\lex}w'$. Now we can conclude that $u=v^svw'>_{\lex} v^sw'=w$ as desired.
\end{proof}

\begin{definition}
A word $u\in \langle X\rangle$ is called \textit{Lyndon}  if $u$ is nonempty and satisfies the equivalent conditions listed in Proposition \ref{Proposition-character-Lyndon}. The set of all Lyndon words on $X$ is denoted by $\L=\mathbb{L}(X)$.
\end{definition}
	
\begin{remark}
We follow \cite{Kh} for the definition of  Lyndon words, where they are called \emph{standard words} after Shirshov. In \cite{GF,Lo}, Lyndon words are defined in terms of the lexicographic order (denoted $<_{\rm lex}$ here for convenience). By definition, $u<_{\rm lex} v$ means that  either $u$ is a proper prefix of $v$ or  $u\prec v$; and a word $u\in \langle X\rangle$ is Lyndon if $u\neq1$ and $u<_{\rm lex}wv$ for every factorization $u=vw$ with $v,w\neq1$. So the results in \cite{GF,Lo} need to be changed accordingly in our context.
\end{remark}
	

For a word  $u$ of length $\geq2$, define   $u_R\in \langle X\rangle $ to be  the pseudo-lexicographically largest proper suffix of $u$ and define  $u_L\in \langle X\rangle $ by the decomposition $u=u_Lu_R$.  The pair of words
$$\sh(u):=(u_L,u_R)$$ is called  the {\em Shirshov factorization} of $u$. As an example,  $\sh(x_2^2x_1x_2x_1)=(x_2^2x_1,x_2x_1)$.

Lyndon words enjoy many excellent combinatorial properties, we collect below some of them for reader's interests. Due to the difference between  the lexicographic order and the pseudo-lexicographic order, the  statements that we present are adjusted accordingly from the given references.

\begin{proposition}\label{fact-Lyndon} \hfill
\begin{enumerate}
		\item[(L1)]  (\cite[Proposition 5.1.3]{Lo} $\&$ \cite[Lemma 4.3 (3)]{GF}) Let $u=w_1w_2,\, v=w_2w_3$ be Lyndon words. If $u >_{\lex} v$ (which holds in priori when $w_2\neq1$), then $w_1w_2w_3$ is a Lyndon word.
		\item[(L2)] (\cite[Proposition 5.1.3]{Lo}) Let $u$ be word of length $\geq2$. Then $u$ is a Lyndon word if and only if $u_L$ and $u_R$ are both Lyndon words and $u_L>_{\lex} u_R$.
		\item[(L3)](\cite[Proposition 5.1.4]{Lo}) Let $u,v$ be Lyndon words. Then $\sh(uv)=(u,v)$ if and only if either $u$ is a letter or  $u$ is of length $\geq 2$ with $u_R\leq_{\lex} v$.
		\item[(L4)](\cite[Theorem 5.1.5]{Lo}) Every word $u$ can be written uniquely as a  nondecreasing product of Lyndon words 
 (the \emph{Lyndon decomposition}): 
$$u=u_1u_2\cdots u_r,\quad u_i\in \mathbb{L}, \, \, u_1\leq_{\lex} u_2\leq_{\lex}\cdots\leq_{\lex} u_r.$$ 
The words $u_i\in\mathbb{L}$ appearing in the decomposition are called the \emph{Lyndon atoms} of $u$.
		\item[(L5)] (\cite[Lemma 4.5]{GF})  If a Lyndon word $v$ is a factor of a word  $u$ then it  is a factor of some Lyndon atom of $u$.   
           \item[(L6)] (\cite[Lemma 4]{Kh})
Let  $u_1>_{\lex}u_2>_{\lex} u'$ be nonempty words. If  $u_1u_2$ and $u'$ are Lyndon words, then $u_1u_2u' >_{\lex} u_1 u' >_{\lex} u'$ and $u_1u_2u' >_{\lex} u_2u' >_{\lex} u'$.
           \item[(L7)] (\cite[Lemma 5]{Kh}) 
Let $u=u_1\cdots u_m$ and $v=v_1\cdots v_n$ be two nonempty words in Lyndon decomposition. Then $u<_{\lex} v$ if and only if either $n< m$ and $u_i= v_i$ for $i\leq n$, or there is an integer $l \leq \min\{m, n\}$ such that $u_i=v_i$ for $i< l$ but $u_l <_{\lex } v_l$. \hfill $\Box$
	\end{enumerate}	
\end{proposition}
	
\begin{remark}
Due to Proposition \ref{fact-Lyndon} (L1, L2), one may obtain every Lyndon word by starting with $X$ and concatenating inductively each pair of Lyndon words $v,w$ with $v>_{\lex}w$.
\end{remark}

\begin{lemma}\label{lex-order-2}
Let $u,\, v \in \langle X \rangle_+$  with $v$ Lyndon. The following statements are equivalent:
	\begin{enumerate}
		\item $u\prec v^n$ for some $n\geq 1$.
		\item $u<_{\lex} v^n$ for all $n\geq 1$.
		\item The first Lyndon atom of $u$ is $<_{\lex} v$.
	\end{enumerate} 
\end{lemma}

\begin{proof}
Let $u= u_1\cdots u_r$ be the Lyndon decomposition of $u$. The implication (3) $\Rightarrow$ (2) is a direct consequence of Proposition \ref{fact-Lyndon} (L7). Note that  $v^n$ is not a prefix of $u$ for $n\gg 0$, the  implication (2) $\Rightarrow$ (1) follows readily. Next we show  (1) $\Rightarrow$ (3). Assume $u\prec v^n$ for some $n\geq 1$. By Proposition  \ref{fact-Lyndon} (L7), one has $u_1\leq_{\lex} v$. If $u_1=v$, then $v=u_1=\cdots =u_i <_{\lex} u_{i+1}$ for some $i<r$.
Note that $v^n$ is not a prefix of $u$, it follows that $u_{i+1}\leq_{\lex} v$ as well by Proposition \ref{fact-Lyndon} (L7), which is absurd.
\end{proof}



\section{Braided calculus on free algebras} 
In this section, we construct a system of PBW generators for algebras that satisfy certain conditions by the combinatorial method based on Lyndon words and braided commutators.

Throughout, let $X$ be a well-ordered alphabet. The free algebra on $X$ is denoted by $k\langle X\rangle$. It has $\langle X\rangle$ as a linear basis and its elements are called (noncommutative) polynomials on $X$. The subspace of $k\langle X\rangle$ spanned by  a subset $U \subseteq \langle X\rangle$ is denoted by $kU$. For each word $u$ on $X$,
\begin{itemize}
\item let $\langle X\rangle_u$ be the set of words having the same number of occurrences of each letter as $u$;
\item let $\langle X\rangle^{\prec u}$ (resp. $\langle X\rangle^{\preceq u}$) be the set of words that $\prec u$ (resp. $\preceq u$);
\item let $\langle X\rangle^{\vartriangleleft u}$ (reps. $\langle X\rangle^{\trianglelefteq u}$) be the set of words with Lyndon atoms all $<_{\lex} u$ (resp. $\leq_{\lex} u$).
\end{itemize}
We write $\langle X\rangle^{\prec u}_v:= \langle X\rangle^{\prec u} \cap \langle X\rangle_v$ for words $u,v\in \langle X\rangle$. The notation  $\langle X\rangle^{\vartriangleleft u}_v$, $\langle X\rangle^{\preceq u}_v$  and $\langle X\rangle^{\trianglelefteq u}_v$ are defined similarly. These subsets of $\langle X\rangle$ will be useful in the sequel.

\begin{definition}
	Let $\tau$ be a braiding on an algebra $A$. The \emph{$\tau$-commutator} of $x,y\in A$ is the element
	$$[x,y]_\tau:= (\mu_A-\mu_A\circ \tau) (x\otimes y).$$ When $\tau$ is the flipping map of $A\otimes A$, we simply write $[x,y]:=[x,y]_\tau=xy-yx$.
\end{definition}

Let $\tau: kX\otimes kX \to kX\otimes kX$ be a braiding on $kX$. It  extends uniquely to a braiding on $k\langle X\rangle$, which is also denoted by $\tau$, so that $(k\langle X\rangle, \tau)$ is a braided algebra.

The \emph{$\tau$-bracketing} on $k\langle X\rangle$ is the linear endomorphism $[-]_\tau: k\langle X\rangle \to k\langle X\rangle$ defined as follows.  First set $[1]_\tau =1$ and  $[x]_\tau:=x$ for  $x\in X$; and then for  words $u$ of length $\geq 2$, inductively  set
\begin{eqnarray*}
		[u]_\tau=
		\left\{
		\begin{array}{ll}
			~ [[u_L]_\tau, [u_R]_\tau]_\tau, & u \text{ is Lyndon},  \\
			~ [u_L]_\tau  [u_R]_\tau, &  u \text{ is not Lyndon};
		\end{array}
		\right.
\end{eqnarray*}
finally extend it by linearity to all polynomials on $X$.
Note that $[u_1u_2\cdots u_n]_\tau = [u_1]_\tau[u_2]_\tau\cdots [u_n]_\tau$ for any nondecreasing sequence of Lyndon words $u_1\leq_{\lex}u_2\leq_{\lex} \cdots \leq_{\lex} u_n$. 

\begin{remark}
We  simply write $[-] :=[-]_\tau$ when  $\tau$ is the flipping map of $k X\otimes kX$. It is well-known that $[\mathbb{L}]:= \{~ [u]~|~ u\in \mathbb{L}~\}$ is a basis of the free Lie algebra on $X$. Thereof, by the classical PBW theorem on the universal enveloping algebra of Lie algebras,  $[\langle X\rangle]$ is a basis of $k\langle X\rangle$.
\end{remark}

Let $\rho$ be an $X$-diagonal braiding on $kX$. Then for each pair of words $u, v\in \langle X\rangle$, there is a unique nonzero scalar $\rho_{u,v}\in k$ such that
$\rho(u\otimes v) = \rho_{u,v}  v\otimes u.$
These scalars bear the identities that $\rho_{1,u} =\rho_{u,1} =1$, $\rho_{u,vw} = \rho_{u,v} ~\rho_{u,w} $ and $\rho_{uv,w} =\rho_{u,w}~ \rho_{v,w}$ for any words $u,v,w\in \langle X\rangle$. So
$$
\rho(f\otimes g) = \rho_{u,v} g\otimes f
$$
for any words $u,v \in \langle X\rangle$ and any polynomials $f\in k\langle X\rangle_u$, $g\in k\langle X\rangle_v$. Moreover, the $\rho$-commutator satisfies two ``braided'' derivation equations and a ``braided'' Jacobi identity  as follows:
\begin{eqnarray*}
[f,gh]_\rho &=& [f,g]_\rho h  + \rho_{u,v} g[f,h]_\rho,\\ 
~ [fg,h]_\rho &=& f[g,h]_\rho + \rho_{v,w} [f,h]_\rho g,\\ 
~  [[f,g]_\rho, h]_\rho&=&  [f,[g, h]_\rho]_\rho - \rho_{u,v} g[f,h]_\rho +\rho_{v,w}[f,h]_\rho g,
\end{eqnarray*}
for any words $u,v,w \in \langle X\rangle$ and any polynomials $f\in k\langle X\rangle_u, g\in k\langle X\rangle_v, h\in k\langle X\rangle_w$.

\begin{lemma}\label{bracketing-diagonal}
	Let $\rho$ be an $X$-diagonal braiding on $kX$. Then for each word $u\in \langle X\rangle$,  $$[u]_\rho\in u +k\langle X\rangle_u^{\prec u}.$$
Moreover, the set $[\langle X\rangle]_\rho := \{~ [u]_\rho~|~ u\in \langle X\rangle~\}$ is a basis of $k\langle X\rangle$.
\end{lemma}
\begin{proof}
The first statement is by  \cite[Lemma 5]{Kh} and Lemma \ref{lex-order-1}; and the second one is by \cite[Theorem 1]{Kh}. One may obtain them directly by an easy induction on words with respect to $\prec$.
\end{proof}

\begin{lemma}\label{bracketing-diagonal-more}
Let $\rho$ be an $X$-diagonal braiding on $kX$. Then 
\begin{enumerate}
	\item $[k\langle X\rangle^{\trianglelefteq u}]_\rho$ and $[k\langle X\rangle^{\vartriangleleft u}]_\rho$ are subalgebras of $k\langle X\rangle$ for each word $u\in \langle X\rangle$.
	\item $[[u]_\rho, [v]_\rho]_\rho\in  [k\langle X\rangle_{uv}^{\trianglelefteq uv}]_\rho$ for each pair of  Lyndon words $u>_{\lex}v$.
	\item $[[u]_\rho, [k\langle X\rangle^{\trianglelefteq v}]_\rho]_\rho \subseteq [k\langle X\rangle^{\trianglelefteq uv}]_\rho$ for each pair of  Lyndon words $u>_{\lex}v$.
	\item $[[u]_\rho, [k\langle X\rangle^{\vartriangleleft v}]_\rho]_\rho \subseteq [k\langle X\rangle^{\vartriangleleft uv}]_\rho$ for each pair of  Lyndon words $u>_{\lex}v$.
\end{enumerate}
\end{lemma}

\begin{proof}
Part (1) is \cite[Lemma 7]{Kh}; Part  (2) is by  \cite[Lemma 6, Lemma 7]{Kh};  Next we  show Part (3). Let $w_1, \ldots, w_r$ be a finite sequence of Lyndon words that $\leq_{\lex} v$. One has
	\[
	[[u]_\rho, [w_1]_\rho\cdots [w_r]_\rho]_\rho \in \sum_{i=1}^{r} k\cdot \Big( [w_1]_\rho\cdots [w_{i-1}]_\rho \cdot  [[u]_\rho, [w_i]_\rho ]_\rho \cdot [w_{i+1}]_\rho \cdots [w_r]_\rho \Big)
	\]
	by the braided derivation equation. Note that $w_i\leq_{\lex} v <_{\lex} uv$. In addition,  $$[[u]_\rho, [w_i]_\rho]_\rho \in [k\langle X\rangle^{\trianglelefteq uw_i}]_\rho \subseteq [k\langle X\rangle^{\trianglelefteq uv}]_\rho$$ by Part (2) and the observation that $uw_i\leq_{\lex} uv$. The result then follows immediately by Part (1). One may obtain Part (4) in a similar way as that of Part (3).
\end{proof}



In the remaining of this section, we fix a nontrivial well-ordered abelian monoid $(\Gamma,<)$, and assume that  the free algebra $k\langle X\rangle=\bigoplus_{\gamma\in \Gamma} k\langle X\rangle_\gamma$ is a $\Gamma$-graded algebra with each letter homogeneous of positive degree.  The degree of a homogeneous polynomial $f$ is denoted by $\deg(f)$.  In addition, we write $\langle X\rangle_\gamma$ for the set of words of degree $\gamma$. Note that $k\langle X\rangle_+ =\bigoplus_{\gamma\neq 0} k\langle X\rangle_\gamma$.

The \emph{graded lex order} on $\langle X\rangle$, denoted by $<_{\glex}$, is defined as follows. For words $u, \, v$, 
\begin{eqnarray*}
	\label{definition-deglex}
	u <_{\rm \glex} v\,\, \Longleftrightarrow \,\, \left\{
	\begin{array}{llll}
		\deg(u)<\deg(v), \quad \text{or} && \\
		\deg(u) =\deg(v)\, \text{ but } u\prec v.
	\end{array}\right.
\end{eqnarray*}
Clearly, it is a well order that is compatible with concatenation of words from both sides. We write $\lw(f)$ for the largest word that occurs in a nonzero polynomial $f$   with respect to  $<_{\glex}$.

Let $I$ be an ideal of $k\langle X\rangle$. A word on $X$ is called \emph{$I$-reducible}  if it is the leading word of some nonzero polynomial in $I$. A word that is not $I$-reducible is called \emph{$I$-irreducible}. A  word is called \emph{$I$-restricted} if its Lyndon decomposition  is of the form $w_1^{r_1}\cdots w_m^{r_m}$ with $w_1<_{\lex} \cdots <_{\lex} w_m \in \mathbb{L}$ and $w_{1}^{r_1}, \cdots, w_m^{r_m}$ all $I$-irreducible. In addition, an \emph{obstruction} of $I$ is an $I$-reducible word of which the proper factors are all $I$-irreducible. In the sequel, we will employ the following notations:
\begin{itemize}
	\item let $\B_I$ be the set of words with $I$-irreducible Lyndon atoms;
	\item let $\C_I$ be the set of $I$-restricted words;
	\item let $\D_I$ be the set of all $I$-irreducible words;
	\item let $\N_I$ be the set of $I$-irreducible Lyndon words;
	\item let $\OO_I$ be the set of obstructions of $I$.
\end{itemize}
Clearly, $\B_I \supseteq \C_I \supseteq \D_I$. By Proposition \ref{fact-Lyndon} (L4, L5), it is not hard to see that $\B_I=\D_I$ (resp. $\C_I=\D_I$) if and only if $\mathcal{O}_I$ consists of Lyndon words (resp. powers of Lyndon words).



\begin{lemma}\label{basis-quotient-algebra}
Let $I$ be an ideal of $k\langle X\rangle$ and $A:= k\langle X\rangle/I$. Let $\rho$ be an $X$-diagonal braiding on $kX$. Then for every index $\gamma\in \Gamma$, the set $\{~[w]_\rho+I~| ~ w\in \D_I, \, \deg(w)\leq \gamma~ \}$  is a basis of the subspace  $F_\gamma (A):=(k\langle X\rangle_{\leq \gamma} +I)/ I$ of   $A$. Moreover,  the set $\{~[w]_\rho+I~| ~ w\in \D_I~ \}$ is  a basis of $A$. In particular, $A$ is generated by the set $\{~[w]_\rho+I~| ~ w\in \N_I~ \}$.
\end{lemma}

\begin{proof}
We prove the first statement, the others are  clear.
Firstly, we show that these residue classes are linearly independent. Suppose not, then there exists a polynomial  $f=\sum_{i=1}^r \lambda_i [u_i]_\rho \in I$, with $u_i$ pairwise distinct $I$-irreducible words of degree $\leq \gamma$ and $\lambda_i\in k\backslash\{0\}$. We may assume that $u_1> _{\glex} \cdots >_{\glex} u_r$. By Lemma \ref{bracketing-diagonal}, the leading word of $f$ is $u_1$, which is impossible.
	
Now we show these residue classes span $F_\gamma A$.  It suffices to show they span the residue classes of all words of degree $\leq \gamma$. We show this by induction on words with respect to the graded lex  order. Clearly, it is true for the empty word, which is the smallest element with respect to $<_{\glex}$. Let $u$ be a nonempty word of degree $\leq \gamma$. If $u$ is $I$-reducible, then $u+I= f_u+I $ with $f_u$ a linear combination of words that $<_{\glex} u$; and if $u$ is $I$-irreducible, then $u+I= ([u]_\rho+I) + (f_u +I)$ with  $f_u:=u-[u]_\rho$, which  by Lemma \ref{bracketing-diagonal} is also a  linear combination of words that $<_{\glex} u$. So by the induction hypothesis,  $u+I$ is a linear combination  of  $\{~[w]_\rho+I~| ~ w\in \D_I, \, \deg(w)\leq \gamma~ \}$.
\end{proof}

\begin{definition}\label{definition-height}
Let $I$ be an ideal of $k\langle X\rangle$. For a Lyndon  word $u\in \mathbb{L}$, the \emph{height} of $u$ is defined by
	\[
	h_I(u):= \min\{ ~ n\geq 1 ~ | ~  u^n  ~ \text{is $I$-reducible} ~ \}.
	\]
By convention,
$h_I(u) :=\infty$ if there is no integer $n$ such that $u^n$ is $I$-reducible.  Note that $h_I(u)=1$ if  $u$ is $I$-reducible; and $h_I(u)\geq 2$ for every $I$-irreducible Lyndon word $u$. 
\end{definition}

\begin{proposition}\label{PBW-generator}
Let $I$ be an ideal of $k\langle X\rangle$ with $\C_I=\D_I$. Let $A:=k\langle X\rangle/ I$. Let $\rho$ be an $X$-diagonal braiding on $kX$ such that for any Lyndon word $u\in \mathbb{L}$ with finite height $n:=h_I(u)$, 
\begin{eqnarray*}\label{soft}
[u]^{n}_\rho \in [k\langle X\rangle^{\vartriangleleft u}_{\deg(u^{n})}]_\rho +k\langle X\rangle_{<\deg(u^{n})}+I.
\end{eqnarray*} 
Then $(\{[u]_\rho+I\}_{u\in \N_I}, h_I, <)$ is a system of PBW generators of   $A$ for each  total order  $<$ on $\N_I$.
\end{proposition}

Before we prove the above proposition, we need some preparations.

Let $\mathcal{P}$ be the set of all finite sequences of Lyndon words on $X$. Define partial orders $<_L$ and $<_R$  on $\mathcal{P}$ as follows. For $V=(v_1,\ldots, v_r)$ and $W= (w_1,\ldots, w_s)$ in $\mathcal{P}$, we write  $V<_L W$ (resp. $V<_R W) $ if and only if   $\sum_{i=1}^r \deg(v_i) = \sum_{j=1}^s\deg(w_j)$ but there is an integer $l\leq \min\{r, s\}$ such that $v_i =w_i$ for $i< l$ and $v_l<_{\lex} w_l$ (resp. $v_{r+1-i} =w_{s+1-i}$ for $i< l$ and $v_{r+1-l}<_{\lex} w_{s+1-l}$).  Clearly, $<_L$ and $<_R$ are both compatible with left and right concatenation of finite sequences. In addition,  they  both satisfy the descending chain condition, because the restriction of $<_{\lex}$ on the set of words of degree less than a fixed index is a well order (by a similar argument of \cite[Lemma 1.1]{LZ}). For an ideal $I$ of $k\langle X\rangle$, let $$\mathcal{P}_I:= \{~(w_1,\ldots, w_s) \in \mathcal{P}~|~ w_1\leq_{\lex} \cdots \leq_{\lex} w_s, \,\, w_1\cdots w_s\in \D_I ~ \}.$$ Note that the canonical map $\mathcal{P}_I \to \D_I$ given by $(w_1,\ldots, w_s) \mapsto w_1\cdots w_s$ is bijective.

\begin{lemma}\label{rearragement}
Assume the notations and conditions of Proposition \ref{PBW-generator}. Let  $(w_1,\ldots, w_s)\in \mathcal{P}$ be a finite sequence of Lyndon words on $X$ and let  $\gamma:= \sum_{i=1}^s\deg(w_i)$. Let  $i_1,\ldots, i_s$ be a permutation of $1,\ldots, s$ such that $w_{i_1} \leq_{\lex}  \cdots \leq_{\lex} w_{i_s}$.  
If $w_{i_1}\cdots w_{i_s} \in \D_I$, that is $(w_{i_1},\ldots, w_{i_s}) \in \mathcal{P}_I$, then
\begin{eqnarray*}
	[w_1]_\rho\cdots [w_s]_\rho  &\in & (\prod_{1\leq p<q\leq s \atop w_{i_q}<_{\lex}w_{i_p}} \rho_{w_{i_p},w_{i_q}}) \cdot [w_{i_1}]_\rho\cdots [w_{i_s}]_\rho  \\ && + \sum_{(u_1,\ldots, u_t) ~ \in ~ \mathcal{P}_I \atop  (u_1,\ldots, u_t) ~<_R ~ (w_{i_1},\ldots, w_{i_s})} k\cdot [u_1]_\rho \cdots [u_t]_\rho + k\langle X\rangle_{<\gamma} + I;
\end{eqnarray*}
otherwise, 	if $w_{i_1}\cdots w_{i_s} \not\in \D_I$, that is $(w_{i_1},\ldots, w_{i_s}) \not\in \mathcal{P}_I$,  then
	\begin{eqnarray*}
		[w_1]_\rho\cdots [w_s]_\rho  &\in &  \sum_{(u_1,\ldots, u_t) ~ \in ~ \mathcal{P}_I \atop  (u_1,\ldots, u_t) ~<_R ~ (w_{i_1},\ldots, w_{i_s})} k\cdot [u_1]_\rho \cdots [u_t]_\rho + k\langle X\rangle_{<\gamma} + I.
	\end{eqnarray*}
\end{lemma}

\begin{proof}
	We proceed  by induction on $(w_1,\ldots, w_s)$ with respect to the well order $<_L$. First assume $(w_1,\cdots, w_s)$ is  nondecreasing.  If $w_1\cdots w_s\in \D_I$, there is nothing to prove.  Otherwise, assume $w_1\cdots w_s\not \in\D_I$. 
	Write $w_1\cdots w_s= v_1^{p_1}\cdots v_r^{p_r}$ with  $p_i\geq 1$, $v_i:=w_{p_1+\cdots +p_{i-1}+1}$ and $v_1<_{\lex} \cdots <_{\lex} v_r$. By the assumption, one may fix an integer $i$ such that $p_i\geq h_I(v_i)$, and then one has that 
	\begin{eqnarray*}
		[v_1]_\rho^{p_1}\cdots [v_r]_\rho^{p_r}
		&\in&  \sum_{(b_1,\ldots, b_t)}k\cdot [v_1]^{p_1}_\rho\cdots[v_{i-1}]^{p_{i-1}}_\rho[b_1]_\rho\cdots [b_t]_\rho [v_i]_\rho^{p_i - h_I(v_i)} [v_{i+1}]^{p_{i+1}}_\rho\cdots [v_r]_\rho^{p_r} \\ && +k\langle X\rangle_{<\gamma} + I,
	\end{eqnarray*}
where $(b_1,\ldots, b_t)$ runs over nondecreasing finite sequences of Lyndon words on $X$ such that $b_1,\ldots, b_t<_{\lex} v_i\leq_{\lex} u$ and  $\sum_{j=1}^t\ \deg(b_j) = \deg(v_i^{h_I(v_i)})$. Clearly, the finite sequences 
	\[
	( {\small \overbrace{v_{1},\ldots, v_{1}}^{p_{1}}}, \ldots, {\small \overbrace{v_{i-1},\ldots, v_{i-1}}^{p_{i-1}}}, b_1,\ldots, b_t,  {\small \overbrace{v_{i},\ldots, v_{i}}^{p_i-h_I(v_i)}},  {\small \overbrace{v_{i+1},\ldots, v_{i+1}}^{p_{i+1}}},\ldots, {\small \overbrace{v_{r},\ldots, v_{r}}^{p_{r}}},) ~ \in ~ \mathcal{P}
	\]
	are smaller than $(w_1,\ldots, w_s)$ with respect to $<_L$, and the nondecreasing rearrangement of them  are smaller than $(w_{i_1},\ldots, w_{i_s})$ with respect to $<_R$. Therefore, if $(w_1,\ldots, w_s)$ is a minimal element of $\mathcal{P}$ (with $w_1\cdots w_s\not\in \D_I$) with respect to $<_L$, which is necessarily nondecreasing, then  $$[w_1]_\rho\cdots [w_s]_\rho \in k\langle X\rangle_{<\gamma} + I,$$ and hence the initial step of the induction argument hold. If $(w_1,\ldots, w_s)$ is not minimal, then by the induction hypothesis and the above observation, the desired formula hold.
	
	Next we consider the another case, i.e. the finite sequence $(w_1,\ldots, w_s)$ is not nondecreasing. So assume $w_{i+1}<_{\lex} w_{i}$ for some $i$.  Then by Lemma \ref{bracketing-diagonal-more} (2),  
	\begin{eqnarray*}
		[w_1]_\rho\cdots [w_s]_\rho
		&\in&  \rho_{w_i, w_{i+1}} [w_1]_\rho\cdots [w_{i-1}]_\rho [w_{i+1}]_\rho[w_i]_\rho [w_{i+2}]_\rho \cdots [w_s]_\rho \\
		&& + \sum_{(c_1,\cdots, c_t)} k \cdot [w_1]_\rho\cdots [w_{i-1}]_\rho [c_1]_\rho\cdots [c_t]_\rho[w_{i+2}]_\rho \cdots [w_s]_\rho,
	\end{eqnarray*}
	where $(c_1,\ldots, c_t)$ runs over finite sequences of Lyndon words on $X$  such  that $c_1,\ldots, c_t\leq_{\lex} w_iw_{i+1} <_{\lex} w_i \leq_{\lex} u $ and  $\sum_{j=1}^t \deg(c_j) = \deg (w_iw_{i+1})$.  Clearly, the finite sequences
	\begin{eqnarray*}
		(w_1,\ldots, w_{i-1}, w_{i+1}, w_i, w_{i+2},\ldots, w_s),~~  	(w_1,\ldots, w_{i-1}, c_1, \ldots, c_t, w_{i+2},\ldots, w_s) ~ \in ~ \mathcal{P}
	\end{eqnarray*} 
	are both smaller than $(w_1,\ldots, w_s)$ with respect to $<_L$. In addition, $(w_{i_1},\ldots, w_{i_s})$ is also the nondecreasing rearrangement of $(w_1,\ldots, w_{i-1}, w_{i+1}, w_i, w_{i+2},\ldots, w_s)$, and it is bigger than that  of other ones with respect to $<_R$. The desired formula  then follows immediately by the induction hypothesis.
\end{proof}

\begin{proof}[Proof of Proposition \ref{PBW-generator}]
To simplify the notations, we write $\deg(W) := \sum_{i=1}^s \deg(w_i)$ and $[W]_\rho= [w_1]_\rho\cdots [w_s]_\rho$ for any sequence $W=(w_1,\ldots, w_s) \in \mathcal{P}$. Let $<$ be a total order on $\N_I$. Let 
$$\mathcal{P}_{<} := \{~ ( {\small \overbrace{v_{1},\ldots, v_{1}}^{p_{1}}}, \ldots,  {\small \overbrace{v_{l},\ldots, v_{l}}^{p_{l}}})~ |~ v_1< \cdots < v_l\in \N_I, \, p_i <h_I(v_i), \, i=1,\ldots, l ~ \}.$$
There is a natural bijection $\Pi: \mathcal{P}_{<} \to \mathcal{P}_I$ given by rearranging of finite sequences into pseudo-lexicographically nondecreasing form. It is asked to show  $\mathcal{X}= \{~ [W]_\rho+I ~ \}_{W \in \mathcal{P}_{<}}$ is a basis of $A=k\langle X\rangle/I$.

First we show that $\mathcal{X}$ is linearly independent. Otherwise, there is a positive integer $p\geq1$, distinct finite sequences $W_1, \ldots, W_p \in \mathcal{P}_{<}$ and nonzero scalars $\lambda_1,\ldots, \lambda_p\in k$ such that 
\[ 
F:= \sum_{i=1}^p\lambda_i\cdot [W_i]_\rho ~ \in ~  I.
\]
Let $\gamma = \max \{~ \deg(W_1), \ldots, \deg(W_p) ~\}.$
Among those sequences $\Pi(W_i)\in \mathcal{P}_{I}$ with  $\deg(W_i) =\gamma$, there is a biggest one with respect to $<_R$, which we denote by $V$. Then one has
\begin{eqnarray*}
	F & \in&  k^\times \cdot [V]_\rho  + \sum_{U~ \in ~ \mathcal{P}_I,\, U ~<_R ~ V} k \cdot [U]_\rho + k\langle X\rangle_{<\gamma} + I \\
	 &= &  k^\times \cdot [V]_\rho  + \sum_{U~ \in ~ \mathcal{P}_I,\, U ~<_R ~ V} k \cdot [U]_\rho +  \sum_{U ~ \in ~ \mathcal{P}_I,\, \deg(U) <\gamma } k \cdot [U]_\rho  + I
\end{eqnarray*}
Here, the belonging is by  Lemma \ref{rearragement} and the equality is by  Lemma \ref{basis-quotient-algebra}. It follows that
\begin{eqnarray*}
[V]_\rho &\in & \sum_{U ~ \in ~ \mathcal{P}_I,\, U ~<_R ~ V} k \cdot [U]_\rho +  \sum_{U ~ \in ~ \mathcal{P}_I,\, \deg(U) <\gamma } k \cdot [U]_\rho + I,
\end{eqnarray*}
which contradicts that $\{~[W]_\rho+I ~\}_{W\in \mathcal{P}_I ~}$ is a basis of $k\langle X\rangle/ I$ by Lemma  \ref{basis-quotient-algebra}.

In the following, we show  $\mathcal{X}$ spans  $A$. By Lemma \ref{basis-quotient-algebra}, it suffices to see  $[V]_\rho +I \in k\mathcal{X}$ for any sequence $ V \in \mathcal{P}_I$. We do it by induction on the index 
$\deg(V)$. If $\deg(V) =0$ then $V$ is the empty sequence, and therefore $[V]_\rho +I= 1+I \in k\mathcal{X}$. Let $\gamma$ be an arbitrary positive index. Suppose that $[V]_\rho +I \in k\mathcal{X}$ for every $V \in \mathcal{P}_I$ with $\deg(V)< \gamma$. 
By Lemma \ref{basis-quotient-algebra}, one has 
\[
k\langle X\rangle_{<\gamma} +I = \sum_{U \in \mathcal{P}_{I},\, \deg(U)<\gamma} k\cdot [U]_\rho +I \subseteq \sum_{W \in \mathcal{P}_{<}} k\cdot [W]_\rho +I.
\]
For $V\in \mathcal{P}_I$  with $\deg(V) = \gamma$, Lemma \ref{rearragement} tells us that
\begin{eqnarray*}
[V]_\rho &\in& k^\times \cdot [( \Pi^{-1}(V)]_\rho + \sum_{U\in \mathcal{P}_I,\, U<_RV} k\cdot [U]_\rho + k\langle X\rangle_{<\gamma} +I \\
& \subseteq &\sum_{U\in \mathcal{P}_I,\, U<_RV} k\cdot [U]_\rho + \sum_{W \in \mathcal{P}_{<}} k\cdot [W]_\rho +I.
\end{eqnarray*}
By induction on $V$ with respect to $<_R$, one readily sees that $[V]_\rho\in \sum_{W \in \mathcal{P}_{<}} k\cdot [W]_\rho +I$ and hence $[V]_\rho+I \in k \mathcal{X}$ for any finite sequence $V\in \mathcal{P}_I$ with $\deg(V)=\gamma$.
\end{proof}

In general, it is not easy to check the displayed formula of Proposition \ref{PBW-generator} for all Lyndon words of finite height. The next result makes this problem slightly easier to handle, particularly in the case that $X$ and $\N_I$ are both finite. Note that $I$-reducible Lyndon words are of height $1$. 

\begin{proposition}\label{Soft-criterion}
	Let $I$ be an ideal of $k\langle X\rangle$. Let $\rho$ be an $X$-diagonal braiding on $kX$ such that $\rho(k\langle X\rangle\otimes I + I\otimes k\langle X\rangle) \subseteq k\langle X\rangle\otimes I + I\otimes k\langle X\rangle$. Assume that $$[u]_\rho \in [k\langle X\rangle^{\vartriangleleft u}_{\deg(u)}]_\rho +k\langle X\rangle_{<\deg(u)}+I$$ for every  $I$-reducible Lyndon word  $u$ that is either of length $1$ or of length $\geq 2$ with $u_L, u_R\in N_I$. Then the displayed formula holds for every  $I$-reducible Lyndon word  $u$.
\end{proposition}

\begin{proof}
	We show the result by induction on  $I$-reducible Lyndon words $u$ with respect to $<_{\glex}$. If $u$ is the smallest one among all $I$-reducible Lyndon words, then $u$ is necessarily either of length $1$ or of length $\geq 2$ with $u_L, u_R\in N_I$, and so by he assumption the displayed formula holds.
	
	Now assume $u$ is  not the smallest one. By the assumption, we may assume  $u$ have length $\geq 2$, and either  $u_L$ or $u_R$ is $I$-reducible. If $u_L$ is $I$-reducible then by induction,
	\begin{eqnarray*}
		[u]_\rho& =& [u_L]_\rho[u_R]_\rho - \rho_{u_L,u_R} [u_R]_\rho[u_L]_\rho \\
		&\in&  [k\langle X\rangle_{\deg(u_L)}^{\vartriangleleft u_L}]_\rho\cdot  [u_R]_\rho + [u_R]_\rho  \cdot k\langle X\rangle_{\deg(u_L)}^{\vartriangleleft u_L}  + k\langle X\rangle_{<\deg(u)} +I \\
		&\subseteq&   k\langle X\rangle_{\deg(u)}^{\vartriangleleft u} + k\langle X\rangle_{<\deg(u)} +I.
	\end{eqnarray*}
	Here, the inclusion used two observations: one is  $\langle X\rangle_{\deg(u_L)}^{\vartriangleleft u_L} \subseteq \langle X\rangle_{\deg(u_L)}^{\vartriangleleft u}$ by Lemma \ref{lex-order-1}; the other one is $u_R <_{\lex} u$ by Proposition \ref{Proposition-character-Lyndon}.
	If $u_R$ is $I$-reducible then by induction,
	\begin{eqnarray*}
		[u]_\rho &=& [[u_L]_\rho, [u_R]_\rho]_\rho \\
		&\in&  [[u_L]_\rho, [k\langle X\rangle_{\deg(u_R)}^{<u_R}]_\rho] + k\langle X\rangle_{<\deg(u)} + [[u_L]_\rho, I]_\rho  \\
		&\subseteq&   [k\langle X\rangle_{\deg(u)}^{<u}]_\rho + k\langle X\rangle_{<\deg(u)} +I.
	\end{eqnarray*}
	Here  the inclusion is by Lemma \ref{bracketing-diagonal-more} (3) and the observation that $[k\langle X\rangle, I]_\rho\subseteq I$.
\end{proof}

\section{Bounded comultiplications on free algebras}
\label{section-PBW-generator}

We continue to use the notations and conventions employed in the previous two sections. So $X$ stands for a well-ordered alphabet, $(\Gamma,<)$ is a nontrivial well-ordered abelian monoid and the free algebra $k\langle X\rangle$ is $\Gamma$-graded with each letter homogeneous of positive degree. 

In this section, we will introduce a class of algebra homomorphisms $\Delta:k\langle X\rangle \to  k\langle X\rangle \otimes^\tau  k\langle X\rangle$ for each braiding $\tau$ on $kX$ and study their properties by the theory of Lyndon words and of braided bracketing on polynomials. An interesting observation is that if an ideal $I$ of $k\langle X\rangle$ satisfies that $\Delta(I) \subseteq  k\langle X\rangle \otimes  I + I \otimes  k\langle X\rangle$ for certain such pair $(\tau,\Delta)$, then it necessarily  meets some desirable conditions, including those assumed in Proposition \ref{PBW-generator}.  This observation is the key to explore the structure of connected graded braided bialgebras in the following section.


The following notion is a generalization of \cite[Definition 2.1]{ZSL3}, where the term ``triangular'' was used and only the flipping map on $kX$ was taken into account. 

\begin{definition}\label{definition-bounded-comultiplication}
Let $\tau$ be a braiding on $kX$. A  \emph{$\tau$-comultiplication} on $k\langle X\rangle$ is a  homomorphism  $\Delta: k\langle X\rangle  \to k\langle X\rangle \otimes^\tau k\langle X\rangle$ of  algebras. It is called  \emph{ left bounded}  if  for every letter $x\in X$, 
\[
\Delta(x)  ~ \in ~  1\otimes x + x\otimes 1 +   \big(k\langle X_{<x} \rangle_+ \otimes k\langle X \rangle_+\big)_{\deg(x)} +  \big(k\langle X\rangle \otimes k\langle X\rangle \big)_{<\deg(x)}, 
\]
where $X_{<x}:=\{~x'\in X~| ~ x'<x~\}$. The \emph{right bounded $\tau$-comultiplications} on $k\langle X\rangle$ is defined similarly. A $\tau$-comultiplication on $k\langle X\rangle$  is called \emph{bounded} if it is both left and right bounded. 
\end{definition}

\begin{remark}\label{remark-standard-comul}
The \emph{standard $\tau$-comultiplication} on $k\langle X\rangle$ is  the algebra homomorphism
	$$
	\Delta_s: k\langle X\rangle  \to k\langle X\rangle \otimes^\tau k\langle X\rangle, \quad X \ni x\mapsto 1\otimes x +x\otimes 1.
	$$
It  is   bounded  by definition. Note that $\Delta_s$ is coassociative in the sense that $(\Delta_s\otimes \id) \circ \Delta_s = (\id \otimes \Delta_s) \circ \Delta_s$, which does not necessarily hold for general bounded $\tau$-comultiplications. 
\end{remark}

The following result is well-known in the special case  that $k\langle X\rangle$ is length graded, $\tau$ is $X$-diagonal and $\Delta=\Delta_s$, see \cite{Kh}. The case that $\tau$ is the flipping map is dealt in \cite{ZSL3}.

\begin{proposition}\label{comultiplication-diagonal}
Let $\tau$ be an $X$-diagonal braiding on $kX$. Let $\Delta$ be a left (resp. right) bounded  $\tau$-comultiplication on $k\langle X\rangle$. Then for each  $u\in \mathbb{L}(X)$, each   $w\in \langle X\rangle^{\vartriangleleft u}$ and each  $n\geq 0$, 
	\begin{eqnarray*}
		\Delta([wu^n]_{\tau^{-1}})  &=& \sum_{i=0}^n \binom{n}{i}_{\tau_{u,u}} \tau_{w,u}^i~  [u^i]_{\tau^{-1}} \otimes [wu^{n-i}]_{\tau^{-1}}  + (1-\delta_{1,w}) [wu^n]_{\tau^{-1}}  \otimes 1 +f +g\\
		\Big( \text{resp. }\quad \Delta([wu^n]_\tau) &=& \sum_{i=0}^n \binom{n}{i}_{\tau_{u,u}} ~  [wu^i]_\tau \otimes [u^{n-i}]_\tau + (1-\delta_{1,w})  1\otimes[wu^n]_\tau  +f +g ~\Big)
	\end{eqnarray*}
for some  $f\in  \big(\sum_{i=0}^n [k\langle X \rangle_+^{\vartriangleleft u}u^i]_{\tau^{-1}} \otimes k\langle X\rangle_+\big)_{\gamma}$ (resp. $f\in \big( k\langle X \rangle_+\otimes \sum_{i=0}^n [k\langle X \rangle_+^{\vartriangleleft u}u^i]_\tau\big)_{\gamma} $)  and $g\in \big(k\langle X\rangle \otimes k\langle X\rangle \big)_{<\gamma}$, where $\gamma:=\deg(wu^n)$ and  $\delta_{1,w}$ is the Kronecker symbol.
\end{proposition}
\begin{proof}
We prove the left version in full detail. The right version can be done similarly. We show the result by induction on the number $t$ of Lyndon atoms of $w$. To simplify the notation, let $\rho=\tau^{-1}$. 

First we prove the case $t=0$ by induction on $n$. For $n=1$, we do it by induction on $l=|u|$. For $l=1$, $u$ is a letter and the desired formula holds by assumption. Assume $l>1$.  Let 
\begin{eqnarray*}
\quad \Delta([u_L]_\rho) &= & 1\otimes [u_L]_\rho + [u_L]_\rho\otimes 1 +\sum f_i'\otimes f_i'' + h_L\\
\Delta([u_R]_\rho) &= &1\otimes [u_R]_\rho +[u_R]_\rho\otimes 1 + \sum g_j'\otimes g_j'' +h_R,
\end{eqnarray*}
where $f_i', f_i'', g_j', g_j''$ are homogeneous in each variable with $f_i'\otimes f_i''\in (k\langle X\rangle_+^{\vartriangleleft u_L}\otimes k\langle X\rangle_+)_{\deg(u_L)}$ and $g_j'\otimes g_j''\in k\langle X\rangle_+^{\vartriangleleft u_R} \otimes k\langle X\rangle_+)_{\deg(u_R)}$, $h_L\in (k\langle X\rangle_+\otimes k\langle X\rangle_+)_{<\deg(u_L)}$ and $h_R\in (k\langle X\rangle_+\otimes k\langle X\rangle_+)_{<\deg(u_R)}$. Note that $[k\langle X\rangle_+^{\vartriangleleft u}]_\rho$ is a subalgebra of $k\langle X\rangle$ and $[u_R]_\rho, f_i', g_j'\in [k\langle X\rangle_+^{\vartriangleleft u}]_\rho$, one obtains that 
\begin{eqnarray*}
\Delta([u]_\rho) &=&  \Delta([u_L]_\rho) \Delta([u_R]_\rho)-\rho_{u_L, u_R}\Delta([u_R]_\rho)\Delta([u_L]_\rho) \\
 &\in & 1\otimes[u]_\rho+ [u]_\rho\otimes 1+ \sum [[u_L]_\rho, g_j']_{\rho} \otimes g_j'' \\
 && + ([k\langle X\rangle_+^{\vartriangleleft u}]_\rho\otimes k\langle X\rangle_+)_{\deg(u)} + (k\langle X\rangle_+\otimes k\langle X\rangle_+)_{<\deg(u)}\\
 &\subseteq & 1\otimes[u]_\rho + [u]_\rho\otimes 1 +  ([k\langle X\rangle_+^{\vartriangleleft u}]_\rho\otimes k\langle X\rangle_+)_{\deg(u)} + (k\langle X\rangle_+\otimes k\langle X\rangle_+)_{<\deg(u)}.
\end{eqnarray*}
Here, we used $[[u_L]_\rho, g_j']_{\rho}\in [[u_L]_\rho, [k\langle X\rangle^{\vartriangleleft u_R}]_\rho] \subseteq [k\langle X\rangle^{\vartriangleleft u}]_\rho$ by Lemma \ref{bracketing-diagonal-more} (4).
The induction step for $n>1$ follows readily from the decomposition $\Delta([u^n]_\rho) = \Delta([u]_\rho)\Delta([u^{n-1}]_\rho)$ and the fact that  $[u]_\rho [k\langle X\rangle^{\vartriangleleft u}]_\rho \subseteq [k\langle X\rangle^{\vartriangleleft u}]_\rho + [k\langle X\rangle^{\vartriangleleft u}]_\rho[u]_\rho$, which is again by Lemma \ref{bracketing-diagonal-more} (4).

Now assume $t>0$.  Write $w=w_1w_2$ with $w_1$ the first Lyndon atom of $w$. The desired formula then follows readily from the decomposition $\Delta([wu^n]_\rho) = \Delta ([w_1]_\rho) \Delta([w_2u^n]_\rho) $.
\end{proof}

The next result is the keystone of this paper. For an ideal $I$ of $k\langle X\rangle$, a word $w\in \langle X\rangle$ and an index $\gamma\in \Gamma$, we write $\langle X|I\rangle^{\vartriangleleft w}_\gamma:=\D_I\cap \langle X\rangle_\gamma^{\vartriangleleft w}$ and $\langle X|I\rangle_\gamma^{\trianglelefteq w}:=\D_I\cap \langle X\rangle_\gamma^{\trianglelefteq w}$.

\begin{proposition}\label{PBW-diagonal}
Let $\tau$ be an $X$-diagonal braiding on $kX$. Let $I$ be an ideal of $k\langle X\rangle$ and $R:=k\langle X\rangle/ I$. Let $\Delta$ be a left (resp. right) bounded $\tau$-comultiplication on $k\langle X\rangle$ such that $\Delta(I) ~\subseteq~ I\otimes k\langle X\rangle + k\langle X\rangle \otimes I$. Let $\rho=\tau^{-1}$ (resp. $\rho=\tau$). Then
\begin{enumerate}
	\item Every $I$-restricted word on $X$ is $I$-irreducible, that is $\C_I=\D_I$.
	\item For each Lyndon word $u$ on $X$ of finite height $n\geq 1$, 
	\[
	[u]^n_\rho ~\in~ [k\langle X|I\rangle^{\vartriangleleft u}_{\deg(u^n)}]_\rho +k\langle X\rangle_{<\deg(u^n)}+I.
	\]
	\item For each pair of  $I$-irreducible Lyndon words  $u,v\in \N_I$ with $v <_{\lex} u$,
	\[
	[[u]_\rho,[v]_\rho]_\rho ~\in~ [k\langle X|I\rangle_{\deg(uv)}^{\trianglelefteq uv}]_\rho +k\langle X\rangle_{<\deg(uv)} + I.
	\]
	\item $(\{[u]_\rho+I\}_{u\in \N_I}, h_I, <)$ is a system of PBW generators of  $R$ for each  total order  $<$ on $\N_I$.
	\item For each Lyndon word $u\in \N_I$ with $h_I(u) <\infty$, the scalar $\tau_{u,u}$ is a root of unity. More precisely,  if  $k$ is of characteristic $0$ then $\tau_{u,u}$ is of order $h_I(u)$, and in particular, $\tau_{u,u}\neq 1$; and if $k$ is of characteristic $p>0$ then  $\tau_{u,u} $ is of order $h_I(u)/p^s$ for some integer $s\geq 0$.
\end{enumerate}
\end{proposition}

\begin{proof}
To make the reading more fluent, the proof will be addressed in Section \ref{section-proof-PBW-generator}.
\end{proof}

\begin{remark}
In  \cite{ZSL3} and \cite{Kh}, there are similar results as above in the special cases that $\tau$ is the flipping map and that $k\langle X\rangle$ is length graded, $\tau$ is an $X$-diagonal braiding and $\Delta=\Delta_s$, respectively. The above proposition generalizes and strengthens these works in several aspects. 
\end{remark}
	
\begin{definition}\label{definition-closed}
A subset $Y$ of $X$ is called \emph{closed} if  $y<x$ for each $y\in Y$ and each $x\in X\backslash Y$.
\end{definition}


\begin{corollary}\label{Lyndon-compare}
Assume the notations and conditions of Proposition \ref{PBW-diagonal}. Assume in addition that $I$ is homogeneous. Let $Y$ be a closed subset of $X$, $J:=I\cap k\langle Y\rangle$ and $A:=k\langle Y\rangle/ J$. Then 
\begin{enumerate}
\item $h_J(u)=h_I(u)$ for every Lyndon word $u$ on $Y$. 
\item  $\N_J=\N_I \cap \langle Y\rangle$,  $\C_J=\D_J= \D_I\cap  \langle Y\rangle$ and $\OO_J=\OO_I \cap \langle Y\rangle$.
\item For each Lyndon word $u$ on $Y$ of finite height $n\geq 1$ (with respect to $J$), 
\[
[u]^n_\rho ~\in~ [k\langle Y|J\rangle^{\vartriangleleft u}_{\deg(u^n)}]_\rho + J.
\]
\item For each pair of  $J$-irreducible Lyndon words  $u,v\in \N_J$ with $v<_{\lex} u$,
\[
[[u]_\rho,[v]_\rho]_\rho ~\in~ [k\langle Y|J\rangle_{\deg(uv)}^{\trianglelefteq uv}]_\rho  + J.
\]
\item $(\{[u]_\rho+J\}_{u\in \N_J}, h_J, <)$ is a system of PBW generators of  $A$ for each  total order  $<$ on $\N_J$.
\end{enumerate}
\end{corollary} 

\begin{proof}
First note that the biggest letter in a Lyndon word must occur in the left-most place, so Lyndon words on $X$ pseudo-lexicographically smaller than a Lyndon word on $Y$ are necessarily Lyndon words on $Y$. Therefore, $\langle X\rangle^{\vartriangleleft u}_{\gamma}=\langle Y\rangle_{\gamma}^{\vartriangleleft u}$ for  every index $\gamma\in \Gamma$ and every Lyndon word $u$ on $Y$.

(1) Let $u$ be a Lyndon word on $Y$. Clearly, $h_I(u) \leq h_J(u)$ by definition. To see the equality, we may assume $h_I(u)=n$ is finite. By Proposition \ref{PBW-diagonal} (2) and the assumption that $I$ is homogeneous, there is an element $f\in I_{\deg(u^n)}$ such that $[u^n]_\rho -f \in  [k\langle X\rangle^{\vartriangleleft u}_{\deg(u^n)}]_\rho= [k\langle Y\rangle^{\vartriangleleft u}_{\deg(u^n)}]_\rho$. It follows that $f\in J$, and thereof $[u^n]_\rho \in [k\langle Y\rangle^{\vartriangleleft u}_{\deg(u^n)}]_\rho +J$. Thus $u^n$ is $J$-reducible and hence $h_J(u) \leq n$.
	
(2)	First we show $\N_J=\N_I \cap \langle Y\rangle$. Note that a Lyndon word on $Y$ is $I$-irreducible (resp. $J$-irreducible) if and only if $h_I(u) \geq 2$ (resp. $h_J(u) \geq 2$). So by Part (1), a  Lyndon word on $Y$ is $I$-irreducible if and only if it is $J$-irreducible. The desired equality  follows directly.  
	
Next we show the remaining equalities. Note that an $I$-irreducible word on $Y$ is necessarily $J$-irreducible, so $\D_I\cap \langle Y\rangle \subseteq \D_J$. By Proposition \ref{PBW-diagonal} (1), one has $\C_I \cap \langle Y\rangle = \D_I\cap \langle Y\rangle$.  In addition,  by Part (1) and the equality $\N_J=\N_I \cap \langle Y\rangle$, one gets $\C_J=\C_I \cap \langle Y\rangle$. Combine the above observations, one has $\D_J\subseteq \C_J =\D_I\cap \langle Y\rangle \subseteq \D_J$. Hence we have $\C_J=\D_J= \D_I\cap  \langle Y\rangle$.  Finally, one gets $\OO_J= \OO_I \cap \langle Y \rangle$ readily from the equality  $\D_J= \D_I\cap  \langle Y\rangle$, according to the definition of $\OO_J$ and $\OO_I$.

(3) Note that $\langle Y|J\rangle^{\vartriangleleft u}_{\deg(u^n)} =\D_J\cap \langle Y\rangle^{\vartriangleleft u}_{\deg(u^n)} = \D_I \cap \langle X\rangle^{\vartriangleleft u}_{\deg(u^n)} = \langle X|I\rangle^{\vartriangleleft u}_{\deg(u^n)}$. So
\[
	[u^n]_\rho ~\in~ \big ([k\langle X|I\rangle^{\vartriangleleft u}_{\deg(u^n)}]_\rho +I \big) \cap k\langle Y\rangle = \big ([k\langle Y|J\rangle^{\vartriangleleft u}_{\deg(u^n)}]_\rho +I \big) \cap k\langle Y\rangle = [k\langle Y|J\rangle^{\vartriangleleft u}_{\deg(u^n)}]_\rho +J
\]
by Part (1) and Proposition \ref{PBW-diagonal} (2). 

(4) The argument is similar to that of Part (3), by using  Proposition \ref{PBW-diagonal} (3).

(5) It is a direct consequence of the equality $\C_J=\D_J$, Part (3) and Proposition \ref{PBW-generator}.
\end{proof}

\begin{remark}
For an ideal $I$ of $k\langle X\rangle$ and a subset $Y \subseteq X$, knowing a generating set $G$ for $I$, it is not necessarily true that $G'=G\cap k\langle Y\rangle$ is also a generating set for $J=I\cap k\langle Y\rangle$.  To make it happen, a well-known choice of $G$ is an arbitrary Gr\"{o}bner basis of $I$ with respect to an arbitrary admissible order $\sigma$ on $\langle X\rangle$ that eliminates $X\backslash Y$ in the sense of \cite[Definition 17]{Nord}. Indeed, $G'$ is  a Gr\"{o}bner basis of $J$ with respect to the restriction of $\sigma$ to $\langle Y\rangle$ (see \cite[Proposition 19]{Nord}), hence $G'$ is a generating set of $J$. 
In the context of Corollary \ref{Lyndon-compare}, one may further obtains that 	
$\mathcal{G}_J= \mathcal{G}_I\cap k\langle Y\rangle,$
where $\mathcal{G}_I$ (resp. $\mathcal{G}_J$) denotes the reduced Gr\"{o}bner basis of $I$ (resp. $J$) with respect to $<_{\glex}$. So one gets more choices of $G$ when $I$ satisfy specific conditions. We refer to \cite{Nord} for those undefined notions.
\end{remark}

\begin{corollary}\label{PBW-relative}
Assume the notations and conditions of Proposition \ref{PBW-diagonal} with $I$ homogeneous. Let $Y_1\subseteq Y_2$ be two closed subsets of $X$ and let $J_i:=I\cap k\langle Y_i\rangle$ for $i=1,2$.
Consider  $B:=k\langle Y_1\rangle/ J_1$ as a subalgebra of $A := k\langle Y_2\rangle/ J_2$ under the natural map $B\to A$. Let $\Xi:= \N_{J_2}\backslash \langle Y_1\rangle$. Then  $(\{[u]_\rho+J_2\}_{u\in \Xi}, h_{J_2}, <)$ is a system of PBW generators of $A$ over $B$ for each total order $<$ on $\Xi$.
\end{corollary} 

\begin{proof}
Note that $\N_{J_i}=\N_I \cap \langle Y_i\rangle$ for $i=1, 2$ by Proposition \ref{Lyndon-compare} (2). So $\N_{J_2}$ is a disjoint union of $\N_{J_1}$ and $\Xi$. Let $<$ be a total order on $\Xi$. Fix a  total order $<'$ on $\N_{J_1} $. Define two total orders $<_1$ and $<_2$ on $\N_{J_2}$ as follows. The restriction of them to $\Xi$  and $\N_{J_1}$ are exactly $<$ and $<'$ respectively, but we require $u<_1 v$ and $v<_2 u$ respectively for every pair of words $u\in \N_{J_1}$ and $v\in \Xi$. For simplicity,   we write $z_u = [u]_\rho+ J_1 \in A$ for $u \in \N_{J_2}$.  Clearly, $[u]_\rho+J_1\in B$ is identified with $z_u\in A$ for every Lyndon word $u\in \N_{J_1}$, under the natural map $B\to A$. By  Corollary \ref{Lyndon-compare} (1, 5), 
\begin{eqnarray*}\label{PBW-1}
\{~ z_{u_1}^{r_1}\cdots z_{u_m}^{r_m} ~|~ m\geq 0,\, u_1<' \cdots <' u_m, \, u_i \in \N_{J_1}, \, r_i< h_{J_2}(u_i), \, i=1,\ldots, m ~\}
\end{eqnarray*}
is a basis of $B$ as a vector space, and for $t=1,2$ the set
\begin{eqnarray*}\label{PBW-2}
\{~ z_{u_1}^{r_1}\cdots z_{u_m}^{r_m} ~|~ m\geq 0,\, u_1<_t \cdots <_t u_m,\, u_i \in \N_{J_2}, \, r_i< h_{J_2}(u_i), \, i=1,\ldots, m ~\}
\end{eqnarray*}
are both bases of $A$ as a vector space. It follows readily that $\{~ z_{u_1}^{r_1}\cdots z_{u_m}^{r_m} ~|~ m\geq 0,\, u_1< \cdots < u_m,\, u_i \in \Xi, \, r_i< h_{J_2}(u_i), \, i=1,\ldots, m ~\}$ is a basis of $A$ as a left and right $B$-module.
\end{proof}

\section{Coideal subalgebras of connected graded braided bialgebras}

This section is devoted to study the structure of homogeneous  coideal subalgebras of connected graded braided bialgebras with braiding satisfying certain mild conditions. 

Let $\Gamma=(\Gamma, +, 0)$ be an abelian monoid. Recall that a $\Gamma$-graded algebra (coalgebra, braided algebra, etc.) is called \emph{connected} if its $0$-th component is  one-dimensional. 

\begin{remark}
Assume that $\Gamma$ is \emph{finitely decomposable}, which means that for each $\gamma\in \Gamma$, the set $$\{~(n,\gamma_1, \ldots, \gamma_n) ~|~ n\geq 1, ~ \gamma_1,\ldots,\gamma_n\in \Gamma\backslash\{0\},~ \gamma_1+\cdots+\gamma_n =\gamma~\}$$ is finite. Then every connected $\Gamma$-graded (braided) bialgebra $R$ is necessarily a  $\Gamma$-graded (braided) Hopf algebra. The antipode is given by $\mathcal{S}_R := \sum_{n\geq 0} (\eta_R\varepsilon_R-\id_R)^{*n}$, where $*$ denotes the convolution product of ${\rm Hom}(R,R)$.  Note that it is well-defined by the assumption of $\Gamma$.
\end{remark}

\begin{definition}
A \emph{left ({\rm resp.} right) coideal subalgebra} of  a braided bialgebra $R$ is a subalgebra $K$ of $R$ such that $\Delta_R(K) \subseteq R\otimes K$ (resp. $\Delta_R(K) \subseteq K\otimes R$).  Note that we require no compatible condition on $K$ with respect to the braiding of the  braided bialgebra $R$.
\end{definition}

In the sequel, we write $V_+:=\bigoplus_{\gamma\in \Gamma\backslash\{0\}}V_\gamma$ for a $\Gamma$-graded vector space $V=\bigoplus_{\gamma\in \Gamma}V_\gamma$.

The following theorem is a crucial result of this paper. As a matter of fact, most of the definitions and results developed in the previous two sections aim to prove it.

\begin{theorem}\label{braided-bialgebra-1}
Let $R=(R,\tau)$ be a connected $\Gamma$-graded braided bialgebra, where $\Gamma$ is a  nontrivial abelian monoid that has admissible well orders. Assume that $\tau=\tau_\chi$ for some bicharacter $\chi$ of $\Gamma$.
Let $A\supseteq B$ be homogeneous left (resp. right) coideal subalgebras of $R$. Then there exists a family $\{z_\xi\}_{\xi\in \Xi}$ of homogeneous elements of $A_+$, a map $h:\Xi\to \mathbb{N}_\infty$ and a total order $\vartriangleleft$ on $\Xi$ such that
	\begin{enumerate}	
		\item  $(\{z_\xi\}_{\xi\in \Xi}, h, <)$ is a system of PBW generators of $A$ over $B$ for each total order $<$ on $\Xi$.
		\item  For each $\xi\in \Xi$, the subalgebra $A^{\trianglelefteq \xi}$ of $A$ generated by $\{~z_\eta~|~\eta\in \Xi,~ \eta \trianglelefteq \xi ~\}$ over $B$ has  $$\{~ z_{\xi_1}^{r_1}\cdots z_{\xi_m}^{r_m} ~|~ m\geq 0,\, \xi_1\vartriangleleft \cdots \vartriangleleft \xi_m \trianglelefteq \xi,\, \xi_i \in \Xi,  \, r_i< h(\xi_i),  \, i=1,\ldots, m ~\}$$ as a basis of left $B$-module as well as of right $B$-module. 
		\item  For each $\xi\in \Xi$, the subalgebra $A^{\trianglelefteq \xi}$ is a left (resp. right) coideal of $R$, and  it is a subcoalgebra of $R$ in the case that $A, B$ are both subcoalgebras of $R$  and $\chi^2=\varepsilon_\Gamma$.
		\item For each $\xi\in \Xi$ with $h(\xi)<\infty$,  it follows that  $z_\xi^{h(\xi)}\in \bigcup_{\delta\vartriangleleft \xi} A^{\trianglelefteq \delta}$.
		\item For all $\xi,\eta\in \Xi$ with $\eta \vartriangleleft \xi$, it follows that  $[z_\xi, z_\eta]_\tau \in \bigcup_{\delta\vartriangleleft \xi} A^{\trianglelefteq \delta}$ (resp. $[z_\xi, z_\eta]_{\tau^{-1}} \in \bigcup_{\delta\vartriangleleft \xi} A^{\trianglelefteq \delta}$).
		\item  For each $\xi \in \Xi$ with $h(\xi) <\infty$, the scalar $a_\xi=\chi(\deg(z_\xi),\deg(z_\xi))$ is a root of unity. More precisely,  if  $k$ is of characteristic $0$ then $a_{\xi}$ is of order $h(\xi)$, and in particular, $a_{\xi}\neq 1$; and if $k$ is of characteristic $p>0$ then  $a_{\xi} $ is of order $h(\xi)/p^s$ for some integer $s\geq 0$.
	\end{enumerate}
\end{theorem}

\begin{proof}
	One may choose homogeneous subspaces $U_1\subseteq B_+$, $U_2\subseteq A_+$ and $U_3\subseteq R_+$ such that $U_1$ generates $B$, $U_2\cap B=0$ and $U_1+U_2$ generates $A$, $U_3\cap A=0$ and $V:=U_1+U_2+U_3$ generates $R$. Choose a homogeneous basis $X_i$ of $U_i$ for $i=1,2,3$. Let $X=X_1\cup X_2\cup X_3$. The braiding on $kX=V$  (denoted by $\tau_V$) is clearly $X$-diagonal. Further, equip $X$ with a well order $<$ as follows. 
	First fix an admissible well order $<_\Gamma$ on $\Gamma$ and a well order $<_{i,\gamma}$ on $X_{i,\gamma}:=\{~x\in X_i~|~ \deg(x) =\gamma ~\}$ for each integer $i=1,2,3$ and each  $\gamma \in \Gamma$; then for each $x_1,x_2\in X$, 
	\begin{eqnarray*}
		\label{definition-deglex}
		x_1 < x_2\,\, \Longleftrightarrow \,\, \left\{
		\begin{array}{llll}
			x_1\in X_i \text{ and } x_2\in X_j \text{ with } i<j, \quad \text{or} \\
			x_1, x_2\in X_i \text{ and }	 \left\{
			\begin{array}{llll}
				\deg(x_1)<_\Gamma\deg(x_2), \quad \text{or} && \\
				\deg(x_1) =\deg(x_2)\, \text{ and } x_1<_{i,\deg(x_1)} x_2.
			\end{array}\right. 
		\end{array}\right.
	\end{eqnarray*}
	Clearly, $Y_1:=X_1$ and $Y_2:=X_1\cup X_2$ are closed subsets of $X$.
	Consider $R$ as a $\Gamma$-graded quotient algebra $k\langle X\rangle/ I$ for some homogeneous ideal $I$ of $k\langle X\rangle$. By the counitity of $R$,  one may lift the comultiplication map of $R$ to a $\tau$-comultiplication $\Delta: k\langle X\rangle\to k\langle X\rangle \otimes^{\tau_V} k\langle X\rangle$ such that 
	\[
	\Delta(x)  ~ \in ~  1\otimes x + x\otimes 1 +   \big(k\langle X \rangle_+ \otimes k\langle X \rangle_+\big)_{\deg(x)}, \quad x\in X.
	\]
	and $\Delta(I)\subseteq k\langle X \rangle\otimes I + I\otimes k\langle X \rangle$. In addition, one may require for $i=1,2$ that
	\begin{eqnarray*}
		\Delta(x) \in 1\otimes x+x\otimes 1+ (k\langle X\rangle_+ \otimes k\langle Y_i\rangle_+)_{\deg(x)} \quad  \Big(\text{resp. } (k\langle Y_i\rangle_+ \otimes k\langle X\rangle_+)_{\deg(x)}  \Big ), \quad x\in Y_i,
	\end{eqnarray*}
	since $A$ and $B$ are left (resp. right) coideals of $R$. Clearly, $\Delta$ is right (resp. left) bounded. If $A, B$ are also both subcoalgebras of $R$, then we may assume for  $i=1,2$ that
	$$\Delta(x) \in 1\otimes x+x\otimes 1+ (k\langle Y_i\rangle_+ \otimes k\langle Y_i\rangle_+)_{\deg(x)}, \quad x\in Y_i $$
	and so in this case $\Delta$ is bounded. Let $J_i:=k\langle Y_i\rangle \cap I$ for $i=1,2$. We may identify $B$ and $A$ with $k\langle Y_1\rangle/ J_1$ and $k\langle Y_2\rangle/ J_2$ respectively under the isomorphisms of $\Gamma$-graded algebras induced from the canonical projection $k\langle X\rangle \to R$. Now let $\Xi:=\N_{J_2} \backslash \langle Y_1\rangle$ and for each $\xi\in \N_{J_2}$, let
	\[
	z_\xi:= [\xi]_\tau +J_2 \in A \quad \Big(~ \text{resp. } z_\xi:= [\xi]_{\tau^{-1}} +J_2 \in A ~\Big).
	\]
	Let $h:\Xi\to \mathbb{N}_\infty$ and $\vartriangleleft$ be the restriction of  $h_{J_2}:\mathbb{L}(Y_2) \to\mathbb{N}_\infty$ and $<_{\lex}$ to $\Xi$ respectively. 
	
	Part (1) is a direct consequence of Corollary \ref{PBW-relative}. Let $\xi\in \Xi$. Note that $[k\langle Y_2|J_2\rangle^{\trianglelefteq \xi}]_\tau+J_2$ (resp. $[k\langle Y_2|J_2\rangle^{\trianglelefteq \xi}]_{\tau^{-1}}+J_2$) is closed under multiplication  by  Corollary \ref{Lyndon-compare} (2,3) and Lemma \ref{rearragement}. Also,  $\langle Y_2|J_2\rangle^{\trianglelefteq \xi} :=\D_{J_2} \cap \langle Y_2\rangle^{\trianglelefteq \xi} \supseteq \D_{J_2} \cap \langle Y_1\rangle = \D_{J_1}$ by Corollary \ref{Lyndon-compare} (2). So  $A^{\trianglelefteq \xi}$ is the image of $[k\langle Y_2|J_2\rangle^{\trianglelefteq \xi}]_\tau+J_2$ (resp. $[k\langle Y_2|J_2\rangle^{\trianglelefteq \xi}]_{\tau^{-1}}+J_2$)  under the canonical projection $k\langle Y_2\rangle \to A$. Therefore, one has Part (2)  since the displayed set is  linearly independent over $B$ both from left and right according to  Part (1). Furthermore, by Lemma \ref{bracketing-diagonal-more} (1) and Corollary \ref{Lyndon-compare} (3),    
	$$[k\langle Y_2\rangle^{\trianglelefteq \xi}]_\tau +J_2= [k\langle Y_2|J_2\rangle^{\trianglelefteq \xi}]_\tau+J_2 \quad  (\text{resp.~} [k\langle Y_2\rangle^{\trianglelefteq \xi}]_{\tau^{-1}} +J_2= [k\langle Y_2|J_2\rangle^{\trianglelefteq \xi}]_{\tau^{-1}} + J_2 ).$$ 
	Then  Part (3) is easy to see by Proposition \ref{comultiplication-diagonal}. Parts (4, 5) are direct consequence of Corollary \ref{Lyndon-compare} (3, 4) respectively. Finally, Part (5) is easy to see by Corollary \ref{Lyndon-compare} (1) and Proposition \ref{PBW-diagonal} (5).
\end{proof}

Next, we  consider connected graded braided bialgebras with braidings that are more general than that of Theorem \ref{braided-bialgebra-1}. Similar but weaker results are obtained, as we shall see. 

A braided vector space $(V,\tau)$ is called \emph{of diagonal type} if $\tau$ is $\mathcal{X}$-diagonal for some basis $\mathcal{X}$ of $V$.
\begin{lemma}\label{homogenous-diagonal}
	Let $(V,\tau)$ be a $\Gamma$-graded braided vector space of diagonal type, where $\Gamma$ is an abelian monoid. Then $\tau$ is  diagonal with respect to  some homogeneous basis  of $V$.
\end{lemma}
\begin{proof}
	Let $\mathcal{X}$ be a basis of $V$ such that $\tau$ is $\mathcal{X}$-diagonal. Let $\mathcal{X}_h$ be the set of all homogeneous components of elements of $\mathcal{X}$.  Let $\mathcal{X}_h'$ be an arbitrary maximal linearly independent subset of $\mathcal{X}_h$. It is easy  to see that $\mathcal{X}_h'$ is a homogeneous basis of $V$, and $\tau$ is $\mathcal{X}_h'$-diagonal.
\end{proof}

\begin{proposition}\label{braided-bialgebra-2}
Let $R=(R,\tau)$ be a connected $\Gamma$-graded braided bialgebra, where $\Gamma$ is a nontrivial abelian monoid that has admissible well orders. Assume that $R$ is generated by a homogeneous braided subspace  of diagonal type.  Then there exists a family $\{z_\xi\}_{\xi\in \Xi}$ of homogeneous elements in $R_+$, a map $h:\Xi\to \mathbb{N}_\infty$ and a total order $\vartriangleleft$ on $\Xi$ such that
	\begin{enumerate}	
		\item  $(\{z_\xi\}_{\xi\in \Xi}, h, <)$ is a system of PBW generators of $R$ for each total order $<$ on $\Xi$; and moreover, $\tau$ is diagonal with respect to the basis $B(\{z_\xi\}_{\xi\in \Xi}, h, <)$ of $R$.
		\item  For each $\xi\in \Xi$, the subalgebra $R^{\trianglelefteq \xi}$ of $R$ generated by $\{~z_\eta~|~\eta\in \Xi,~ \eta \trianglelefteq \xi ~\}$ has a  basis $$\{~ z_{\xi_1}^{r_1}\cdots z_{\xi_m}^{r_m} ~|~ m\geq 0,\, \xi_1\vartriangleleft \cdots \vartriangleleft \xi_m \trianglelefteq \xi,\, \xi_i \in \Xi,  \, r_i< h(\xi_i),  \, i=1,\ldots, m ~\}.$$
		\item For each $\xi\in\Xi$, the subalgebra $R^{\trianglelefteq \xi}$ is a  left (resp. right) coideal of $R$, and it is a  subcoalgebra of $R$ in the case that  $\tau^2$ is the identity map of $R\otimes R$.
	    \item For each $\xi\in \Xi$ with $h(\xi)<\infty$,  it follows that  $z_\xi^{h(\xi)}\in \bigcup_{\delta\vartriangleleft \xi} R^{\trianglelefteq \delta}$.
	     \item For all $\xi,\eta\in \Xi$ with $\eta \vartriangleleft \xi$, it follows that  $[z_\xi, z_\eta]_\tau \in \bigcup_{\delta\vartriangleleft \xi} R^{\trianglelefteq \delta}$ (resp. $[z_\xi, z_\eta]_{\tau^{-1}} \in \bigcup_{\delta\vartriangleleft \xi} R^{\trianglelefteq \delta}$).
		\item  For each $\xi \in \Xi$ with $h(\xi) <\infty$, the scalar $a_{\xi}$ such that $\tau(z_\xi\otimes z_\xi) = a_{\xi}\cdot z_\xi\otimes z_\xi$ is a root of unity. More precisely,  if  $k$ is of characteristic $0$ then $a_{\xi}$ is of order $h(\xi)$, and in particular, $a_{\xi}\neq 1$; and if $k$ is of characteristic $p>0$ then  $a_{\xi} $ is of order $h(\xi)/p^s$ for some integer $s\geq 0$.
	\end{enumerate}
\end{proposition}

\begin{proof}
	Let $V\subseteq R_+$ be a homogeneous braided subspace of $R$  which is of diagonal type and generates $R$. By Lemma \ref{homogenous-diagonal}, one may fix a homogeneous basis $X$  of $V$ so that the braiding $\tau_V$ on $V$ is $X$-diagonal. Further, equip $X$ with a well order $<$ as follows. First fix an admissible well order $<_\Gamma$ on $\Gamma$ and a well order $<_\gamma$ on $X_\gamma:=\{~x\in X~|~\deg(x)=\gamma~\}$ for each  $\gamma\in \Gamma$; then for each $x_1,x_2\in X$,
	\begin{eqnarray*}\label{definition-deglex}
		x_1 < x_2\,\, \Longleftrightarrow \,\, \left\{
		\begin{array}{llll}
			\deg(x_1)<_\Gamma\deg(x_2), \quad \text{or} && \\
			\deg(x_1) =\deg(x_2)\, \text{ and } x_1 <_{\deg(x_1)} x_2.
		\end{array}\right.
	\end{eqnarray*}
	Consider $R$ as a $\Gamma$-graded quotient algebra $k\langle X\rangle/ I$ for some homogeneous ideal $I$ of $k\langle X\rangle$, where $k\langle X\rangle$ is $\Gamma$-graded by that of $kX=V$.
	By the counitity of $R$, one may lift the comultiplication map of $R$ to a $\tau$-comultiplication $\Delta: k\langle X\rangle\to k\langle X\rangle \otimes^{\tau_V} k\langle X\rangle$ such that 
	\[
	\Delta(x)  ~ \in ~  1\otimes x + x\otimes 1 +   \big(k\langle X \rangle_+ \otimes k\langle X \rangle_+\big)_{\deg(x)}, \quad x\in X,
	\]
	and $\Delta(I)\subseteq k\langle X \rangle\otimes I+I\otimes k\langle X \rangle.$ Clearly, $\Delta$ is bounded. 
	Now let $\Xi:=\N_I$ and for each $\xi\in \N_I$, let
	\[
	z_\xi:= [\xi]_\tau +I \in R \quad \Big(~ \text{resp. } z_\xi:= [\xi]_{\tau^{-1}} +I \in R ~\Big).
	\]
	Let $h:\Xi\to \mathbb{N}_\infty$ and $\vartriangleleft$ be the restriction of  $h_I:\mathbb{L}(X) \to\mathbb{N}_\infty$ and $<_{\lex}$ to $\Xi$ respectively. 
	
	The first statement of Part (1) follows directly from Proposition \ref{PBW-diagonal} (4), and the second one is clear. Let $\xi\in \Xi$. Note that  $[k\langle X|I\rangle^{\trianglelefteq \xi}]_\tau+I$ (resp. $[k\langle X|I\rangle^{\trianglelefteq \xi}]_{\tau^{-1}}+I$) is closed under multiplication  by  Proposition \ref{PBW-diagonal} (1,2) and Lemma \ref{rearragement}. So  $R^{\trianglelefteq \xi}$ is the image of $[k\langle X|I\rangle^{\trianglelefteq \xi}]_\tau+I$ (resp. $[k\langle X|I\rangle^{\trianglelefteq \xi}]_{\tau^{-1}}+I$)  under the canonical projection $k\langle X\rangle \to R$. Therefore, one has Part (2)  since the displayed set is  linearly independent according to  Part (1). Furthermore, by Lemma \ref{bracketing-diagonal-more} (1) and Proposition \ref{PBW-diagonal} (2),   
	$$[k\langle X\rangle^{\trianglelefteq \xi}]_\tau +I= [k\langle X|I\rangle^{\trianglelefteq \xi}]_\tau+I \quad  (\text{resp.~} [k\langle X\rangle^{\trianglelefteq \xi}]_{\tau^{-1}} +I= [k\langle X|I\rangle^{\trianglelefteq \xi}]_{\tau^{-1}} +I ).$$ Then Part (3) is easy to see by Proposition \ref{comultiplication-diagonal}. Finally, Parts (4),  (5) and  (6) are direct consequences of Proposition \ref{PBW-diagonal}  (2), (3) and (5) respectively.
\end{proof}

Let $V$ be a vector space and $\mathcal{X}$ a partially ordered basis of $V$. A braiding $\tau$ on $V$ is said to be \emph{categorical with respect to $\mathcal{X}$} (or simply \emph{$\mathcal{X}$-categorical}) if 
\[
\tau(x\otimes y) ~\in ~ \sum_{u, z\in \mathcal{X},~ u\leq y, z\leq x}k\cdot (u\otimes z), \quad x, ~ y\in \mathcal{X}.
\] 
Every $\mathcal{X}$-diagonal braiding on $V$ is clearly $\mathcal{X}$-categorical. A braided vector space $(V,\tau)$ is called \emph{of categorical type} if $\tau$ is categorical with respect to some well-ordered basis of $V$.

\begin{lemma}\label{homogenous-categorical}
	Let $(V,\tau)$ be a $\Gamma$-graded braided vector space of categorical type, where $\Gamma$ is an abelian monoid. Then $\tau$ is  categorical with respect to  some well-ordered homogeneous basis  of $V$.
\end{lemma}

\begin{proof}
	
	Let $\mathcal{X}$ be a well-ordered basis of $V$ such that $\tau$ is $\mathcal{X}$-categorical. For each $u\in V$ and  $\gamma\in \Gamma$, we denote by $u_\gamma$ the $\gamma$-th component of $u$. An element  $x\in \mathcal{X}$ is called $\gamma$-critical if $x_\gamma$ is not a linear combination of the set $\{~y_\gamma ~|~y\in X, y<x ~\}$. Let $\mathcal{X}_{\gamma}'$ be the set of the $\gamma$-th component of all $\gamma$-critical elements of $\mathcal{X}$. It is easy to see that $\mathcal{X}_{\gamma}'$ is a basis of $V_\gamma$ for each $\gamma\in \Gamma$. So $\mathcal{X}_h':=\bigcup_{\gamma\in \Gamma} \mathcal{X}_\gamma'$ is a homogeneous basis of $V$. Equip $\mathcal{X}_h'$ with a well order $<$ defined as follows. First choose an arbitrary well order $<_\Gamma$ on $\Gamma$; then for $x_\gamma, y_\delta\in \mathcal{X}_h'$, 
	$$x_\gamma<y_\delta \Longleftrightarrow \gamma<_\Gamma \delta ~ \text{ or } ~ \gamma=\delta ~ \text{ but } ~ x<y.$$
	Since $\tau(V_\gamma\otimes V_\delta) =V_\delta\otimes V_\gamma$ for all $\gamma, \delta\in \Gamma$, it follows readily that $\tau$ is $\mathcal{X}_h'$-categorical. 
\end{proof}

The following result is inspired from \cite[Theorem 1.3]{Kh1}.

\begin{proposition}\label{braided-bialgebra-3}
Let $R=(R,\tau)$ be a connected $\Gamma$-graded braided bialgebra, where $\Gamma$ is a nontrivial abelian monoid that has admissible well orders. Assume that $R$ is generated by a homogeneous braided subspace of categorical type. Let $A\supseteq B$ be  homogeneous left (resp. right) coideal  subalgebras of $R$. Then there exists a family $\{z_\xi\}_{\xi\in \Xi}$ of homogeneous elements in $A_+$ and a map $h:\Xi\to \mathbb{N}_\infty$ such that $(\{z_\xi\}_{\xi\in \Xi}, h, <)$ is a system of PBW generators of $A$ over $B$ for each total order $<$ on $\Xi$.
\end{proposition}
\begin{proof}
The proof is  by the filtered-graded method. Theorem \ref{braided-bialgebra-1} is used in an essential way. Details will be addressed in Section \ref{section-proof-braided-bialgebra-3}.
\end{proof}

\begin{remark}
Following Milnor-Moore \cite[Definition 4.16]{MM}, we may define the notion of \emph{braided quasi-bialgebra} to be a braided algebra $(R, \tau)$ together with two homomorphisms of algebras $\Delta: R\to R\otimes^\tau R$ and $\varepsilon: R\to k$ such that $(\varepsilon \otimes \id_R) \circ \Delta =\id_R$ and $(\id_R \otimes \varepsilon) \circ \Delta =\id_R$. Here, we don't require $\Delta$ to be coassociative in the sense that $(\Delta\otimes \id_R) \circ \Delta = (\id_R \otimes \Delta) \circ \Delta$. Our argument shows that the results presented in this section 
hold in the context of $\Gamma$-graded braided quasi-bialgebras.
\end{remark}

\section{Coideal subalgebras of  pointed Hopf algebras}

In this section, we study the structure of  coideal subalgebras of pointed Hopf algebras. The key point is that by the bosonizations (or the Radford biproducts) of appropriate braided Hopf algebras, the problems may reduce to the one handled in the previous section.

Let $C$ be a coalgebra. The \emph{coradical} of  $C$ is denoted by ${\rm Corad}(C)$. It  is defined to be the sum of all simple subcoalgebras of $C$.  Let $G(C)$ be the set of all group elements of $C$. Clearly, $kG(C) \subseteq {\rm Corad}(C)$. We call $C$ \emph{pointed} if $kG(C) = {\rm Corad}(C)$. If $C=\bigoplus_{\gamma\in \Gamma}C_\gamma$ is a  $\Gamma$-graded coalgebra with $\Gamma$  an abelian monoid that has admissible well orders, then ${\rm Corad} (C)= {\rm Corad}(C_0)$, and hence $C$ is pointed if and only if $C_0$ is so. A bialgebra (Hopf algebra) is called pointed if it is so as a coalgebra. Note that a pointed bialgebra $H$ is a Hopf algebra if and only if $G(H)$ is a group.

Let $K$ be a Hopf algebra. For an  algebra $H$ that contains $K$ as a  subalgebra, the \emph{(left) adjoint action}  of  $K$ on $H$ is the linear map  $\ad= \ad_H: K\otimes H\to H$ given by	
\[
\ad(a\otimes x):=\sum a_{(1)} \cdot x \cdot \mathcal{S}_K(a_{(2)}) , \quad a\in K,~ x\in H.
\]
Note that $H$ becomes a (left) $K$-module algebra under this action.

The following theorem  is one of the main results of this paper.  By a \emph{semi-invariant} of a Hopf algebra $H$ we mean an element $x\in H$ such that $kx$ is stable under the adjoint action of $kG(H)$ on $H$. A module $M$ over an algebra $A$ is called \emph{locally finite} if  $\dim (A\cdot v)< \infty$ for every $v\in M$.

\begin{theorem}\label{pointed-Hopf-1}
Let $H$ be a pointed Hopf algebra with  $G=G(H)$ abelian. Assume that one of the following two conditions hold: 
(1) $H$ is locally finite as a $kG$-module under the adjoint action of $kG$ on $H$, and the base field $k$ is algebraically closed; (2) $H$ is generated over  $kG$ by a set of semi-invariants of $H$.
Let $A\supseteq B$ be left (resp. right) coideal subalgebras of $H$ that contain $G$. Then there exists a family $\{z_\xi\}_{\xi\in \Xi}$ of elements in $A$ and a map $h:\Xi\to \mathbb{N}_\infty$ such that $(\{z_\xi\}_{\xi\in \Xi}, h, <)$ is a system of PBW generators of $A$ over $B$ for each total order $<$ on $\Xi$.
\end{theorem}

\begin{proof}
Let $\mathcal{F}=(H_{(n)})_{n\geq 0}$ be the coradical filtration of $H$. Since $H$ is pointed, $\mathcal{F}$ is a Hopf algebra filtration of $H$, i.e. $\mathcal{F}$ is  an algebra  filtration and a coalgebra filtration of $H$, and $\mathcal{S}_H(H_{(n)})\subseteq H_{(n)}$ for all $ n\geq 0$. Thus $Q:=\gr(H, \mathcal{F}) = \bigoplus_{n\geq 0} H_{(n)}/H_{(n-1)}$ is an $\mathbb{N}$-graded Hopf algebra, where $H_{(-1)}:=0$.  By definition, $Q_0=kG$. Note that  $H_{(n)}$ are all $kG$-submodules of $H$ (equipped with the adjoint action of $kG$ on it).  If the condition (1) hold, then $Q$ is  locally finite as a $Q_0$-module under the adjoint action of $Q_0$ on $Q$. Now assume the condition (2) hold. Then it is easy to see that $H$ equals to the sum of its  one-dimensional $kG$-submodules. So $H_{(n)}/H_{(n-1)}$ all equal to the sum of their one-dimensional $kG$-submodules by \cite[Lemma 9.1]{AnFu}. Thus, $Q$ is generated over $Q_0$ by its semi-invariants in this setting. Since $ \gr(A, \mathcal{F}|_A)\supseteq\gr(B,\mathcal{F}|_B)$ are both homogeneous left (resp. right) coideal subalgebras of $Q$ that contain $Q_0$, the result follows from Lemma \ref{PBW-tranfer} and Proposition \ref{pointed-Hopf-2} given below.
\end{proof}

\begin{proposition}\label{pointed-Hopf-2}
Let $H=\bigoplus_{\gamma\in \Gamma}H_\gamma$ be a $\Gamma$-graded Hopf algebra with  $H_0=kG$ for some abelian group $G$, where $\Gamma$ is a nontrivial abelian monoid that has admissible well orders. Assume one of the following two conditions hold: (1) $H$ is locally finite as an $H_0$-module under the adjoint action of $H_0$ on $H$, and the base field $k$ is algebraically closed; (2) $H$ is generated over  $H_0$ by a set of semi-invariants of $H$.
Let $A\supseteq B$ be homogeneous left (resp. right) coideal subalgebras of $H$ that contain $H_0$. 
Then there exists a family $\{z_\xi\}_{\xi\in \Xi}$ of homogeneous elements of $A_+$ and a map $h:\Xi\to \mathbb{N}_\infty$ such that $(\{z_\xi\}_{\xi\in \Xi}, h, <)$ is a system of PBW generators of $A$ over $B$ for each total order $<$ on $\Xi$.
\end{proposition}

The remaining of this section is devoted to prove this proposition.

Firstly, let us recall some well-known facts that we need in our context. Let $\Gamma$ be an abelian monoid. Let $K$ be a Hopf algebra with bijective antipode. Let $H=\bigoplus_{\gamma\in \Gamma} H_\gamma$ be a $\Gamma$-graded Hopf algebra with $H_0=K$. Let $\pi:H\to K$ be the canonical projection, which is a homomorphism of Hopf algebras. Then $(H, \ad, \rho)$ is a $\Gamma$-graded Yetter-Drinfeld module over $K$, where $\ad:K\otimes H\to H$ is the adjoint action of $K$ on $H$, and where $\rho: H\to K\otimes H$ is the linear map given by 
$$\rho(x):= \sum \pi(x_{(1)}) \otimes x_{(2)}, \quad x\in H.$$
Let $R:=H^{{\rm co}\, K}$ be the right coinvariants of $H$ with respect to $K$, i.e.
$$R:=\{~x\in H ~|~ \sum x_{(1)}\otimes \pi(x_{(2)}) =x\otimes 1 ~\}.$$ 
Clearly, $R=\bigoplus_{\gamma\in \Gamma}R_\gamma$, where $R_\gamma= R\cap H_\gamma$. It turns out that $R$ is naturally 
a $\Gamma$-graded  Hopf algebra in ${}^K_K\mathcal{Y}\mathcal{D}$. In particular, $R$ is a homogeneous subalgebra of $H$ as well as a homogeneous Yetter-Drinfeld submodule of $(H, \ad, \rho)$. 
Moreover, the multiplication map
\begin{eqnarray*}\label{Bosonization}
\Phi: R ~\#~K \to H, \quad r\otimes a\mapsto ra
\end{eqnarray*}
is an isomorphism of $\Gamma$-graded Hopf algebras with inverse given by
\begin{eqnarray*}
\Psi: H\to R ~\#~K, \quad x\mapsto \sum \big(x_{(1)}\cdot \mathcal{S}_K(\pi(x_{(2)}))\big) \otimes \pi (x_{(3)}),
\end{eqnarray*}
 where $R~\#~ K$ is the \emph{bosonization}  of $R$ by $K$. Note that $R~\# ~ K=R\otimes K= \bigoplus_{\gamma\in \Gamma}R_\gamma\otimes K$ as a $\Gamma$-graded vector space, and its algebra (resp. coalgebra) structure is given by the smash product (resp. coproduct) of $R$ and $K$. 
We refer to \cite{HS, Rad85} for above-mentioned notions and facts. 

Let  $\mathfrak{E}_r(H)$ be the set of all homogeneous right coideal subalgebras of $H$ that contain $K$, and  $\mathfrak{F}_r(R)$ the set of all homogeneous right coideal subalgebras of $R$ which are also $K$-submodules of $R$. Here, we don't require elements of $\mathfrak{F}_r(R)$ to be Yetter-Drinfeld submodules of $R$.

\begin{lemma}\label{subalgebra-correspongding}
With the above notation, the map $\mathfrak{E}_r(H) \to \mathfrak{F}_r(R)$ given by $ A\mapsto A\cap R$, and the map $\mathfrak{F}_r(R) \to \mathfrak{E}_r(H)$ given by $F\mapsto FK$ form mutually inverse bijections of sets. In addition, for each $A\in \mathfrak{E}_r(H)$, it follows that $(A\cap R)\otimes K$ is a homogeneous right coideal subalgebra of $R ~\# ~ K$ and the multiplication map $(A\cap R) ~\otimes ~ K\to A$ is an isomorphism of $\Gamma$-graded algebras.
\end{lemma}

\begin{proof}
The first statement is simply the graded version of \cite[Lemma 12.4.3 (2)]{HS}. Since $\Phi: R ~\#~K \to H$ is an isomorphism of $\Gamma$-graded Hopf algebras, the second statement follows immediately.
\end{proof}

Let $G$ be an abelian group and  $V\in {}^{kG}_{kG}\mathcal{Y}\mathcal{D}$ a Yetter-Drinfeld module over $kG$. For each $g\in G$, let $V_g=\{~v\in V ~|~ \rho(v)=g\otimes v~\}$, where $\rho$ is the coaction of $kG$ of $V$.  Note that $V_g$ are all $kG$-submodules of $V$ and $V=\bigoplus_{g\in G} V_g$. In addition, the induced braiding  on $V$ satisfies 
\[
\tau_{V,V} (v\otimes w) = (g\cdot w)\otimes v, \quad g\in G,~ v\in V_g, ~w\in W.
\]

\begin{lemma}\label{categorical-criterion}
Let $G$ be an abelian group. Assume that $k$ is algebraically closed. 
\begin{enumerate}
\item Let $M$ be a locally finite $kG$-module. Then $M$ has a well-ordered basis $\mathcal{X}$ such that for each $x\in \mathcal{X}$, the subspace spanned by $\mathcal{X}_{\leq x}=\{~y\in \mathcal{X} ~|~ y\leq x ~\}$ is a $kG$-submodule of $M$.
\item Let $V$ be a Yetter-Drinfeld module over $kG$ which is locally finite as a $kG$-module. Then the braiding $\tau_{V,V}$ on $V$ is categorical with respect to some well-ordered basis of $V$.
\end{enumerate}
\end{lemma}

\begin{proof}
(1) Define an increasing sequence $F_0(M) \subseteq F_1(M) \subseteq \cdots$ of $kG$-submodules of $M$ as follows. Firstly, let $F_0(M)$ be the sum of all one-dimensional $kG$-submodules of $M$; then for  $n\geq 1$, inductively set $F_n(M)$ to be the $kG$-submodule of $M$ such that $F_n(M)/F_{n-1}(M)$ is the sum of all one-dimensional $kG$-submodules of $M/F_{n-1}(M)$.  Since $G$ is abelian and $k$ is algebraically closed, every finite-dimensional simple $kG$-module is one-dimensional, see \cite[Lemma 5.4.12]{HS}. Therefore,
$$M=\bigcup_{n\geq 0} F_n(M).$$
Let $F_{-1}(M)=0$. Note that $Q_n(M):=F_{n}(M)/F_{n-1}(M)$ is a direct sum of one-dimensional $kG$-submodules by \cite[Proposition 9.3]{AnFu}. So one may fix a set $\mathcal{X}_n\subseteq F_n(M)$ such that $\{~\bar{x}:= x+F_{n-1}(M) ~\}_{x\in \mathcal{X}_n}$ is a basis of $Q_n(M)$ and each $\bar{x}$ spans a $kG$-submodule of $Q_n(M)$. Clearly, $\mathcal{X}_n$ are mutually disjoint. It is also easy to check that $\mathcal{X}=\bigcup_{n\geq 0}\mathcal{X}_n$ is a basis of $M$. Now fix a well order $<_n$ on $\mathcal{X}_n$ for each $n\geq 0$. Then equip  $\mathcal{X}$ with a well order $<$ as follows: for each $x_1,\, x_2\in \mathcal{X}$, 
	\begin{eqnarray*}
	\label{definition-deglex}
	x_1 < x_2\,\, \Longleftrightarrow \,\, \left\{
	\begin{array}{llll}
		x_1\in \mathcal{X}_i \text{ and } x_2\in \mathcal{X}_j \text{ with } i<j, \quad \text{or} \\
		x_1, x_2\in \mathcal{X}_i \text{ and } x_1<_i x_2	. 
	\end{array}\right.
\end{eqnarray*}
Readily, the subspace spanned by $\mathcal{X}_{\leq x}$ is a $kG$-submodule of $M$ for each $x\in \mathcal{X}$.

(2) By (1), one may fix a well-ordered basis $(\mathcal{X}_g, <_g)$ of $V_g$ for each $g\in G$ such that the set $\mathcal{X}_{g,\leq_g x}$ spans a $kG$ submodule of $V_g$ for any $x\in \mathcal{X}_g$. Clearly, $\mathcal{X}=\bigcup_{g\in G}\mathcal{X}_g$ is a basis of $V$. Now fix a well order $<_G$ on $G$, and then equip  $\mathcal{X}$ with a well order $<$ as follows: for each $x_1,\, x_2\in \mathcal{X}$, 
\begin{eqnarray*}
	\label{definition-deglex}
	x_1 < x_2\,\, \Longleftrightarrow \,\, \left\{
	\begin{array}{llll}
		x_1\in \mathcal{X}_g \text{ and } x_2\in \mathcal{X}_h \text{ with } g<_G h, \quad \text{or} \\
		x_1, x_2\in \mathcal{X}_g \text{ and } x_1<_g x_2	. 
	\end{array}\right.
\end{eqnarray*}
It is easy to check that $\tau_{V,V}$ is categorical with respect to $(\mathcal{X},<)$.
\end{proof}

\begin{proof}[Proof of Proposition \ref{pointed-Hopf-2}]
Since the antipode of $H$ is bijective, we may assume $A\supseteq B$ are both homogeneous right coideal subalgebras of $H$. Let $R=H^{{\rm co}\, kG}$ be the right invariants of $H$ with respect to $kG$. So $R$ is naturally a $\Gamma$-graded Hopf algebra in ${}^{kG}_{kG}\mathcal{Y}\mathcal{D}$. Note that $A\cap R\supseteq B\cap R$ are homogeneous right coideal subalgebras of $R$. By Lemma \ref{subalgebra-correspongding}, it suffices to show there exists a family $\{z_\xi\}_{\xi\in \Xi}$ of homogeneous elements in $(A \cap R)_+$ and a map $h:\Xi\to \mathbb{N}_\infty$ such that $(\{z_\xi\}_{\ xi\in \Xi}, h, <)$ is a system of PBW generators of $A\cap R$ over $B\cap R$ for each total order $<$ on $\Xi$. Then by Proposition \ref{braided-bialgebra-3}, it remains to show that the braided vector space $(R,\tau_{R,R})$ is of categorical type in both settings. 

To see this, note that $R$ is a homogeneous Yetter-Drinfeld submodule of $H$.
If the condition (1) hold, then  $R$ is locally finite as a $kG$-module, hence the desired statement hold by Lemma \ref{categorical-criterion} (2). Now assume the condition (2) holds. Then it is clear that $(H_\gamma)_g$  equals to the sum of its one-dimensional $kG$-submodules for each $\gamma\in \Gamma$ and each $g\in G$. It follows that each $(R_\gamma)_g$ is a direct sum of a family of its one-dimensional $kG$-submodules by \cite[Lemma 9.1]{AnFu}. Therefore, one may choose for each  $(R_\gamma)_g$ a basis $\mathcal{X}_\gamma$  which consists of semi-invariants of $H$. Clearly, $\mathcal{X}:=\bigcup_{\gamma\in \Gamma, ~ g\in G} (\mathcal{X}_\gamma)_g$ is a basis of $R$ and  $\tau_{R,R}$  is $\mathcal{X}$-diagonal.  The desired statement follows  by fixing any well order on $\mathcal{X}$.
\end{proof}

\section{Coideal subalgebras of connected Hopf algebras}

This section is devoted to study the structure of coideal subalgebras of connected  Hopf algebras.
In addition to Theorem B stated in the introduction, we generalize some of the main results of \cite{BG, Zh, LZ, ZSL3}. Particularly, we don't use any geometric fact, unlike that of \cite{BG, Zh}.

Recall that a coalgebra $C$ is called \emph{connected} if its coradical is one-dimensional. Note that if $C=\bigoplus_{\gamma\in \Gamma}C_\gamma$ is a connected $\Gamma$-graded coalgebra with $\Gamma$  an abelian monoid that has admissible well orders, then  $C$ is connected. A bialgebra (Hopf algebra) is called connected if it is  so as a coalgebra. Note that a connected bialgebra is necessarily a connected Hopf algebra.

\begin{theorem}\label{connected-graded-Hopf-algebra-PBW}
Assume that the base field $k$ is of characteristic $0$. Let $H$ be a connected $\Gamma$-graded Hopf algebra, where $\Gamma$ is a  nontrivial abelian monoid that has admissible well orders. Let $A\supseteq B$ be homogeneous left (resp. right) coideal subalgebras of $H$. Then there exists a family $\{z_\xi\}_{\xi\in \Xi}$ of homogeneous elements of $A_+$ and a total order $\vartriangleleft$ on $\Xi$ such that
\begin{enumerate}	
		\item  $(\{z_\xi\}_{\xi\in \Xi}, <)$ is a system of PBW generators of $A$ over $B$ for each total order $<$ on $\Xi$.
		\item  For each $\xi\in \Xi$, the subalgebra $A^{\trianglelefteq \xi}$ of $A$ generated by $\{~z_\eta~|~\eta\in \Xi,~ \eta\trianglelefteq \xi ~\}$ over $B$ has  $$\{~ z_{\xi_1}^{r_1}\cdots z_{\xi_m}^{r_m} ~|~ m\geq 0,\, \xi_1\vartriangleleft \cdots \vartriangleleft \xi_m \trianglelefteq \xi,\, \xi_i \in \Xi,  \, r_i\geq0,  \, i=1,\ldots, m ~\}$$ as a basis of left $B$-module as well as of right $B$-module. 
		\item  For each $\xi\in \Xi$, the subalgebra $A^{\trianglelefteq \xi}$ is a left (resp. right) coideal of $H$, and  it is a Hopf  subalgebra of $H$ in the case that $A, B$ are both Hopf subalgebras of $H$.
		\item For all $\xi,\eta\in \Xi$ with $\eta \vartriangleleft \xi$, it follows that $[z_\xi, z_\eta] \in \bigcup_{\delta\vartriangleleft \xi} A^{\trianglelefteq \delta}$.
	\end{enumerate}
\end{theorem}

\begin{proof}
Consider $H$ as a connected $\Gamma$-graded braided bialgebra with braiding the flipping map, i.e., the one induced by the trivial bicharacter of $\Gamma$. By Theorem \ref{braided-bialgebra-1}, the result follows.
Particularly, Theorem \ref{braided-bialgebra-1} (6) tells us that the hidden map $h:\Xi\to \mathbb{N}_\infty$ is given by $h(\xi)=\infty$ for each $\xi\in \Xi$.
\end{proof}

\begin{corollary}\label{polynomial-generator-graded}
Assume that $k$ is of characteristic $0$. Let $H$ be a connected $\Gamma$-graded Hopf algebra, where $\Gamma$ is a  nontrivial abelian monoid that has admissible well orders. Let $A\supseteq B$ be homogeneous left (resp. right) coideal subalgebras of $H$.  If $A$ is commutative, then $A$ is isomorphic as a $\Gamma$-graded algebra to the graded polynomial algebra over $B$ in some family of graded variables. 
\end{corollary}

\begin{proof}
It is a direct consequence of Theorem \ref{connected-graded-Hopf-algebra-PBW} (1).
\end{proof}

\begin{definition}\label{definition-Ore-extension}
Let $R$ be an algebra. An \emph{Ore extension}  of  $R$  is a quadruple $(S,x, \phi, \delta)$ consists of an algebra $S$, an element $x\in S$, an automorphism $\phi$ of $R$ and  a $\phi$-derivation $\delta$  of $R$ such that 
\begin{enumerate}
	\item $S$ contains $R$ as a subalgebra;
	\item $S$ is a free left $R$-module with basis $\{1, x, x^2,\ldots\}$;
	\item $xr =\phi(r) x +\delta(r)$ for all $r\in R$.
\end{enumerate}
We write $S=R[x; \phi, \delta]$ to denote such an Ore extension of $R$. Frequently, we simply call the algebra $S$ itself an Ore extension of $R$. An Ore extension of $R$ is called \emph{of derivation type} if $\phi=\id_R$.

In the context of graded setting, one defines the notion of \emph{graded Ore extensions} of graded algebras in a similar way by requiring $x$, $\phi$ and $\delta$ to be homogeneous of appropriate degree.
\end{definition}

\begin{theorem}\label{IHOE-graded}
Assume that $k$ is of characteristic $0$. Let $H$ be a connected $\Gamma$-graded Hopf algebra, where $\Gamma$ is a  nontrivial abelian monoid that has admissible well orders. Let $A\supseteq B$ be homogeneous left (resp. right) coideal subalgebras of $H$, which  are of finite GK dimension $m$ and $n$ respectively. Then there is a sequence  $B=K_0 \subset K_1 \subset \cdots \subset K_{m-n} = A$ of homogeneous left (resp. right) coideal subalgebras of $H$ such that each $K_i$ is a graded Ore extension of $K_{i-1}$ of derivation type. Moreover, if $A$ and $B$ are Hopf subalgebras of $H$ then $K_i$ can be chosen to be Hopf subalgebras of $H$.
\end{theorem}

\begin{proof}
Firstly, fix a homogeneous system of PBW generators $(\{z_\xi\}_{\xi\in \Xi}, \vartriangleleft)$ of $A$ over $B$  satisfying Theorem \ref{connected-graded-Hopf-algebra-PBW} (2, 3, 4). By Lemma \ref{PBW-GK-dimension}, one has $\#(\Xi)\leq m-n$. 
Thus, we may assume $\Xi= \{1,\ldots, l\}$ which is equipped with the natural order. Now let $K_0= B$; for $i=1,\ldots, l$, let $K_i$ be the subalgebra of $A$ generated by  $z_1,\ldots, z_i$ over $B$.  By Theorem \ref{connected-graded-Hopf-algebra-PBW} (2), the set $$\{1, z_i, z_i^2, \ldots ~\}$$ is a homogeneous basis of $K_i$ as a $\Gamma$-graded left $K_{i-1}$-module. According to Theorem \ref{connected-graded-Hopf-algebra-PBW} (4), one may define a homogeneous linear map $\delta_i:K_{i-1} \to K_{i-1}$ of degree $\deg(z_i)$  by $$f\mapsto z_i \cdot f -f\cdot z_i, \quad f\in K_{i-1}.$$ It is easy to check that $\delta_i$ is a derivation of $K_{i-1}$. Therefore, one has 
$$K_i=K_{i-1}[z_i; \id, \delta_i], \quad i=1,\ldots, l.$$ 
Finally, apply Theorem \ref{connected-graded-Hopf-algebra-PBW} and Lemma \ref{PBW-GK-dimension} to the pair $B\supseteq k$, one deduces that $B$ is finitely generated over $k$. So $\gkdim A=\gkdim B+ l$ by \cite[Proposition 3.5]{KL}, hence $l=m-n$. 
\end{proof}

The next result generalizes \cite[Theorem 6.9]{Zh} in two aspects. The first one is that  we extend the context from connected Hopf algebras of finite GK dimension to their coideal subalgebras of arbitrary GK dimension; the other one is that we don't ask $k$ to be algebraically closed. Note that \cite[Theorem 6.9]{Zh} is in some sense the starting point of the recent research line on connected Hopf algebras.
See \cite[Theorem 1.1]{BG} and \cite[Proposition 3.4]{ZSL3} for other weaker generalizations.

\begin{theorem}\label{GKdim-transfer}
Assume that $k$ is of characteristic $0$. Let $H$ be a connected  Hopf algebra. Let $A$ be a one-sided coideal subalgebra of $H$.  Let $\mathcal{F}=(H_{(n)})_{n\geq 0}$ be the coradical filtration of $H$. Then $\gr(A,\mathcal{F}|_A)$ is a graded polynomial algebra over $k$ with $ \gkdim A =\gkdim \gr(A,\mathcal{F}|_A) \in \mathbb{N}_\infty$. 
\end{theorem}

\begin{proof}
By \cite[Proposition 6.4]{Zh}, $\gr(H,\mathcal{F})$ is a commutative connected $\mathbb{N}$-graded Hopf algebra. Note that $C:=\gr(A,\mathcal{F}|_A)\supseteq k$ are  homogeneous one-sided coideal subalgebras of $\gr(H,\mathcal{F})$, so $C$ is a graded polynomial algebra over $k$ by Corollary \ref{polynomial-generator-graded}. Since $\gkdim A\geq \gkdim C$ by \cite[Lemma 6.5]{KL}, to see the result, one may assume $\gkdim C<\infty$. Then $C$ is finitely generated and $\gkdim C\in \mathbb{N}$ by \cite[Example 3.6]{KL}. Thus, $\gkdim A=\gkdim C\in \mathbb{N}$ by \cite[Proposition 6.6]{KL}.
\end{proof}

The following result tells us that some fundamental ring-theoretical and homological properties are enjoyed by coideal subalgebras of connected Hopf algebras of finite GK dimension over a field of characteristic zero. 
It generalizes  \cite[Theorem 4.3]{BG} from algebraically closed fields of characteristic zero to arbitrary fields of characteristic zero, and also generalizes \cite[Theorem 3.15]{ZSL3} from connected Hopf algebras to their  coideal subalgebras. The unexplained terminology used in the theorem is standard, and can be found for example in  \cite{BZ,  RRZ}.

\begin{proposition}\label{homological-property}
Assume that $k$ is of characteristic $0$. Let $H$ be a connected  Hopf algebra. Let $A$ be a one-sided coideal subalgebra of $H$ of finite GK dimension $d$. Then
	\begin{enumerate}
		\item $A$ is a finitely generated domain over $k$.
		\item $A$ is universally noetherian (i.e., $A\otimes R$ is noetherian for any noetherian algebra $R$).
		\item $A$ is Auslander regular, (GK-)Cohen-Macaulay and skew $d$-Calabi-Yau.
		\item $A$ is of global dimension $d$ and Krull dimension $\leq d$.
		\item $A$ is Artin-Schelter regular of dimension $d$ (as an augmented algebra by $(\varepsilon_H)|_A$) .
	\end{enumerate}
\end{proposition}

\begin{proof}
With Theorem \ref{GKdim-transfer} at hand, the proof is literally the same as that of \cite[Theorem 3.15]{ZSL3}, following the  idea  of Brown and Gilmartin in \cite[Theorem 4.3]{BG}.
\end{proof}

\begin{remark}
The Nakayama automorphism of the algebra $A$ as stated in Proposition \ref{homological-property} has a closed formula in terms of the antipode of $H$ and the character of $A$ determined by the homological integral of $A$. We refer to \cite[Proposition 4.6]{BG} and \cite[Theorem 3.6]{LiWu} for details.
\end{remark}

\begin{remark}
Due to Theorem \ref{GKdim-transfer}, one may define the invariants of signature and lantern for  coideal subalgebras of connected Hopf algebras over arbitrary fields of characteristic zero, as that have done in \cite[Section 5]{BG} with the additional assumption that the base field is algebraically closed. Note that the results in loc. cit. hold literally in this more general context.
\end{remark}

\begin{lemma}\label{gkdim-compare}
Assume that $k$ is of characteristic $0$. Let $H$ be a connected Hopf algebra. Let $A\supseteq B$ be  left (resp. right) coideal subalgebras of $H$. Let $(\{z_\xi\}_{\xi\in \Xi}, \vartriangleleft)$ be a system of PBW generators of $A$ over $B$ that satisfying  Theorem \ref{connected-graded-Hopf-algebra-PBW} (2, 4). Then $\gkdim A= \gkdim B+ \#(\Xi)$.
\end{lemma}

\begin{proof}
By Lemma \ref{PBW-GK-dimension}, one has $\gkdim A\geq  \gkdim B+ \#(\Xi)$. So to see the desired equality, one may assume $\gkdim B$ and $\#(\Xi)$ are finite.  By Proposition \ref{homological-property} (1), $B$  is finitely generated over $k$. The argument of Theorem \ref{IHOE-graded} then almost literally deduces the desired equality.
\end{proof}

\begin{proposition}\label{connected-Hopf-algebra-PBW}
Assume that  $k$ is of characteristic $0$. Let $H$ be a connected  Hopf algebra. Let $A\supseteq B$ be left (resp. right) coideal subalgebras of $H$. Then there exists a family $\{z_\xi\}_{\xi\in \Xi}$ of  elements in $A\cap \ker(\varepsilon_H)$ such that $(\{z_\xi\}_{\xi\in \Xi}, <)$ is a system of PBW generators of $A$ over $B$ for each total order $<$ on $\Xi$. Moreover, one may requires that  $\gkdim A= \gkdim B +\#(\Xi)$.
\end{proposition}

\begin{proof}
Let $\mathcal{F}$ be the coradical filtration of $H$. Note that $\gr(A,\mathcal{F}|_A) \supseteq \gr(B,\mathcal{F}|_B)$ are homogeneous left (resp. right) coideal subalgebras of $\gr(H,\mathcal{F})$. By Theorem \ref{connected-graded-Hopf-algebra-PBW} and Lemma \ref{gkdim-compare}, one may fix a family $\{z_\xi\}_{\xi\in \Xi}$ of  elements of $A\cap \ker(\varepsilon_H)$ such that $(\{\overline{z_\xi}\}_{\xi\in \Xi}, <)$ is a system of PBW generators of  $\gr(A,\mathcal{F}|_A)$ over $\gr(B,\mathcal{F}|_B)$ for each total order $<$ on $\Xi$, and moreover $\gkdim \gr(A,\mathcal{F}|_A) = \gkdim \gr(B,\mathcal{F}|_B) +\#(\Xi)$. The result then follows by Lemma \ref{PBW-tranfer} and Theorem \ref{GKdim-transfer}.
\end{proof}

\begin{corollary}\label{polynomial-generator}
	Assume that $k$ is of characteristic $0$. Let $H$ be a connected  Hopf algebra. Let $A\supseteq B$ be left (resp. right) coideal subalgebras of $H$.  If $A$ is commutative, then $A$ is isomorphic as an algebra to the polynomial algebra over $B$ in some family of variables. 
\end{corollary}

\begin{proof}
	It is a direct consequence of Proposition \ref{connected-Hopf-algebra-PBW}.
\end{proof}

Connected Hopf algebras over fields of positive characteristic is much more complicated than that over fields of characteristic zero. Particularly noteworthy is the fact that there are plenty of finite-dimensional examples in the positive characteristic setting, while there are no such examples other than the base field in the zero characteristic setting. It is well-known that  finite-dimensional connected Hopf algebras over a field of characteristic $p>0$ are of dimension powers of $p$, we extend this fact to their  coideal subalgebras in the following result.

\begin{proposition}
Assume that $k$ is of characteristic $p>0$. Let $H$ be a connected  Hopf algebra. Let $A$ be a finite-dimensional one-sided coideal subalgebra of $H$.  Then $\dim A= p^n$ for some $n\geq 0$.
\end{proposition}

\begin{proof}
Let $\mathcal{F}$ be the coradical filtration of $H$. Consider $\gr(H, \mathcal{F})$ as a connected $\mathbb{N}$-graded braided bialgebra with braiding the flipping map, i.e., the one induced by the trivial bicharacter of $\mathbb{N}$. Since $\gr(A,\mathcal{F}|_A) \supseteq k$ are homogeneous one-sided coideal subalgebras of $\gr(H, \mathcal{F})$, $\dim \gr(A, \mathcal{F}|_A)$ is a power of $p$ by Theorem \ref{braided-bialgebra-1} (1, 6). The result follows because $\dim A= \dim \gr(A, \mathcal{F}|_A)$.
\end{proof}

\section{Proof of Proposition \ref{PBW-diagonal}}

\label{section-proof-PBW-generator}

This section is devoted to prove Proposition \ref{PBW-diagonal}. It is divided into five separate Propositions below. We prove the  left case in full detail. The right case can be done similarly.  The arguments have benefited from some of the ideas in \cite{Kh,Uf,ZSL3}. Actually, the argument is modified from that of \cite[Proposition 2.4]{ZSL3}.
The notation and conventions employed in  Secton \ref{section-PBW-generator} are retained. 

Throughout,  $\Gamma=(\Gamma,<)$ is a nontrivial well-ordered  abelian monoid, $k\langle X\rangle$ is $\Gamma$-graded with each letter homogeneous of positive degree, $I$ is an ideal of $k\langle X\rangle $, $\tau$ is an $X$-diagonal  braiding on $kX$, $\Delta$ is a left bounded $\tau$-comultiplication on $k\langle X\rangle$ such that $\Delta(I) ~\subseteq~ I\otimes k\langle X\rangle + k\langle X\rangle \otimes I$, and $\rho:=\tau^{-1}$.

The following auxiliary result will be helpful. Recall that $\langle X\rangle^{\prec u}$ denotes the set of words that $\prec u$ and $\langle X\rangle^{\prec u}_\gamma:= \langle X\rangle^{\prec u}\cap \langle X\rangle_\gamma$ for all words $u\in \langle X\rangle$ and all index $\gamma\in \Gamma$.

\begin{lemma}\label{killing-functional}
For each pair of integers $l, n\geq1$ and each $I$-irreducible Lyndon word $u\in \N_I$, 
$$
k\langle X\rangle^{\prec u^n}_{\deg(u^l)}  \subseteq [k\mathcal{E}]_\rho + I \cap k\langle X \rangle_{\leq\deg(u^l)} +  k\langle X \rangle_{<\deg(u^l)},
$$	 
where $\mathcal{E}$ is the set of $I$-irreducible words on $X$ of $\deg(u^l)$ other than $u^l$.	
\end{lemma}
\begin{proof}
Let $\gamma=\deg(u^l)$. We show it by induction on words   $w\in \langle X\rangle^{\prec u^n}_{\gamma}$ with respect to $\prec$, which is exactly the restriction of the well order $<_{\glex}$ on $\langle X\rangle^{\prec u^n}_{\gamma}$. First note that $u^l\not\in \langle X\rangle^{\prec u^n}_{\gamma}$. Assume $w$ is the smallest word in $\langle X\rangle^{\prec u^n}_{\gamma}$. If $w$ is $I$-reducible then $w\in  I \cap k\langle X \rangle_{\leq \gamma} +  k\langle X \rangle_{<\gamma}$; if $w$ is $I$-irreducible then  $w=[w]_\rho$ by Lemma \ref{bracketing-diagonal} . So in both cases, one has $w\in [k\mathcal{E}]_\rho + I \cap k\langle X \rangle_{\leq \gamma} +  k\langle X \rangle_{<\gamma}$. 

For words $w\in \langle X\rangle^{\prec u^n}_{\gamma}$ other than the smallest one, there also have two situations. If $w$ is $I$-reducible, there exists an element $g_w\in I \cap k\langle X\rangle_{\leq \gamma}$ such that $\mathrm{LW}(g_w)=w$ and $w=g_w+(w-g_w)$; If $w$ is $I$-irreducible, then  $w=[w]_\rho+(w-[w]_\rho)$. Note that words occur in $w-g_w$ and $w-[w]_\rho$ that of degree $\gamma$ are all $\prec w$, so $w\in [k\mathcal{E}]_\rho + I \cap k\langle X \rangle_{\leq \gamma} +  k\langle X \rangle_{<\gamma}$ by  the induction hypothesis.	
\end{proof}

\begin{lemma}\label{comparing}
Let $\gamma\in \Gamma$ be an index,  $u\in \mathbb{L}(X)$ a Lyndon word on $X$ and  $i, N\geq 0$ two  integers. Assume that $\deg(u^i)<\gamma\leq \deg(u^N)$. Then
$ [k\langle X\rangle_{\gamma-\deg(u^i)}^{\vartriangleleft u}u^i]_\rho \subseteq k\langle X\rangle^{\prec u^N}_{\gamma}$.
\end{lemma}

\begin{proof}
It is easy to see that $k\langle X\rangle_{\gamma-\deg(u^i)}^{\vartriangleleft u}u^i \subseteq k\langle X\rangle^{\prec u^N}_{\gamma}$ by Proposition \ref{fact-Lyndon} (L7) and Lemma \ref{lex-order-1}. Also, One has $k\langle X\rangle^{\prec u^N}_{\gamma}=[k\langle X\rangle^{\prec u^N}_{\gamma}]_\rho$ by Lemma \ref{bracketing-diagonal}  and an easy induction argument on words of degree $\gamma$ with respect to $\prec$. The result then follows immediately.
\end{proof}

\begin{proposition}\label{rephrase-1}
Every	$I$-restricted word on $X$ is $I$-irreducible, that is $\C_I=\D_I$.
\end{proposition}
\begin{proof}
We may assume $I\neq k\langle X\rangle$. By  Lemma \ref{bracketing-diagonal}, Lemma \ref{basis-quotient-algebra} and that $\D_I\subseteq \C_I$, it suffices to show that the cosets of elements in $[\C_I]_\rho$ are linearly independent in $k\langle X\rangle/I$. For $\gamma\geq \Gamma$, let  $$\C_I^\gamma := \{~ w\in \C_I  ~| ~  \deg(w) =  \gamma ~\}$$  and $U_{\gamma}$ the subspace of $k\langle X \rangle_{\leq \gamma}$ spanned by $\{~ [w]_\rho  ~| ~ w\in \C_I,~ \deg(w)\leq \gamma ~\}$.  To see the result it suffices to show $U_{\gamma} \cap I = 0$ for each $\gamma\in \Gamma$.  We proceed by induction on $\gamma$. It is clear that $U_{ 0}\cap I=0$.
	
Assume $\gamma > 0 $. Suppose  $ U_{\gamma} \cap I$ has a nonzero element $g$. Then,
	\[
	g =  \sum_{w\in \C_I^\gamma} \lambda_w [w]_\rho +g', \quad \lambda_w\in k
	\]
with some $\lambda_w$ nonzero and $g'\in k\langle X\rangle_{<\gamma}$. Choose $u$ as the pseudo-lexicographically biggest Lyndon word occurring as a Lyndon atom of some words $w\in \C_I^\gamma$ with $\lambda_w\neq 0$. Then $u$ occurs in the Lyndon decompositions of words $w\in \C_I^\gamma$ with $\lambda_w\neq 0$ only at the end.  Let $l$ be the maximal number of occurrences of $u$ in a Lyndon decomposition of word $w\in \C_I^\gamma$ with $\lambda_w\neq 0$. Thus we can express  $g$ as:
	\[
	g =  \sum_{i=0}^l  \sum_{w_i\in P_i} a_{w_i} [w_iu^{i}]_\rho +g',
	\]
where $a_{w_i} := \lambda_{w_iu^{i}} $ and $$P_i = \{~w \in \C_I^{\gamma-\deg(u^i)} \cap \langle X\rangle^{\vartriangleleft u}~| ~ \lambda_{wu^i} \neq 0 ~\}.$$ 
Here $\gamma-\deg(u^i)$ denotes the unique index $\beta_i\in \Gamma$ such that $\beta_i+\deg(u^i)=\gamma$.
Note that $P_l \neq \emptyset. $
If  $\gamma= \deg(u^l)$, then $P_l=\{1\}$ and $\lw(f) =u^l$ by Proposition \ref{fact-Lyndon} (L7) and Lemma \ref{bracketing-diagonal} , which contradicts the assumption that $l<h_I(u)$. Thus we must have $\deg(u^l)<\gamma$, and  the words $w_l\in P_l$ are all of positive degree $\gamma-\deg(u^l)$. By Proposition \ref{comultiplication-diagonal} and Lemma \ref{comparing},
	\begin{eqnarray*}
		\Delta(g) &= & \sum_{i=0}^{l}  \sum_{w_i\in P_i} a_{w_i} ~ \Delta([w_iu^i]_\rho) +\Delta(g')  \\
		&\in& \sum_{i=0}^l  \sum_{w_i \in P_i}  \sum_{j=0}^{i} \binom{i}{j}_{\tau_{u,u}} a_{w_i} ~\tau_{w_i,u}^j ~  [u^j]_\rho \otimes [w_iu^{i-j}]_\rho
		+ \sum_{i=0}^l  \sum_{w_i \in P_i} a_{w_i}  [w_iu^i]_\rho\otimes 1 \\
		&&+ \Big (k\langle X \rangle_+^{\prec u^N} \otimes k\langle X\rangle_+\Big)_{\gamma} + \Big(k\langle X\rangle\otimes k\langle X\rangle\Big)_{<\gamma}, \quad  \text{for integers $N\gg l$}.
	\end{eqnarray*}
By Lemma \ref{basis-quotient-algebra}, one may define a linear functional $\phi: k\langle X\rangle \to k$  as follows:
\begin{itemize}
	\item $\phi(I) =0$, $\phi([w]_\rho)=0$ for $w\in \D_I \backslash \{ u^l \}$ and  $\phi([u^l]_\rho) =1$.
\end{itemize}
Note that $\phi(k\langle X\rangle_{<\deg(u^l)})=0$, by Lemma \ref{basis-quotient-algebra} again. So $\phi(k\langle X\rangle^{\prec u^N}_{\deg(u^l)}) =0$ by Lemma \ref{killing-functional}. Hence
	\[
	\tilde{g}:=(\phi\otimes\id) \Big(\Delta(f) \Big) ~ \in ~ \sum_{w_l\in P_l} a_{w_l}\tau_{w_l,u}^l ~ [w_l]_\rho + k\langle X\rangle_{<\gamma-\deg(u^l)}.
	\]
Clearly, $\lw(\tilde{g}) \in P_l$ by Lemma \ref{bracketing-diagonal}, which is $I$-irreducible by induction,  so $\tilde{g} \not\in I$. But $\Delta(g)\in I \otimes  k\langle X\rangle+k\langle X\rangle \otimes I$. It  follows that $\tilde{g} \in I$, a contradiction.
\end{proof}

\begin{proposition} \label{rephrase-2}
	For each Lyndon word $u$ on $X$ of finite height $n\geq 1$, $$[u]^n_\rho\in [k\langle X|I\rangle^{\vartriangleleft u}_{\deg(u^n)}]_\rho + k\langle X\rangle_{<\deg(u^n)} +I .$$
\end{proposition}
\begin{proof}
By assumption,  there is a nonzero monic polynomial $g\in I \cap k\langle X\rangle_{\leq \deg(u^n)}$ with $\lw(g) =u^n$. If $g-[u^n]_\rho\in k\langle X\rangle_{<\deg(u^n)} + I $ then we are done. So we may assume $g- [u^n]_\rho\not\in k\langle X\rangle_{<\deg(u^n)} + I$. By Lemma \ref{bracketing-diagonal}  and Lemma \ref{basis-quotient-algebra},  it is easy to see that $g - [u^n]_\rho$ has an expression of the form
\[
g - [u^n]_\rho = \sum_{i=0}^l  \sum_{w_i\in Q_i} a_{w_i} [w_iv^{i}]_\rho + g'+ g'' , \quad  a_{w_i} \in k,
\]
where $g'\in k\langle X\rangle_{<\deg(u^n)}$, $g''\in k\langle X\rangle_{\leq\deg(u^n)}\cap I$ is zero or with leading word $<_{\lex } u^n$,   $v\in \N_I$, $l\geq 1$ and
$$Q_i:= \{~ w \in \D_I \cap \langle X\rangle^{\vartriangleleft v}_{\deg(u^n)- \deg(v^i)} ~| ~ wv^i\in \D_I, \, wv^i \prec u^n~\}, \quad i=0,\ldots, l.$$
with $Q_l\neq \emptyset$ and  $a_{w_l}\neq 0$ for some $w_l\in Q_{l}$. Note that $l< h_I(v)$. 

First consider the case that $\deg(v^l) = \deg(u^n)$. According to the construction, $v^l<_{\lex} u^n$. Then by Proposition \ref{fact-Lyndon} (L7), one has $v<_{\lex} u$ and hence the desired formula holds.

Next we proceed to deal with the case that $\deg(v^l) < \deg(u^n)$ or equivalently  $1\not\in Q_l$. It suffices to show $v<_{\lex}u$. Assume not. By Proposition \ref{comultiplication-diagonal} and Lemma \ref{comparing},
	\begin{eqnarray*}
		\Delta(g-g'') &=& \Delta([u^n]_\rho) + \sum_{i=0}^l   \sum_{w_i\in Q_i} a_{w_i} ~ \Delta([w_iv^{i}]_\rho)+ \Delta(g') \\
		&\in& \sum_{i=0}^n \binom{n}{i}_{\tau_{u,u}} ~  [u^i]_\rho \otimes [u^{n-i}]_\rho
		+  \Big (k\langle X \rangle_+^{\prec u^n} \otimes k\langle X\rangle_+\Big)_{\deg(u^n)} \\
		&& + \sum_{i=0}^l  \sum_{w_i \in Q_i}  \sum_{j=0}^{i} \binom{i}{j}_{\tau_{v,v}} a_{w_i} ~\tau_{w_i,v}^j ~  [v^j]_\rho \otimes [w_iv^{i-j}]_\rho
		+ \sum_{i=0}^l  \sum_{w_i \in Q_i}  a_{w_i} [w_iv^i]_\rho\otimes 1 \\
		&&+ \Big (k\langle X \rangle_+^{\prec v^N} \otimes k\langle X\rangle_+\Big)_{\deg(u^n)} + \Big (k\langle X\rangle \otimes k\langle X\rangle\Big)_{<\deg(u^n)}, \quad \text{for integers $N\gg l$}.
	\end{eqnarray*}
By Lemma \ref{basis-quotient-algebra}, one may define a linear functional $\phi: k\langle X\rangle \to k$  as follows:
\begin{itemize}
	\item $\phi(I) =0$, $\phi([w]_\rho)=0$ for $w\in \D_I \backslash \{ v^l \}$ and  $\phi([v^l]_\rho) =1$.
\end{itemize}
Note that $\phi(k\langle X\rangle_{<\deg(v^l)})=0$, by Lemma \ref{basis-quotient-algebra} again. Next we claim that $$\langle X \rangle^{\prec u^n}_{\deg(v^l)} \subseteq \langle X \rangle^{\prec v^N}_{\deg(v^l)}.$$ 
Indeed, if $w\in \langle X \rangle^{\prec u^n}_{\deg(v^l)}$ then its first Lyndon atom $w_1$ satisfies that  $w_1<_{\lex} u\leq_{\lex} v$ by Lemma \ref{lex-order-2}. So $w<_{\lex} v^N$ by Proposition \ref{fact-Lyndon} (L7). Since $\deg(w) =\deg(v^l) < \deg(v^N)$, one has $w\prec v^N$ by Lemma \ref{lex-order-1}. So $\phi(k\langle X \rangle^{\prec u^n}_{\deg(v^l)}) =\phi(k\langle X \rangle^{\prec v^N}_{\deg(v^l)}) =0$ by Lemma \ref{killing-functional}. 
	
Now	we break the discussion into two cases. If $u<_{\lex} v$ then 
	\begin{equation*}
		\tilde{g}:=	(\phi \otimes \id) \Big(\Delta(g-g'') \Big) ~ \in ~ \sum_{w_l\in Q_l} a_{w_l} \rho_{w_l,v}^l ~ [w_l]_\rho + k\langle X\rangle_{<\deg(u^n)-\deg(v^l)};
	\end{equation*}
	and if $u=v$ then  $n = h_I(u)=h_I(v)>l$, and hence one have  that
	\[
	\tilde{g}:=(\phi\otimes \id) \Big(\Delta(g-g'') \Big) ~ \in ~
	\binom{n}{l}_{\tau_{u,u}} ~ [u^{n-l}]_\rho + \sum_{w_l\in Q_l} a_{w_l} \rho_{w_l,v}^l ~  [w_l]_\rho + k\langle X\rangle_{<\deg(u^n)-\deg(v^l)}.
	\]
	In both cases, the leading word of $\tilde{g}$ is in $ Q_l\cup\{u^{n-l}\} $ by Lemma \ref{bracketing-diagonal} , so $\tilde{g} \not\in I$.  But $g-g''\in I$, so $\Delta(g-g'')\in I\otimes k\langle X\rangle+k\langle X\rangle \otimes I$. Therefore, $\tilde{g} \in I$, a contradiction.
\end{proof}

\begin{proposition}\label{rephrase-3}
 For each pair of $I$-irreducible Lyndon words  $u,v\in \N_I$ with $u>_{\lex} v$,
\[
[[u]_\rho,[v]_\rho]_\rho ~\in~ [k\langle X|I\rangle_{\deg(uv)}^{\trianglelefteq uv}]_\rho +k\langle X\rangle_{<\deg(uv)} + I.
\]
\end{proposition} 

\begin{proof}
By Lemma \ref{bracketing-diagonal-more} (2), one has $[[u]_\rho,[v]_\rho]_\rho ~\in~ [k\langle X\rangle_{\deg(uv)}^{\trianglelefteq uv}]_\rho $. By Proposition \ref{rephrase-1}, Proposition \ref{rephrase-2} and Lemma \ref{rearragement}, one has $[k\langle X\rangle_{\deg(uv)}^{\trianglelefteq uv}]_\rho \subseteq [k\langle X|I\rangle_{\deg(uv)}^{\trianglelefteq uv}]_\rho +k\langle X\rangle_{<\deg(uv)} + I.$ The result follows.
\end{proof}

\begin{proposition}\label{rephrase-4}
For each  total order  $<$ on $\N_I$, the triple   $(\{[u]_\rho+I\}_{u\in \N_I}, h_I, <)$ is a system of PBW generators of the quotient algebra $R:=k\langle X\rangle/ I$.
\end{proposition}

\begin{proof}
It is  a direct consequence of Proposition \ref{rephrase-1}, Proposition \ref{rephrase-2} and Proposition \ref{PBW-generator}.
\end{proof}

\begin{proposition}\label{rephrase-5}
For each Lyndon word $u\in \N_I$ with $h_I(u) <\infty$, the scalar $\tau_{u,u}$  is a root of unity. More precisely,  if  $k$ is of characteristic $0$ then $\tau_{u,u}$ is of order $h_I(u)$, and in particular, $\tau_{u,u}\neq 1$; and if $k$ is of characteristic $p>0$ then  $\tau_{u,u} $ is of order $h_I(u)/p^s$ for some integer $s\geq 0$.
\end{proposition}

\begin{proof}
	Let $n:=h_I(u)$, which is $\geq 2$. By Proposition \ref{rephrase-2}, there is an element $g\in I\backslash \{0\}$ of the form
	\[
	g =  [u^n]_\rho + \sum_{w\in Q} a_{w} [w]_\rho +g', \quad a_{w} \in k,
	\]
	with $Q:= \langle X|I\rangle^{\vartriangleleft u}_{\deg(u^n)}=\D_I\cap \langle X\rangle_{\deg(u^n)}^{\vartriangleleft u}$ and $ g'\in k\langle X\rangle_{<\deg(u^n)}$. Note that $\lw(g)=u^n$ by Lemma \ref{bracketing-diagonal}. Apply Proposition \ref{comultiplication-diagonal} and Lemma \ref{comparing},
	\begin{eqnarray*}
		\Delta(g) &=& \Delta([u^n]_\rho) + \sum_{w\in Q} a_{w} [w]_\rho +\Delta(g')\\
		&\in& \sum_{i=0}^n \binom{n}{i}_{\tau_{u,u}} ~  [u^i]_\rho \otimes [u^{n-i}]_\rho
		+ \sum_{w\in Q} a_w (1\otimes [w]_\rho +[w]_\rho\otimes 1) \\
		&& + \Big (k\langle X \rangle_+^{\prec u^n} \otimes k\langle X\rangle_+\Big)_{\deg(u^n)} + \Big (k\langle X\rangle \otimes k\langle X\rangle\Big)_{<\deg(u^n)}.
	\end{eqnarray*}
	By Lemma \ref{basis-quotient-algebra},  one may define a linear functional  $\phi: k\langle X\rangle \to k$ as follows.
	\begin{itemize}
		\item $\phi(I) = 0$, $\phi([w]_\rho) =0$ for $w \in \D_I \backslash \{ u \}$ and $\phi([u]_\rho) =1$.
	\end{itemize}
	Note that $\phi(k\langle X \rangle_{<\deg(u)}) =0$. So $\phi(k\langle X \rangle^{\prec u^n}_{\deg(u)}) =0$ by Lemma \ref{killing-functional}. Hence, 
	\begin{equation*}\label{equality-formula}
		\tilde{g}:= (\phi\otimes \id )(\Delta(g)) ~\in~	\binom{n}{1}_{\tau_{u,u}}[u^{n-1}]_\rho + k\langle X\rangle_{<\deg(u^{n-1})}. 
	\end{equation*}
	Since $\Delta(g)\in I \otimes  k\langle X\rangle  +  k\langle X\rangle \otimes I$, $\tilde{g}\in I$. But $\lw([u^{n-1}]_\rho) =u^{n-1}$ is $I$-irreducible. So $$1+\tau_{u,u}+\cdots +\tau_{u,u}^{n-1}=\binom{n}{1}_{\tau_{u,u}}= 0$$
	in $k$ and it follows that $\tau_{u,u}$ is a root of unity, say of order $t$, which divides $n$.
	
	If $k$ is of characteristic $0$, then clearly $\tau_{u,u}\neq 1$, in this case we write $n=tq^sa$ with $q=s=1$; if $k$ is of positive characteristic $p$, then we write $n=tq^sa$ with $a$ doesn't divisible by $q:=p$. It remains to show $a=1$. Assume not, let $\psi: k\langle X\rangle \to k$ be the functional given by 
	\begin{itemize}
		\item $\psi(I)   =0$, $\psi([w]_\rho) =0$ for $w \in \D_I \backslash \{u^{tq^s} \}$ and $\psi([u^{tq^s}]_\rho) =1$.
	\end{itemize}
	Note that $\psi(k\langle X \rangle_{<\deg(u^{tq^s})})=0$ and thereof $\psi(k\langle X \rangle^{\prec u^n}_{\deg(u^{tq^s})}) =0$ by Lemma \ref{killing-functional}. So 
	\[
	\hat{g}:=(\psi\otimes \id)(\Delta(g)) ~\in~\binom{tq^sa}{tq^s}_{\tau_{u,u}}[u^{tq^s(a-1)}]_\rho+  k\langle X\rangle_{<\deg(u^{tq^s(a-1)})}.
	\]
	Since $\hat{g} \in I$ but $\lw([u^{tq^s(a-1)}]_\rho) = u^{tq^s(a-1)}$ is $I$-irreducible, one has
	\[
	a=\binom{a}{1} = \binom{q^sa}{q^s} =  \binom{tq^sa}{tq^s}_{\tau_{u,u}} =0,
	\]
	in $k$,	which is absurd. Here, the third equality is because $\tau_{u,u}$ is of order $t$. 
\end{proof}	

\section{Proof of Proposition \ref{braided-bialgebra-3}}
\label{section-proof-braided-bialgebra-3}

This section is devoted to prove Proposition \ref{braided-bialgebra-3} by the filtered-graded method. The argument has benefited from some of the ideas in the proof of \cite[Proposition 7.3]{Kh1}.

Let $\Gamma=(\Gamma, <)$ be a well-ordered abelian monoid. Recall that a \emph{braided algebra (resp. coalgebra, bialgebra) $\Gamma$-filtration} of a braided algebra (resp. coalgebra, bialgebra) $R$ is by definition an algebra (resp. coalgebra, algebra and coalgebra) $\Gamma$-filtration $\mathcal{F}=(F_\gamma(R))_{\gamma\in \Gamma}$ such that
$$\tau(F_\alpha(R) \otimes F_\beta(R)) = F_\beta(R) \otimes F_\alpha(R), \quad \alpha, \beta\in \Gamma.$$
In these cases,  $\gr(R,\mathcal{F})$ is naturally a $\Gamma$-graded braided algebra (resp. coalgebra, bialgebra).

\begin{lemma}\label{filtration-braided}
Let $R$ be a connected $\Gamma$-graded braided bialgebra, where $\Gamma$ is a nontrivial abelian monoid that has admissible well orders. Assume $R$ is generated by a homogeneous braided subspace of categorical type. Then there is a  well-ordered abelian monoid $\mathbb{M}$, a bicharacter $\chi$ of $\mathbb{M}\times\Gamma $ and  a braided bialgebra $\mathbb{M}$-filtration $\mathcal{F}=(F_{\alpha}(R))_{\alpha\in \mathbb{M}}$ of $R$ such that $F_{\alpha}(R)$ are all homogeneous and 
$$\gr(R,\mathcal{F})=\bigoplus_{(\alpha,\gamma)\in \mathbb{M}\times\Gamma }\big(F_\alpha(R)/F_\alpha^-(R)\big)_\gamma$$ is a connected $(\mathbb{M}\times\Gamma)$-graded braided bialgebra with braiding  $\tau_\chi$, the one induced by $\chi$.
\end{lemma}

\begin{proof}
	Let $V\subseteq R_+$ be a homogeneous braided subspace of $R$  which is of categorical type and generates $R$. By Lemma \ref{homogenous-categorical}, one may choose a well-ordered homogeneous basis  $X$ of $V$  so that for all $x,y\in X$, there exists a unique nonzero scalar $p_{x,y}\in k$ such that
	\begin{eqnarray*}
		\tau_V(x\otimes y) \in  p_{x,y} y\otimes x+\sum_{z_1\in X,~ z_1< y} k\cdot z_1\otimes x + \sum_{z_2\in X,~ z_2 < x} k\cdot y\otimes z_2  + \sum_{z_1,\, z_2\in X, ~ z_1<y,\, z_2<x} k\cdot z_1\otimes z_2,
	\end{eqnarray*}
	where $\tau_V$ denotes the induced braiding on $kX=V$. Since $\tau_V(V_\gamma\otimes V_\delta) = V_\delta\otimes V_\gamma$ for all $\gamma, \delta\in \Gamma$, one may assume $\deg(z_1)=\deg(y)$ and $\deg(z_2)=\deg(x)$ in the above formula. 
	
	Let $\mathbb{M}$ be the free abelian monoid on $X$ with the canonical basis denoted by $\{e_x\}_{x\in X}$. Every element of $\mathbb{M}$ has a unique expression of the form 
	$\sum_{x\in X}m_x e_x$, where $m_x$ are nonnegative integers and they are nonzero for only finitely many $x$. 
	Equip $\mathbb{M}$ with the admissible well order $<_{\mathbb{M}}$ defined as follows. First fix  an admissible well order $<_\Gamma$ on $\Gamma$; then for $\alpha= \sum_{x\in X}m_x e_x, ~ \beta= \sum_{x\in X}n_x e_x \in \mathbb{M}$,
	\begin{eqnarray*}
		\alpha<_{\mathbb{M}} \beta ~ \Longleftrightarrow  ~  
		\exists ~ y\in X \text{ such that }\left\{
		\begin{array}{llll}
		m_y<n_y  \text{ and }  m_x=n_x \text{ for all } x\in X \text{ with }&& \\
		\deg(y) <_\Gamma \deg(x) \text{ or } \deg(y) = \deg(x) \text { but } y<x.
		\end{array}\right.
	\end{eqnarray*}
	Let $D: \langle X\rangle \to \mathbb{M}$ be the monoid map given by  $D(x)= e_x$ for all $x\in X$.
	
	Consider the free algebra $k\langle X\rangle$ as a $\Gamma$-graded braided algebra with braiding the one induced by $\tau_V$.
	Let $\tilde{\mathcal{F}} =(\tilde{F}_{\alpha}(k\langle X\rangle))_{\alpha\in \mathbb{M}}$ be the  $\mathbb{M}$-filtration on $k\langle X\rangle$ given by
	\begin{eqnarray*}
		\tilde{F}_{\alpha}(k\langle X\rangle) := \sum_{u\in \langle X\rangle,~ D(u)\leq \alpha} k u.
	\end{eqnarray*}
	It is easy to see that $\tilde{\mathcal{F}} $ is an algebra $\mathbb{M}$-filtration of $k\langle X\rangle$ and $\tilde{F}_\alpha(k\langle X\rangle)$ are all homogeneous.  
	
	Let $\chi$ be the bicharacter of $\mathbb{M}\times\Gamma$ given by 
	\[
	\chi((e_x,\gamma), (e_y, \delta) )=p_{x,y}, \quad \gamma,\, \delta \in \Gamma, ~~ x,\, y\in X.
	\]
	Then for all words $u, v\in \langle X\rangle$,  one has
	\begin{eqnarray*}\label{key-bicharcater}
		\tau_{V}(u\otimes v) &\in& \chi((D(u),\deg(u)), (D(v),\deg(v)))\cdot v\otimes u \\ && + ~ \tilde{F}_{D(v)}(k\langle X\rangle)_{\deg(v)} \otimes \tilde{F}_{D(u)}^-(k\langle X\rangle)_{\deg(u)} + \tilde{F}_{D(v)}^-(k\langle X\rangle)_{\deg(v)} \otimes \tilde{F}_{D(u)}(k\langle X\rangle)_{\deg(u)}
	\end{eqnarray*}
	by an  evident induction on the length of $u$ and $v$. So
	$$\tau_V(\tilde{F}_{\alpha}(k\langle X\rangle)_\gamma\otimes \tilde{F}_{\beta}(k\langle X\rangle)_\delta) \subseteq  \tilde{F}_{\beta}(k\langle X\rangle)_\delta \otimes \tilde{F}_{\alpha}(k\langle X\rangle)_\gamma, \quad (\alpha,\gamma), ~(\beta,\delta) \in \mathbb{M}\times\Gamma.$$ 
	Similarly, one has $\tau_V^{-1} (\tilde{F}_{\beta}(k\langle X\rangle)_\delta \otimes \tilde{F}_{\alpha}(k\langle X\rangle)_\gamma) \subseteq \tilde{F}_{\alpha}(k\langle X\rangle)_\gamma\otimes \tilde{F}_{\beta}(k\langle X\rangle)_\delta$. Hence,
	\[
	\tau_V(\tilde{F}_{\alpha}(k\langle X\rangle)_\gamma\otimes \tilde{F}_{\beta}(k\langle X\rangle)_\delta) =  \tilde{F}_{\beta}(k\langle X\rangle)_\delta \otimes \tilde{F}_{\alpha}(k\langle X\rangle)_\gamma, \quad (\alpha,\gamma), ~(\beta,\delta) \in \mathbb{M}\times\Gamma.
	\]
	Therefore, $\tilde{\mathcal{F}}$ is a braided algebra $\mathbb{M}$-filtration of $k\langle X\rangle$. In addition,   $\gr(k\langle X\rangle, \tilde{\mathcal{F}})$ is readily a connected $(\mathbb{M}\times \Gamma)$-graded braided algebra with braiding the one induced by $\chi$.
	
	Note that the canonical map $\pi: (k\langle X\rangle, \tau_V) \to (R, \tau)$ is a surjective homomorphism of $\Gamma$-graded braided algebras. Let $\mathcal{F}:=(F_{\alpha}(R))_{\alpha\in \mathbb{M}}$ be the $\mathbb{M}$-filtration of $R$ given by $$F_{\alpha}(R):=\pi(\tilde{F}_{\alpha}(k\langle X\rangle)).$$
	Clearly, $\mathcal{F}$ is a braided algebra $\mathbb{M}$-filtration of $R$ and $F_\alpha(R)$ are all homogeneous. In addition, $\gr(R,\mathcal{F})$ is a connected $(\mathbb{M}\times \Gamma)$-graded braided algebra with braiding $\tau_\chi$, since the canonical map 
	$$\gr(\pi):\gr(k\langle X\rangle, \tilde{\mathcal{F}}) \to \gr(R,\mathcal{F})$$  induced from $\pi$ is a surjective homomorphism of $(\mathbb{M}\times \Gamma)$-graded braided algebras. By the counitity of $R$,  there is  a homomorphism of $\Gamma$-graded algebras $\Delta: k\langle X\rangle\to k\langle X\rangle \otimes^{\tau_V} k\langle X\rangle$ such that 
	$$\Delta(x)  ~ \in ~  1\otimes x + x\otimes 1 +   \big(k\langle X \rangle_+ \otimes k\langle X \rangle_+\big)_{\deg(x)}, \quad x\in X$$
	and the following diagram commutes
	\begin{eqnarray*}
		\xymatrix{
			k\langle X\rangle \ar[r]^-{\Delta} \ar[d]_\pi & k\langle X\rangle\otimes^{\tau_V}k\langle X\rangle \ar[d]^{\pi\otimes \pi} \\
			R\ar[r]^-{\Delta_R} &  R\otimes^{\tau} R.
		}
	\end{eqnarray*}
	For all words $u\in \langle X\rangle$, one has 
	\[
	\Delta(u) \in \sum_{\beta_1,\, \beta_2\in \mathbb{M}, ~\beta_1+\beta_2 \leq D(u)} \Big(\tilde{F}_{\beta_1}(k\langle X\rangle) \otimes \tilde{F}_{\beta_2}(k\langle X\rangle)\Big)_{\deg(u)}
	\]
	by an easy induction on the length of $u$. It follows that 
	\[
	\Delta(\tilde{F}_\alpha(k\langle X\rangle)) \subseteq \sum_{\beta_1,\, \beta_2\in \mathbb{M}, ~\beta_1+\beta_2 \leq \alpha} \tilde{F}_{\beta_1}(k\langle X\rangle) \otimes \tilde{F}_{\beta_2}(k\langle X\rangle), \quad \alpha \in \mathbb{M}.
	\]
	Thus, by the commutative square given above,
	$$\Delta(F_\alpha(R)) \subseteq \sum_{\beta_1,\, \beta_2\in \mathbb{M}, ~\beta_1+\beta_2 \leq \alpha} F_{\beta_1}(R) \otimes F_{\beta_2}(R), \quad \alpha \in \mathbb{M}.$$
	So $\mathcal{F}$ is also a coalgebra $\mathbb{M}$-filtration of $R$ and hence the result follows.
\end{proof}

\begin{proof}[Proof of Proposition \ref{braided-bialgebra-3}]
Fix a well-ordered abelian monoid $\mathbb{M}$, a bicharacter $\chi$ of $\mathbb{M}\times\Gamma $ and  a braided bialgebra $\mathbb{M}$-filtration $\mathcal{F}=(\mathcal{F}_{\alpha}(R))_{\alpha\in \mathbb{M}}$ of $R$ as that in Lemma \ref{filtration-braided}. Clearly, $\mathbb{M}\times \Gamma$ has admissible well orders. Note that $\gr(A,\mathcal{F}|_A) \supseteq \gr(B,\mathcal{F}|_B)$ are homogeneous left (resp. right) coideal subalgebras of $\gr(R,\mathcal{F})$. The  result then follows from Lemma \ref{PBW-tranfer} and Theorem \ref{braided-bialgebra-1} (1).	
\end{proof}

\appendix

\section{Lyndon ideals of free algebras}

In this appendix, we introduce the class of Lyndon ideals for free algebras, study the combinatorial and homological properties of the corresponding quotient algebras, and relate them to the class of Artin-Schelter regular algebras. Partial results of this appendix have been published in Chinese version \cite{Zhou}. Gateva-Ivanova has been obtained similar results in \cite{Ga}.  However, compare to her argument, which takes advantage of a deep result of \cite{Ga0}, ours is  more direct and concrete. 


Throughout, let $\Gamma=\mathbb{N}^\theta$ for some integer $\theta\geq 1$ and equip it with an admissible well order. Let $|\cdot|: \Gamma\to \mathbb{N}$ be the monoid homomorphism given by $(\gamma_1,\ldots,\gamma_\theta)\mapsto \gamma_1+\cdots +\gamma_\theta$. We assume that the free algebra $k\langle X\rangle$ is  $\Gamma$-graded with each letter homogeneous of positive degree. 

For an ideal $I$ of $k\langle X\rangle$, we denote by $I^*$ the \emph{homogeneous closure} of $I$. It is the smallest homogeneous ideal of $k\langle X\rangle $ that contains $I$. It is not hard to see that $I^*$ is spanned by the set of highest nonzero homogeneous components of elements of $I$, and moreover
\[
\B_{I^*}=\B_I, \quad \C_{I^*} =\C_I, \quad  \D_{I^*} =\D_I, \quad \N_{I^*} = \N_I,   \quad \text{and} \quad \OO_{I^*} = \OO_I.
\]

\begin{definition}\label{definition-Lyndon-ideal}
An ideal $I$ of $k\langle X\rangle$ is called \emph{Lyndon} if $\OO_I$ consists of Lyndon words, or equivalently  
$\B_I=\D_I$. Clearly, $I$ is Lyndon if and only if its homogenization $I^*$ is so.
\end{definition}

A $\Gamma$-graded vector space $V=\bigoplus_{\gamma\in \Gamma}V_\gamma$ is called \emph{locally finite} if $\dim V_\gamma <\infty$ for each $\gamma\in \Gamma$. In this case, the \emph{Hilbert series of $V$} is defined to be the power series $${\rm H}_V(\textbf{\rm t}) = \sum_{\gamma\in \Gamma} \dim(V_\gamma) ~ \textbf{\rm t}^\gamma \in \mathbb{Z}[[t_1,\ldots, t_\theta]],$$ where $\textbf{\rm t}^\gamma$ is an abbreviation of $t_1^{\gamma_1}\cdots t_\theta^{\gamma_\theta}$ for $\gamma=(\gamma_1,\ldots, \gamma_\theta)$.

\begin{proposition}\label{Lyndon-Hilbert-series}
	Let $I$ be a homogeneous Lyndon  ideal of $k\langle X\rangle$.  Then $A= k\langle X\rangle/ I$ is locally finite if and only if $\N_I$ is locally finite; and in this case the Hilbert series of $A$ is $${\rm H}_A(t) =  \prod_{u \in \N_I} \big (1-\textbf{\rm t}^{\deg (u)}  \big )^{-1}.$$
\end{proposition}

\begin{proof}
	It is well-known that $\dim A_\gamma$ equals the number of $I$-irreducible words of degree $\gamma$, which is $\#\{ u\in \B_I ~|~ \deg(u)=\gamma \}$ for each $\gamma\in \Gamma$. The result follows by an easy combinatoric argument.
\end{proof}

\begin{lemma}\label{Homogenization}
Let $I$ be an ideal of $k\langle X\rangle$ and let $A= k\langle X\rangle/ I$.  Let $\F:=(F_\gamma (A))_{\gamma\in \Gamma}$ be the $\Gamma$-filtration of $A$ given by $F_\gamma (A):= (k\langle X\rangle_{\leq \gamma} +I) /I$. Then $\gr(A, \mathcal{F}) \cong k\langle X\rangle/ I^*$ as $\Gamma$-graded algebras.
\end{lemma}

\begin{proof}
	Let  $\phi: k\langle X\rangle \to \gr(A, \mathcal{F})$ be the $\Gamma$-graded algebra homomorphism given by
	$$\phi(x)= (x+I) + F^-_{\deg(x)}(A) \in \gr(A, \mathcal{F})_{\deg(x)}, \quad x\in X.$$
	It is easy to check that $\phi$ is a surjective homomorphism  with kernel contains $I^*$. Let $\widetilde{\phi}: k\langle X\rangle/ I^* \to \gr(A, \mathcal{F})$ be the map induced by $\phi$. By  standard Gr\"{o}bner basis theory,  $\{u+I^*~|~ u\in \D_I \cap \langle X\rangle_\gamma \}$ is a basis of  $(k\langle X \rangle /I^*)_\gamma$ and $\{~ (u+I) +F^-_\gamma(A) ~|~ u\in \D_I \cap \langle X\rangle_\gamma~\}$ is a basis of $\gr(A, \mathcal{F})_\gamma$, for each  $\gamma \in \Gamma$. Since $\widetilde{\phi}$ sends $u+I^*$ to $(u+I) +F^-_\gamma(A)$ for any $u\in \D_I \cap \langle X\rangle_\gamma$, the restriction of $\widetilde{\phi}$ on the $\gamma$-th component  is an isomorphism. Since $\gamma$ is arbitrary, the result follows.
\end{proof}

\begin{proposition}\label{Lyndon-GKdim}
Let $I$ be a Lyndon ideal of $k\langle X\rangle$ and let $A= k\langle X\rangle/ I$. Then $\gkdim A = \#(\N_I).$
In particular, $\gkdim A$ is either infinite or a natural number.
\end{proposition}

\begin{proof}
First note that $(\{u+I\}_{u\in \N_I}, <_{\lex})$ is a system of PBW generators of $A$. By Lemma \ref{PBW-GK-dimension}, one has $\gkdim A \geq \#(\N_I)$. So to see the desired equality, one may assume $\N_I$ is  finite. 

First we show the case that $I$ is homogeneous. By changing of grading along the map $|\cdot|:\Gamma \to \mathbb{N}$, naturally one has an $\mathbb{N}$-graded algebra  $A^{\rm tot}=\bigoplus_{n\in \mathbb{N}}(A^{\rm tot})_n$,  which is $A$ as an algebra, and where $(A^{\rm tot})_n:=\sum_{\gamma\in \Gamma, ~|\gamma|=n} A_\gamma$. Note that 
$H_{A^{\rm tot}} (t) =\prod_{u \in \N_I} \big (1-t^{|\deg (u)|}\big )^{-1}$. Then
\begin{equation*}\label{Hilbert-GKdim}
\gkdim A = \gkdim A^{\rm tot} = \overline{\lim_{n\to \infty}} \log_n \dim ((A^{\rm tot})_{\leq n}) = \#(\N_I). 
\end{equation*}
Here, the last two equalities are by \cite[Lemma 6.1 (b)]{KL} and  \cite[Proposition 2.21]{ATV2}, respectively.
	
Now we show the general case. Let $\F$ be the $\Gamma$-filtration of $A$  given in Lemma \ref{Homogenization}. By Lemma \ref{Homogenization} and the proved claim above,   $\gr(A, \F) \cong k\langle X\rangle/I^*$ is finitely generated and 
$$\gkdim \gr(A, \F) = \gkdim k\langle X\rangle/I^* = \#(\N_{I^*}) = \#(\N_I).$$
Therefore, $\gkdim A = \gkdim \gr(A, \F) = \#(\N_I)$ by \cite[Proposition 6.6]{KL}.
\end{proof}

To detect the homological properties of $K\langle X\rangle/ I$ with $I$ Lyndon, we need some preparations.

\begin{lemma}\label{Lemma-Shirshov-factor-dichotomy}
	Let $u\in \langle X\rangle$ be a Lyndon word of length $\geq 2$. If $v$ is a Lyndon factor of $u$, then $v$ is either a factor of $u_L$, or a factor of $u_R$, or a prefix of $u$ with $|v|> |u_L|$. 
\end{lemma}

\begin{proof}
	We only need to exclude the possibility that there are factorizations $u_L=pl,\, v=lr,\, u_R=rq$ with $p,l,r\neq 1$. Otherwise, one would have a proper Lyndon suffix $lu_R$ of $u$ by Proposition \ref{fact-Lyndon} (L1), which contradicts the assumption that $u_R$ is the longest proper Lyndon suffix of $u$.
\end{proof}

We call a set of words on $X$ an \emph{antichain} if it contains no  proper factors of any of its elements, and call a set of Lyndon words on $X$  \emph{closed} if  it  contains all Lyndon factors of any of its elements. 
For a closed set $U$ and an antichain $V$ of Lyndon words on $X$, let
\begin{align*} 
	\Phi_{\mathbb{L}}(U) &:= \left \{\, v\in \L\backslash U ~ | ~ \text{every proper Lyndon factors of $v$ belongs to $U$} \, \right \}  \\
	\Psi_{\mathbb{L}}(V) &:= \{\ u\in \L ~|~V \text{ contains no  factor of } u\ \}.
\end{align*}
Clearly, $\Phi_{\mathbb{L}}(U)$ is an antichain and $\Psi_{\mathbb{L}}(V)$ is closed, and
\begin{equation*}\label{minimal-closed-correspondence}
	\Phi_\mathbb{L} \big (\Psi_\mathbb{L}(V)\big) =V  \quad \text{ and } \quad \Psi_\mathbb{L} \big (\Phi_\mathbb{L}(U)\big) =U. 
\end{equation*}
Note that for any ideal $I$ of $k\langle X\rangle$,  $\N_I$ is a closed set of Lyndon words and $\OO_I$ is an antichain of words. Moreover, one has $\Phi_\mathbb{L}(\N_I)= \OO_I\cap\mathbb{L}$ and $\Psi_\mathbb{L}(\OO_I\cap\mathbb{L}) =\N_I$.

\begin{lemma}\label{lemma-bound-Psi-construction}
	Let $U$ be a closed set of  Lyndon words on $X$. Then for each pair $u,v\in U$ with $u>_{\lex} v$, if there's no $w\in U$ such that $u >_{\lex}w>_{\lex} v$ then
	\[
	\sh(uv)=(u,v) \quad \text{ and }\quad uv\in \Phi_\mathbb{L}(U).
	\]
	Consequently, if $U$ is finite then $\#(\Phi_{\mathbb{L}}(U) \backslash X) \geq \#(U) - 1$.
\end{lemma}

\begin{proof}
	The equality $\sh(uv)=(u,v)$ is a direct consequence of Proposition \ref{fact-Lyndon} (L3) because obviously  $u_R \leq_{\lex} v$ when $|u|\geq 2$.  Assume $uv\not\in \Phi_\mathbb{L}(U)$. Since $u>_{\lex}uv>_{\lex}v$, $uv\not\in U$. Thus $\Phi_\mathbb{L}(U)$ contains a word $w$ that is a proper Lyndon factor of $uv$.  By Lemma \ref{Lemma-Shirshov-factor-dichotomy}, $w$ is a prefix of $uv$  and $|w|> |u|$. If $|w_L|<|u|$ then $u>_{\lex}w_R >_{\lex} v$, a contradiction; if $|w_L|=|u|$ then  $u=w_L>_{\lex} w_R >_{\lex} v$, a contradiction; and if $|w_L|>|u|$ then $u>_{\lex} w_L  >_{\lex} v$, a contradiction too. The last statement is clear.
\end{proof}

Let $V$ be an antichain of words on $X$ and let $p\geq 2$ be an integer. A \emph{$p$-ambiguity on $V$}  is a word $w\in \langle X\rangle$ which has a decomposition 
$w=v_1 v_2\cdots v_{p+1}$ such that 
\begin{itemize}
	\item $v_1\in X \backslash V$ and $v_2,\ldots, v_{p+1}$ are nonempty words that have no factors in $V$;  
	\item for all $i$, $v_iv_{i+1}$ has a factor in $V$ but $v_it$ has no factor in $V$ for any proper prefix $t$ of $v_{i+1}$.
\end{itemize}
Note that $w$ has at most one decomposition as above. We shall write $\mathcal{A}_{0}(V)=\{1\}$, $\mathcal{A}_{1}(V)=X\backslash V$, $\mathcal{A}_{2}(V)= V\backslash X$,  and  write $\mathcal{A}_n(V)$ for the set of all $(n-1)$-ambiguities on $V$ for $n\geq 3$.

\begin{lemma}\label{behavior-ambiguity-Lyndon}
	Let $V$ be an antichain of Lyndon words on $X$ and let $U=\Psi_{\mathbb{L}}(V)$. Then 
	\begin{enumerate}
		\item $\mathcal{A}_n(V)=\emptyset$ for $n>\#(V\backslash X)+1$.
		\item $\mathcal{A}_n(V) \subseteq \{\, u_1u_2\cdots u_{n} ~ | ~ u_1,\ldots, u_{n}\in U, \, u_1>_{\lex} \cdots >_{\lex} u_{n}  \, \}$ for $n\geq 1$.
		\item $\mathcal{A}_n(V) \supseteq \left\{\, u_1 u_2\cdots u_{n}  \left | \begin{array}{l} u_1,\ldots, u_{n}\in U, \, u_1>_{\lex} \cdots >_{\lex} u_{n} \text{ but }\\
			\not\exists ~ w \in U  \text{ with }\, u_i>_{\lex}w>_{\lex}u_{i+1} \text{ for some } i
		\end{array} \right. \, \right\}$ for $n\geq 2$.
	\end{enumerate}
\end{lemma}

\begin{proof}
	(1) The case that $V\backslash X=\emptyset$  is clear.  So let us assume $V\backslash X\neq \emptyset$. For $n\geq 3$ with  $\mathcal{A}_n(V)\neq \emptyset$, choose an element 
	$w\in \mathcal{A}_n(V)$  with the required decomposition $w=v_1 v_2\cdots v_{n}$. Let $w_i=s_iv_{i+1}$ be the unique suffix of $v_iv_{i+1}$ that lies in $V$ for $i=1,\ldots, n-1$. Note that  $s_i$ is a nonempty suffix of $v_i$. So 
	\[
	w_1>_{\lex} s_2 >_{\lex} w_2 >_{\lex} s_3>_{\lex} \cdots >_{\lex} s_{n-1} >_{\lex} w_{n-1}
	\]
	by Proposition \ref{Proposition-character-Lyndon}. Consequently, $n\leq \#(V\backslash X)+1$ and then the result follows.
	
	(2) The cases that $n=1$ and $n=2$ are clear. Suppose $n\geq 3$ and let $w\in \mathcal{A}_n(V)$ be with the required decomposition $w = v_1v_2\cdots v_{n}$.   Let $t_{n+1} =1$, and then recursively define $u_i$ to be the longest Lyndon suffix of $v_it_{i+1}$ and define $t_i$ by $v_it_{i+1} = t_iu_i$ for $i=n, n-1, \cdots, 1$. By construction,  $u_n\in U$ and 
	\[
	w=v_1v_2\cdots v_n t_{n+1}= v_1\cdots v_{n-1}t_n u_n = \cdots =t_1u_1\cdots u_n.
	\]  We claim that 
	\[|u_{1}| > |t_{2}|, \ldots, |u_n|>|t_{n+1}| \quad \text{and} \quad u_{1}>_{\lex} \cdots >_{\lex}u_n.\]
	It follows that $|t_1|<|v_1|,\ldots, |t_n|<|v_n|$. So $u_1,\cdots, u_{n-1}\in U$, $t_1=1$ and hence $w= u_1u_2\cdots u_n.$
	
	Now we proceed to show the claim above.  For $i\in \{1,\ldots, n-1\}$, let $w_i=s_iv_{i+1}$ be the unique suffix of $v_iv_{i+1}$ that lies in $V$. Note that $s_i$ is a nonempty suffix of $v_i$. Obviously, one has $|u_n|>|t_{n+1}|$. Assume $l\in \{1,\ldots, n-1\}$ such that $|u_{l+1}| > |t_{l+2}|, \ldots, |u_n|>|t_{n+1}|$ and  $u_{l+1}>_{\lex} \cdots >_{\lex}u_n$. Since $u_{l+1} = a_{l+1}t_{l+2}$ for some nonempty suffix $a_{l+1}$ of $v_{l+1}$,  it follows that $w_lt_{l+2} = s_l v_{l+1}t_{l+2}$ is a  Lyndon word by Proposition \ref{fact-Lyndon} (L1). Let $z_l$ be the shortest Lyndon  suffix of $w_lt_{l+2}$ that are of length greater than $ |u_{l+1}|$. By the construction of $u_{l+1}$, one has $|z_l|> |v_{l+1} t_{l+2}|$, $(z_l)_R = u_{l+1}$ and  $(z_l)_L$ is a Lyndon suffix of $v_lt_{l+1}$ of length greater than $t_{l+1}$. By the construction of $u_l$, it contains $(z_l)_L$ as a suffix. Consequently,  $|u_l| \geq  |(z_l)_L| >|t_{l+1}|$ and $u_l\geq_{\lex} (z_l)_L >_{\lex} (z_l)_R= u_{l+1}$. The claim then follows by induction.
	
	(3) Let $u_1, \ldots,  u_{n} \in U$ with $u_1>_{\lex} \cdots >_{\lex} u_{n}$ but there's no $w\in U$ between $u_i$ and $u_{i+1}$ in pseudo-lexicographical order for each $i=1,\ldots, n-1$.  By Lemma \ref{lemma-bound-Psi-construction}, one has 
	$$u_iu_{i+1} \in \Phi_\mathbb{L}(U) = V, \quad i=1,\ldots, n-1.$$ The case that $n=2$ is then clear. For $n\geq 3$, we write $u_1=xs_1$ with $x$ a letter. To see that  $u_1u_2\cdots u_n = x (s_1u_2) u_3\cdots u_{n}$ is an $(n-1)$-ambiguity on $V$, it remains to show that $s_1u_2r$ has no factor in $V$ for any proper prefix $t$ of $u_3$. Otherwise,  there is a nonempty suffix $l$ of $s_1$ and a nonempty prefix $r$ of $u_3$ such that $w=lu_2r\in V$. By Lemma \ref{Lemma-Shirshov-factor-dichotomy}, either $w_L = lu_2r'$ or $w_R= l'u_2r$. Accordingly, one has $u_1>_{\lex } w_L >_{\lex } u_3$ or $u_1>_{\lex }w_R >_{\lex} u_3$, which is impossible because $u_2\not \in \{w_L, w_R\}$ but $u_2$ is the unique element of $\Psi_{\mathbb{L}}(V)$ that lie pseudo-lexicographic  strictly between $u_1$ and $u_3$. 
\end{proof}

\begin{proposition}\label{global-dimension-Lyndon-reduction-0}
	Let $I$ be a Lyndon ideal of $k\langle X\rangle$ and let $A= k\langle X\rangle/ I$.
	\begin{enumerate}
		\item The projective dimension of $A$ as an $A$-bimodule, the  left global dimension of $A$ and the right global dimension of $A$ are all less than or equal to $\min\{ \, \#(\N_I), ~  \#(\OO_I\backslash X) +1 \, \}$.
		\item If $\N_I$ is finite then $A$ is \emph{homologically smooth} in the sense that $A$ has a bounded projective resolution as an $A$-bimodule with each term finitely generated as an $A$-bimodule.
	\end{enumerate}
\end{proposition}

\begin{proof}

By the  Anick resolution of $A$ as an $A$-bimodule (see \cite[Theorem 4.1, Theorem 4.2]{CS}), there is a projective resolution of $A$ in the category of $A$-bimodules of the form:
\begin{equation*}\label{Anick-resolution}
	\big ( \cdots \to A\otimes k\mathcal{A}_3(V) \otimes A \xrightarrow{\delta_3} A\otimes k\mathcal{A}_2(V) \otimes A \xrightarrow{\delta_2} A\otimes k\mathcal{A}_1(V) \otimes A \xrightarrow{\delta_1}   A\otimes A \to 0 ~  \big ) ~ \xlongrightarrow{\mu} ~  A,  \tag{AR}
\end{equation*}
where $V:=\OO_I$ and $\mu:A\otimes A\to A$ is the multiplication map. By assumption, $V$ consists of Lyndon words and $\Psi_{\mathbb{L}} (V) =\N_I$. By  Lemma \ref{behavior-ambiguity-Lyndon} (1,2), the  projective dimension of $A$ as an $A$-bimodule  is less than or equal to $ \min\{ \, \#(\Psi_{\mathbb{L}}(V)), ~  \#(V \backslash X)+1 \, \}$. 
It is easy to check that the left and right global dimensions of $A$ are both less than or equal to the  the  projective dimension of $A$ as an $A$-bimodule. If $\N_I$ is finite, the Anick resolution tells us that $A$ is homologically smooth.	
\end{proof}

\begin{theorem}\label{global-dimension-Lyndon-reduction}
Let $I$ be a homogeneous Lyndon ideal of $k\langle X\rangle$ with $\N_I$ finite. Let $A:= k\langle X\rangle/ I$.
\begin{enumerate}
	\item $A$ is homologically smooth in the graded sense.
	\item $\uExt_{A\otimes A^o}^d (A, k) = \uExt_A^d({}_Ak,{}_Ak) = \uExt_A^d(k_A,k_A) = k(l)$, where $d=\#(\N_I)$ and $l=\sum_{u\in \N_I} \deg(u)$ and the notation  $(l)$ is the degree $l$-shifting on $\mathbb{Z}^s$-graded modules. 
	\item The projective dimension of $A$ as an $A$-bimodule, the  left  global dimension of $A$ and the right global dimension of $A$ all equal to $d=\#(\N_I)$.
\end{enumerate}
\end{theorem}

\begin{proof}
Let $\N_I=\{\, u_1<_{\lex} \cdots <_{\lex} u_d \, \}$ and $V:=\OO_I$. By construction, all differentials $\delta_n$ in the Anick resolution (\ref{Anick-resolution}) may be chosen to  preserve degrees (see \cite[Section 5]{CS}).  So $A$ is homologically smooth in the graded sense. This proves Part (1).  By Lemma \ref{behavior-ambiguity-Lyndon} (2,3), one has 
	$\mathcal{A}_n(V) = \emptyset$ for $n\geq d+1$, $\mathcal{A}_{d} (V) = \{\, u_1u_2\cdots u_d\}$ and $\mathcal{A}_{d-1}(V) \subseteq \{\, u_1\cdots \widehat{u_i}\cdots u_d ~|~ i=1,\ldots, d \,\}$. It follows that 
	$$
	{\rm im}~ \delta_d \subseteq A_{+} \otimes  k\mathcal{A}_{d-1}(V) \otimes A + A \otimes k\mathcal{A}_{d-1}(V) \otimes A_{+}.
	$$ 
	So  $\uHom_{A\otimes A^o}(\delta_d, k)=0$ and hence 
	$$
	\uExt_{A\otimes A^o}^d(A, k) = \uHom_{A\otimes A^o}(A\otimes k(-l)\otimes A, k) =  k(l).
	$$
	Apply the functor $-\otimes_A{}_Ak$ (resp. $k_A\otimes_A-$) on the two-sided Anick resolution  (\ref{Anick-resolution}), one obtains a projective resolution of ${}_Ak$ (resp. $k_A$) in the category of  $\mathbb{Z}^s$-graded left (resp. right) $A$-modules. Using these resolutions, a similar argument as that for the two-sided case yields that 
	$$\uExt_A^d({}_Ak,{}_Ak) = \uExt_A^d(k_A,k_A)=k(l).$$ This proves Part (2). Part (3) is a direct consequence of Proposition \ref{global-dimension-Lyndon-reduction-0} (1) and Part (2).
\end{proof}

Finally, we relate  Lyndon ideals to Artin-Schelter regular algebras, which  play an important role in the realm of noncommutative algebraic geometry. Recall that
a connected $\Gamma$-graded algebra $A$ is called \textit{Artin-Schelter regular} (AS-regular, for short) of dimension $d$ if 
\begin{itemize}
	\item $A$ has finite global dimension $d$;
	\item $A$ has finite Gelfand-Kirillov dimension;
	\item $A$ is Gorenstein; that is, 
	$\uExt^i_A(k_A,A)=\left\{\begin{array}{ll}0,&i\neq d,\\k(l),& i=d,\end{array}\right.$ for some $l\in\Gamma$.
	\end{itemize}
The  index $l$ is called the \textit{Gorenstein parameter} of $A$.

\begin{proposition}
	Let $k\langle x_1,\ldots, x_n\rangle$ be $\mathbb{N}$-graded by $\deg(x_i) =1$. Let $I$ be a homogeneous Lyndon ideal of $k\langle x_1,\ldots, x_n\rangle$ contains no linear polynomial and let $A:=k\langle x_1,\ldots, x_n\rangle/I$. 
	Assume that $A$ is AS-regular of global dimension $d$ and Gorenstein parameter $l$. Then 
	\[
	n\leq d \leq l \leq F_{d-n+3} +n-3,
	\]
	where $F_{r}$ denotes the $r$-th Fibonacci number for $r\geq 0$.
\end{proposition}

\begin{proof}
	By the assumption, Proposition \ref{Lyndon-GKdim}  and Theorem \ref{global-dimension-Lyndon-reduction} (2), 
	$\N_I  \supseteq \{x_1,\ldots, x_n\}$ and $\#(\N_I) =d$. Enumerating elements in $\N_I$ in the graded lex order as  
	$$
	x_1 <_{\glex} \cdots <_{\glex} x_{n-2} <_{\glex} x_{n-1}=u_0 <_{\glex} x_n=u_1 <_{\glex} u_2<_{\glex} \cdots <_{\glex} u_{d-n+1}. 
	$$
	Then $\deg(u_i)\leq F_i$ for $i=0,\ldots,d-n+1$ as showed in \cite[Proposition 7.3]{GF}. It follows that
	$$n\leq d\leq  \sum_{u\in \N_I} \deg(u) \leq  n-2 + \sum_{i=0}^{d-n+1} F_i = F_{d-n+3}+n-3.$$ 
	The equality is by the formula $\sum_{i=0}^r F_i = F_{r+2} -1$, which can be simply proved by induction on $r$.
	By Theorem \ref{global-dimension-Lyndon-reduction} (2), $l =\sum_{u\in \N_I} \deg(u) $, the result follows.
\end{proof}

A  $\Gamma$-graded algebra $A$ is called \emph{properly $\Gamma$-graded} if  the support $\{~\gamma\in \Gamma ~|~A_\gamma\neq 0~\}$ spans $\Gamma$ as an abelian monoid. The following result is a restatement of \cite[Theorem 8.1]{ZL}.

\begin{theorem} \label{AS-regular-Lydnon}
	Let $A$ be an AS-regular connected properly $\mathbb{N}^2$-graded algebra with  two generators. Assume that $A$ is a domain of global dimension $\leq 4$ or a domain of global dimension $5$ and GK dimension $\geq 4$. Then there is a homogeneous Lyndon  ideal $I$ of $k\langle x_1,x_2\rangle$ such that $A\cong k\langle x_1,x_2\rangle/I$ as $\mathbb{N}^2$-graded algebras. Here, $k\langle x_1,x_2\rangle $ is $\mathbb{N}^2$-graded by $\deg(x_1)=(1,0)$ and $\deg(x_2)=(0,1)$. \hfill $\Box$
\end{theorem}
	
\section*{Acknowledgments}

The author would like to thank Professor James J. Zhang for his suggestion to study coideal subalgebras of connected Hopf algebras by the theory of Lyndon words during the ``Noncommutative Algebraic Geometry Shanghai Workshop'' at Shanghai in November 2019. Thanks also to Professor D.-M. Lu for reading the manuscript, pointing out some typos and providing some suggestions. This work is supported by Shanghai Key Laboratory of Pure Mathematics and Mathematical Practice (Grant Nos. 22DZ2229014), the NSFC (Grant Nos. 12371039 \& 11971141), the Fundamental Research Funds for the Provincial Universities of Zhejiang, and the K.C. Wong Magna Fund in Ningbo University. Finally, the author thanks the referee for his/her careful reading and valuable suggestions.


\begin{thebibliography}{Mathzhou}
		\bibitem{AnFu} F. W. Anderson, K. R. Fuller, \emph{Rings and categories of modules}, second edition, Graduate Texts in Mathematics, Vol. \textbf{13}, Springer-Verlag New York, (1992).
		
		\bibitem{AnSc1} N. Andruskiewitsch, H.-J. Schneider, \emph{Lifting of quantum linear spaces and
		pointed Hopf algebras of order $p^3$}, J. Alg. \textbf{209} (1998), 658–691.
		
		\bibitem{ATV2} M. Artin, J. Tate, M. Van den Bergh, \emph{Modules over regular algebras of dimension $3$}, Invent. Math. \textbf{106} (1991), 335-388.
		
		\bibitem{BeGr} A. Berenstein, J. Greenstein, \emph{Primitively generated Hall algebras}, Pacific J. Math. \textbf{281} (2016), 287-331.
		
		\bibitem{BG} K. A. Brown, P. Gilmartin, \emph{Quantum homogeneous spaces of connected Hopf algebras}, J. Alg. \textbf{454} (2016), 400-432.
		
		\bibitem{BGZ} K. A. Brown, P. Gilmartin, J. J. Zhang, \emph{Connected (graded) Hopf algebras}, Trans. Amer. Math. Soc. \textbf{372} (2019), 3283-3317.
		
		\bibitem{BOZZ} K. A. Brown, S. O'Hagan, J. J. Zhang, G. Zhuang, \emph{Connected Hopf algebras and iterated Ore extensions}, J. Pure Appl. Alg.  \textbf{219} (2015), 2405-2433.
		
	    \bibitem{BZ} K. A. Brown, J. J. Zhang, \emph{Dualizing complexes and twisted Hochschild (co)homology for noetherian Hopf algebras}, J. Alg. \textbf{320} (2008), 1814-1850.
	    
	    \bibitem{BZ2} K. A. Brown, J. J. Zhang,  \emph{Iterated Hopf Ore extensions in positive characteristic}, J. Noncommut. Geom. \textbf{16} (2022), 787–837.
		
		\bibitem{CS} S. Chouhy, A. Solotar, \emph{Projective resolutions of associative algebras and ambiguities}, J. Alg.  \textbf{432} (2015), 22-61.
		
		\bibitem{DeGa} M. Demazure, P. Gabriel, \emph{Groupes Alg\'{e}briques I}, North Holland, Ansterdam, (1970).
		
		\bibitem{ESW19} K. Erdmann, \O. Solberg, X. Wang, \emph{On the structure and cohomology ring of connected Hopf algebras}, J. Alg. \textbf{527} (2019), 366–398.
		
		\bibitem{FMP} V. O. Ferreira, L. S. I. Murakami,  A. Paques, \emph{A Hopf-Galois correspondence for free algebras}, J. Alg. \textbf{276} (2004), 407–416.
		
		\bibitem{Ga0} T. Gateva-Ivanova, \emph{Global dimension of associative algebras}, in: proc. AAECC-6, LNCS, \textbf{357} (1989), 213-229.
		
		\bibitem{Ga}T. Gateva-Ivanova, \emph{Algebras defined by Lyndon words and Artin-Schelter regularity}, Trans. Amer. Math. Soc. Ser. B \textbf{9} (2022), 648-699.
		
		\bibitem{GF}T. Gateva-Ivanova, G. Fl{\o}ystad, \emph{Monomial algebras defined by Lyndon words}, J. Alg. \textbf{403} (2014), 470-496.
		
		\bibitem{HS0} I. Heckenberger, H.-J. Schneider, \emph{Right coideal subalgebras of Nichols algebras and the Duflo order on the Weyl groupoid}, Israel J. Math. \textbf{197} (2013), 139-187.
		
		\bibitem{HS} I. Heckenberger, H.-J. Schneider, \emph{Hopf algebras and root systems}, Mathematical surveys and Monographs, Vol. \textbf{247}, American Mathematical Society, Rhode Island, (2020).
		
		\bibitem{Kh} V. K. Kharchenko, \emph{A quantum analogue of Poincar\'{e}-Birkhoff-Witt theorem}, Alg. Log. \textbf{38} (1999), 259-276.
		
		\bibitem{Kh1} V. K. Kharchenko, \emph{PBW-bases of coideal subalgebras and a freeness theorem}, Trans. Amer. Math. Soc. \textbf{360} (2008), 5121-5143.
		
		\bibitem{Kh2} V. K. Kharchenko, \emph{Right coideal subalgebras in $U_q^+(\mathfrak{s}\mathfrak{o}_{2n+1})$}, J. Eur. Math. Soc. \textbf{13} (2011), 1677-1735.
		
		\bibitem{KhSa} V. K. Kharchenko, A. V. L. Sagahon,  \emph{Right coideal subalgebras in $U_q(\mathfrak{s}\mathfrak{l}_{n+1})$}, J. Alg. \textbf{319} (2008), 2571-2625.
		
		\bibitem{Krae} U. Kr\"{a}hmer, \emph{On the Hochschild (co)homology of quantum homogeneous spaces}, Israel J. Math. \textbf{189} (2012), 237–266.
		
		\bibitem{KL} G. R. Krause, T. H. Lenagan, \emph{Growth of algebras and Gelfand-Kirillov dimension}, revised edition, Graduate Studies in Mathematics, Vol. \textbf{22}, American Mathematical Society, Rhode Island, (1999).
       
        \bibitem{Lec} B. Leclerc, \emph{Dual canonical bases, quantum shuffles and $q$-characters}, Math. Z. \textbf{246} (2004), 691–732.
       
        \bibitem{Letz0} G. Letzter, \emph{Symmetric pairs for quantized enveloping algebras}, J. Alg. \textbf{220} (1999), 729-767.
        
        \bibitem{Letz} G. Letzter, \emph{Coideal subalgebras and quantum symmetric pairs}, in: S. Montgomery, H.-J. Schneider (Eds.) New directions in Hopf algebras, MSRI Publications, \textbf{43} (2002), 117-165.
		
		\bibitem{LZ} C.-C. Li, G.-S. Zhou, \emph{The structure of connected (graded) Hopf algebras revisited}, J. Alg. \textbf{610} (2022), 684-702.
         
        \bibitem{LiWu} L.-Y. Liu, Q.-S. Wu, \emph{Rigid dualizing complexes of quantum homogeneous spaces}, J. Alg. \textbf{353} (2012), 121-141.
		
        \bibitem{Lo} M. Lothaire, \emph{Combinatorics on words}, Encyclopedia in mathematics and its
		applications, \textbf{17} GC Rota, Editor Addison-Wesley Publishing Company, (1983).
		
		\bibitem{Lusz} G. Lusztig, \emph{Finite-dimensional Hopf algebras arising from quantized enveloping algebras}, J.
		Amer. Math. Soc. \textbf{3} (1990), 257–296.
		
		\bibitem{Masu1} A. Masuoka, \emph{On Hopf algebras with cocommutative coradical}, J. Alg. \textbf{144} (1991), 451-466.
         
        \bibitem{MM} J. W. Milnor, J. C. Moore, \emph{On the structure of Hopf algebras}, Ann. Math. \textbf{81} (1965), 211-264.
		
		\bibitem{MuSc} E. M\"{u}ller, H.-J. Schneider, \emph{Quantum homogeneous spaces with faithfully flat module structures}, Israel J. Math. \textbf{111} (1999), 157–190.
		  
		\bibitem{NWW19} Van C. Nguyen, L. Wang, X. Wang, \emph{Primitive deformations of quantum p-groups}, Algebr.
		  Represent. Theory \textbf{22} (2019), 837-865.

        \bibitem{Nord} P. Nordbeck, \emph{Cononical bases for algebraic computations}, Phd Thesis, Lund University, (2001).		
          
        \bibitem{Pogo} B. Pogorelsky, \emph{Right coideal subalgebras of the quantum Borel algebra of type $G_2$}, J. Alg. \text{322} (2009), 2335–2354.
          
        \bibitem{Rad85} D. E. Radford, \emph{The structure of Hopf algebras with a projection}, J. Alg. \textbf{92}  (1985), 322–347.
	
        \bibitem{RRZ} M. Reyes, D. Rogalski, J. J. Zhang, {\em Skew Calabi-Yau algebras and homological identities}, Adv. Math. \textbf{264} (2014), 308-354.
          
        \bibitem{Ring} C. M. Ringel, \emph{PBW-bases of quantum groups}, J. Reine Angew. Math. \textbf{470} (1996), 51–88.

		\bibitem{Skry} S. Skryabin, \emph{Projectivity and freeness over comodule algebras}, Trans. Amer. Math. Soc. \textbf{359} (2007), 2597–2623.
		
        \bibitem{Take} M. Takeuchi, \emph{Survey of braided Hopf algebras}, New trends in Hopf algebra theory (La Falda, 1999), Contemp. Math., vol. \textbf{267}, Amer. Math. Soc., Providence, RI, (2000), 301–323.
        
        \bibitem{Take1} M. Takeuchi, \emph{Relative Hopf modules - equivalences and freeness criteria}, J. Alg. \textbf{60} (1979), 452-471.
		
        \bibitem{Uf} S. Ufer, \emph{PBW bases for a class of braided Hopf algebras}, J. Alg. \textbf{ 280} (2004), 84-119.
		
        \bibitem{WZZ} D.-G. Wang, J. J. Zhang, G. Zhuang, \emph{Connected Hopf algebras of Gelfand-Kirillov dimension four}, Trans. Amer. Math. Soc. \textbf{367} (2015), 5597-5632.
        
        \bibitem{Wang} W. Wang, \emph{Quantum symmetric pairs}, to appear in Proceedings of ICM 2022, arXiv: 2112.10911.
        
        \bibitem{Yana} T. Yanai, \emph{Galois correspondence theorem for Hopf algebra actions}, in: Algebraic structures
        and their representations, 393–411, Contemporary Mathematics, vol. \textbf{376}, AMS, Providence, RI, (2005).

        \bibitem{Zhou} G.-S. Zhou, \emph{An application of Lyndon words in associative algebras} (Chinese), Appl. Math. J. Chinese Univ. Ser. A \textbf{30} (2015), 245-252.

        \bibitem{ZL} G.-S. Zhou, D.-M. Lu, \emph{Artin-Schelter regular algebras of dimension five with two generators}, J. Pure Appl. Alg. \textbf{218} (2014), 937-961.

        \bibitem{ZSL3} G.-S. Zhou, Y. Shen, D.-M. Lu,  \emph{The structure of connected (graded) Hopf algebras},  Adv. Math. \textbf{372} (2020), 107292.
		
        \bibitem{Zh} G. Zhuang, \emph{Properties of pointed and connected Hopf algebras of finite Gelfand-Kirillov dimension}, J. Lond. Math. Soc. \textbf{87} (2013), 877-898.
\end{thebibliography}
\end{document}